\documentclass[preprint]{imsart}
\RequirePackage[OT1]{fontenc}
\RequirePackage{amsthm,amsmath}
\RequirePackage[numbers]{natbib}
\RequirePackage[colorlinks,citecolor=blue,urlcolor=blue]{hyperref}
\usepackage{pkgfile}

\newcommand{\pr}[2]{\langle {#1} , {#2} \rangle}
\newcommand{\norm}[1]{\left \| #1 \right \|}

\newtheorem*{claim}{Claim}

\def \etc {,\ldots,}

\def \a {\alpha}

\def \t {\tau}

\def \dist {{\rm dist}}
\def \Span {{\rm span}}
\def \rank {{\rm rank }}

\newcommand{\sqrtneg}{\mathrm{i}}
\renewcommand{\Im}{{\rm{Im}}\,}
\renewcommand{\Re}{{\rm{Re}}\, }

\def \d {\delta}
\def\Adj{\mathrm{Adj}}

\makeatletter
\newcommand{\listintertext}{\@ifstar\listintertext@\listintertext@@}
\newcommand{\listintertext@}[1]{
  \hspace*{-\@totalleftmargin}#1}
\newcommand{\listintertext@@}[1]{
  \hspace{-\leftmargin}#1}
\makeatother

\begin{document}

\begin{frontmatter}
\title{The circular law for sparse non-Hermitian matrices}
\runtitle{The circular law for Sparse Matrices}

\begin{aug}
\author{\fnms{Anirban} \snm{Basak}\thanksref{t1}\ead[label=e1]{anirban.basak@weizmann.ac.il}}
\and
\author{\fnms{Mark} \snm{Rudelson}\thanksref{t2}\ead[label=e2]{rudelson@umich.edu}}

\thankstext{t1}{Research partially supported by grant 147/15 from the Israel Science Foundation.}
\thankstext{t2}{Research partially supported by NSF grant DMS 1464514.}
\runauthor{A.~Basak and M.~Rudelson}

\affiliation{Weizmann Institute of Science\thanksmark{t1} and University of Michiagn\thanksmark{t2}}

\address{Department of Mathematics\\ Weizmann Institute of Science\\ POB 26, Rehovot 76100, Israel\\
\printead{e1}\\
\phantom{E-mail:\ }
}

\address{Department of Mathematics\\ University of Michigan\\ East Hall, 530 Church Street\\ Ann Arbor, Michigan 48109, USA\\
\printead{e2}}
\end{aug}











\begin{abstract}
For a class of sparse random matrices of the form $A_n =(\xi_{i,j}\delta_{i,j})_{i,j=1}^n$, where $\{\xi_{i,j}\}$ are i.i.d.~centered sub-Gaussian random variables of unit variance, and $\{\delta_{i,j}\}$ are i.i.d.~Bernoulli random variables taking value $1$ with probability $p_n$, we prove that the empirical spectral distribution of $A_n/\sqrt{np_n}$ converges weakly to the circular law, in probability, for all $p_n$ such that $p_n=\omega({\log^2n}/{n})$. Additionally if $p_n$ satisfies the inequality $np_n > \exp(c\sqrt{\log n})$ for some constant $c$, then the above convergence is shown to hold almost surely. The key to this is a new bound on the smallest singular value of complex shifts of real valued sparse random matrices.
The circular law limit also extends to the adjacency matrix of a directed Erd\H{o}s-R\'{e}nyi graph with edge connectivity probability $p_n$.
\end{abstract}



\end{frontmatter}

\section{Introduction}\label{sec:intro}
For a $n \times n$ matrix $B_n$  let $\lambda_1(B_n), \lambda_2(B_n),\ldots,\lambda_n(B_n)$ be its eigenvalues. The empirical spectral distribution (\abbr{ESD}) of $B_n$, denoted hereafter by $L_{B_n}$, is given by
\[
L_{B_n}:=\f{1}{n}\sum_{i=1}^n \delta_{\lambda_i},
\]
where $\delta_x$ is the dirac measure at $x$. If $B_n$ is a non-normal matrix (i.e.~$B_n B_n^* \ne B_n^* B_n$) then its eigenvalues are complex valued, resulting in $L_{B_n}$ being supported on the complex plane. Furthermore, when $B_n$ is random its \abbr{ESD} is a random probability measure. Thus, to study the asymptotics of \abbr{ESD}s of random matrices one needs to define appropriate notions of convergence. This can be done in one of the two following ways: If $\{\mu_n\}$ is a sequence of random probability measures such that for every $f\in C_b(\C)$, i.e.~$f:\C \mapsto \R$ is bounded,
\(
\int_\C f d\mu_n \to \int_\C f d\mu
\)
in probability, for some probability measure $\mu$ (possibly random), then $\mu_n$ is said to converge weakly to $\mu$, in probability. If $\int_\C f d\mu_n \to \int_\C f d \mu$ almost surely, then $\mu_n$ is said to converge to $\mu$ weakly, almost surely.

The study of the \abbr{ESD} for random matrices can be traced
back to Wigner \cite{Wigner55, wigner1958distribution} who showed that the \abbr{ESD}'s of
$n \times n$ Hermitian matrices
with i.i.d.~centered entries of variance $1/n$ (modulo symmetry) satisfying appropriate moment bounds (for example, Gaussian) converge to the {\em semicircle distribution}. The conditions on the
finiteness of the moments were relaxed in subsequent works, see
\cite{bai2010spectral,pastur1972spectrum} and the references therein.

The rigorous study of non-Hermitian matrices, in particular non-normal matrices, emerged much later. The main difficulties were the  sensitivity of the eigenvalues of non-normal matrices under small perturbations and the lack of appropriate tools. For example, Wigner's proof employed the method of moments. Noting that the moments
of the semicircle law determine it, one computes by combinatorial
means
the expectation and the variance of the normalized trace of powers of the matrix to find the asymptotics of the moments of the \abbr{ESD}s. The analogue of this for non-normal matrices is to compute the mixed moments, i.e.~compute
\beq\label{eq:mm-non-normal}
\int_\C w^k \bar{w}^\ell dL_{B_n}(w) = \f{1}{n}\sum_{i=1}^n \lambda_i(B_n)^k \bar{\lambda}_i(B_n)^\ell.
\eeq
For $B_n$ non-normal, the \abbr{RHS} of \eqref{eq:mm-non-normal} cannot be expressed as powers of traces of $B_n$ and $B_n^*$. So the method of moment approach does not work. Another technique that works well for Hermitian matrices is the method of evaluating limiting {\em Stieltjes transform} (see \cite{bai-survey}). Since the Stieltjes transform of a probability measure is well defined outside its support, and the \abbr{ESD}s of non-normal matrices are supported on $\C$, their Stieltjes transforms fail to capture the spectral measure.

In the $1950$'s, based on numerical evidences, it was conjectured that the \abbr{ESD} of $B_n/\sqrt{n}$, where $B_n$ is an $n \times n$ matrix with i.i.d.~entries of zero mean and unit variance, converges to the circular law, the uniform distribution on the unit disk in the complex plane. In random matrix literature this conjecture is commonly known as the {\em circular law conjecture}.

Using the formula for the joint density function of the eigenvalues  Ginibre \cite{ginibre} solved the case when the entries have a complex Gaussian distribution. The case of real Gaussian entries, where a similar formula is available, was settled by
Edelman \cite{edelman1988eigenvalues}. For the general case when there is no such formula,
the problem remained open for a very long time. An approach to the problem, which eventually
played an important role in the resolution of the conjecture,
was suggested by Girko in the 1980's \cite{girko1984circular}. However mathematically it contained significant gaps. The first non-Gaussian case (assuming the existence of density for the entries) was rigorously treated by Bai \cite{bai}, and the first result without the density assumption was obtained by G\"otze and Tikhomirov \cite{gotze_tikhomirov}. After a series of partial results
(see \cite{bordenave_chafai} and the references therein), the circular law conjecture was established in its full generality in the seminal work of Tao and Vu \cite{tao_vu}:
\begin{thm}[Circular law for i.i.d.~entries {\cite[Theorem 1.10]{tao_vu}}]
 Let $M_n$ be an $n \times n$ random matrix whose entries are i.i.d.~complex random variables with zero mean and unit variance. Then the \abbr{ESD} of $\f{1}{\sqrt{n}}M_n$ converges weakly to the uniform distribution on the unit disk on the complex plane,
both in probability and in the almost sure sense.
\label{thm:circ_iid}
 \end{thm}

A remarkable feature of Theorem \ref{thm:circ_iid} is its {\em universality}. The asymptotic behavior of the \abbr{ESD} of $M_n/\sqrt{n}$ is insensitive to the specific details of the entry distributions as long as they are i.i.d.~and have zero mean and unit variance. Since the work of Tao and Vu, there have been numerous attempts to extend the universality of Theorem \ref{thm:circ_iid} for a wider class of random matrix ensembles. A natural extension would be to prove Theorem \ref{thm:circ_iid} for matrix ensembles with dependent entries. This has been shown in \cite{adamczak2011marchenko, adamczak2016circular, bcc, nguyen, nguyen2013circular}.

Another direction to pursue is to study the asymptotics of the \abbr{ESD} of sparse matrices. Sparse matrices  are  abundant in statistics, neural network, financial modeling, electrical engineering, wireless communications,  neuroscience, and  in  many  other  fields.   We  refer  the  reader  to  \cite[Chapter  7]{bai2010spectral}  for  other  examples, and their relevant references.
One model for sparse random matrices is the adjacency matrices of random $d$-regular directed graphs with $d=o(n)$ (for $\{a_n\}$ and $\{b_n\}$, sequences of positive reals,  the notation $a_n=o(b_n)$ means $\lim_{n \to \infty} a_n/b_n=0$). Recently in \cite{bcz, cook-rrd, LLTTY} the circular law conjecture was established for two different models of random $d$-regular directed graphs.

One of the most natural models for sparse random matrices is the Hadamard product of a matrix of i.i.d.~entries with zero mean and unit variance, and a matrix of i.i.d.~$\dBer(p_n)$ entries, with $p_n=o(1)$. In this paper we focus on the limiting spectral distribution of this class of sparse matrices. When $p_n=n^{\alpha-1}$ for some $\alpha \in (0,1)$, it has been shown that, under the assumption of the existence of $(2+\delta)$-th moment of the entries, the \abbr{ESD} of these sparse matrices (properly scaled) converges weakly to the circular law, in probability and almost surely (see \cite{gotze_tikhomirov, tao2008random}). Later in \cite{wood} the assumption on the existence of $(2+\delta)$-th moment was removed but the convergence was shown to hold in probability.




 In this paper, we prove that the circular law limit continues hold when $n p_n$ grows at a sub-polynomial rate (in $n$). Namely, we show that under certain moment assumptions of the entries the circular law limit holds for sparse non-Hermitian matrices whenever $np_n$ grows at a poly-logarithmic rate. Under an additional assumption on $p_n$ (see \eqref{p:assumption-as} below), the convergence is shown to hold almost surely.


Before stating our result, let us recall the well-known definition of sub-Gaussian random variables.

\begin{dfn}
For a random variable $\xi$, the sub-Gaussian norm of $\xi$, denoted by $\|\xi\|_{\psi_2}$, is defined as
\[
\|\xi\|_{\psi_2}:=\sup_{k \ge 1} k^{-1/2} \|\xi\|_k, \quad \text{ where } \|\xi\|_k:=(\E|\xi|^k)^{1/k}.
\]
If the sub-Gaussian norm is finite, the random variable $\xi$ is called sub-Gaussian.
\end{dfn}

We use the following standard notation: for two sequences positive reals $\{a_n\}$ and $\{b_n\}$ we write $a_n=\omega(b_n)$ if $b_n=o(a_n)$ and $a_n=O(b_n)$  if $\limsup_{n \to \infty} a_n/b_n <\infty$ .

The following is the main result of this article.

\begin{thm}\label{thm:sparse_general}
 Let ${A}_n$ be an $n \times n$ matrix with i.i.d.~entries $a_{i,j}= \delta_{i,j} \xi_{i,j}$, where $\delta_{i,j}, \ i,j \in [n]$ are independent Bernoulli random variables taking value 1 with probability $p_n \in (0,1]$, and $\xi_{i,j}, \ i,j \in [n]$ are real-valued i.i.d.~centered sub-Gaussian with unit variance.

 \begin{enumerate}[(i)]

 \item If $p_n$ is such that $np_n = \omega(\log^2 n)$ the \abbr{ESD} of $A_n/\sqrt{np_n}$ converges weakly to the circular law, as $n \to \infty$, in probability.

 \item There exists a constant $c_{\ref{thm:sparse_general}}$, depending only on the sub-Gaussian norm of $\{\xi_{i,j}\}$, such that if $p_n$ satisfies the inequality
 \beq\label{p:assumption-as}
 np_n > \exp\left(c_{\ref{thm:sparse_general}} \sqrt{\log n}\right)
 \eeq
 then the conclusion of part (i) holds almost surely.
 \end{enumerate}
 \end{thm}

It will be seen in Section \ref{sec:proof_outline} that a key to the proof of Theorem \ref{thm:sparse_general} is a uniform bound on $s_{\min}(A_n -w\sqrt{np_n}I_n)$ for Lebesgue a.e.~$w\in\C$, where $s_{\min}(\cdot)$ denotes the smallest singular value. We initiated this work in \cite{BR} and showed that the desired bound holds when $w \in \R$.
The result of \cite{BR} relied on identifying the obstacles of arithmetic nature by methods of Fourier analysis, and using geometry to show that with high probability none of these obstacles realizes.
However, even if the matrix $A_n$ is real valued, the extension to $w \in \C \setminus \R$ makes the set of potential arithmetic obstacles so rich that it cannot be handled within the framework of the previous argument.
 This required developing  new methods providing both a more precise identification of the arithmetic obstacles arising from the complex structure and entropy bounds showing that with high probability these obstacles are avoided.

The main part of this paper is devoted to find the desired bound on $s_{\min}$ with a probability bound strong enough to apply Borel-Cantelli lemma in order to deduce the almost sure convergence of Theorem \ref{thm:sparse_general}(ii). To remove the condition \eqref{p:assumption-as} one needs an improvement of \cite[Proposition 3.1]{BR}. {See Remark \ref{rmk:p-ass-explain} for more details.}

It is easy to see that if $p_n=\f{\log n}{n}$, then the number of zero columns of $A_n$ is positive (and hence $s_{\min}(A_n)=0$) with probability bounded away from zero. So $\f{\log n}{ n}$ is a natural barrier in this set-up. To extend the bound on $s_{\min}$ beyond this barrier one needs to analyze the smallest singular value of the adjacency matrix of ``core of the graph'', when $A_n$ is viewed as the adjacency matrix of directed random weighted graph. We leave this effort to future ventures.

Another key ingredient for the proof of Theorem \ref{thm:sparse_general} is the bound on the smallish singular values of $(A_n -w\sqrt{np_n}I_n)$ (see Theorem \ref{thm:intermed-sing}). This is derived by relating the inverse second moment of the singular values to that of the distance of a random vector from a random subspace. This idea originated in \cite{tao_vu} and was crucial in relaxing the assumption on the existence of the $(2+\delta)$-th moment and prove Theorem \ref{thm:circ_iid} only under the second moment assumption.  To carry out a similar scheme in the sparse set-up one needs to consider random subspaces of dimension $n-m$ with $m=o(n/\log n)$. Concentration inequalities yield a probability bound $\exp(-c mp_n)$, for some $c>0$. To accommodate a union bound we then need $np_n=\omega(\log^2 n)$ which translates to the required lower bound  on $p_n$ in Theorem \ref{thm:sparse_general}(i).

\begin{rmk}[{Sub-Gaussianity assumption}]\label{rmk:sub-G-I}
The sub-Gaussianity assumption of Theorem \ref{thm:sparse_general} is used in Theorem \ref{thm: smallest singular} to show that $\|A_n\|=O(\sqrt{np_n})$, where $\| \cdot\|$ denotes the operator norm.
From \cite[Remark 1.9]{BR} we note that if $\{\xi_{i,j}\}$ are such that
\beq\label{eq:light-tail}
\E|\xi_{i,j}|^h \le C^h h^{\beta h}, \quad \text{ for all } h \ge 1, \text{ and for some constants } C \text{ and } \beta,
\eeq
then $\|A_n\|=O(\sqrt{np_n})$, for all $p_n$ satisfying $np_n=\Omega((\log n)^{2\beta})$ (for two sequences of positive reals $\{a_n\}$ and $\{b_n\}$ we write $a_n=\Omega(b_n)$ if $b_n=O(a_n)$). The case $\beta=1/2$ corresponds to the sub-Gaussian random variables. So we conclude that if $\{\xi_{i,j}\}$ satisfies the moment assumption \eqref{eq:light-tail}, for some $\beta \ge 1/2$, then the conclusion of Theorem \ref{thm:sparse_general}(i) holds for all $p_n$ satisfying $np_n \ge \omega(\log^2n), \Omega((\log n)^{2\beta})$. It is easy to check that $p_n$ satisfies \eqref{p:assumption-as} whenever $np_n = \omega((\log n)^{2\beta})$. Hence, the conclusion of Theorem \ref{thm:sparse_general}(ii) also holds under the moment assumption \eqref{eq:light-tail}. To retain the clarity of presentation we prove Theorem \ref{thm:sparse_general} for sub-Gaussian random variables.
\end{rmk}

\begin{rmk}[Circular law limit for shifted sparse matrices]
It is well known that the spectrum of normal matrices is stable under small perturbations (see \cite[Lemma 2.1.19]{agz} and \cite[Lemma 2.2]{bai-survey}). However, for a general non-normal matrix its spectrum is highly sensitive to small perturbations, for example see \cite[Section 2.8.1]{tao}. So there are no analogues of \cite[Lemma 2.1.19]{agz} and \cite[Lemma 2.2]{bai-survey} for an arbitrary non-normal matrix. Nevertheless, in \cite{wood} it was shown that if $D_n$ is any $n \times n$ matrix with $\rank(D_n)=o(n)$ and $\Tr(D_nD_n^*)=O(n^2p_n)$ then the \abbr{ESD} of $(A_n+D_n)/\sqrt{np_n}$ admit a circular law limit. Investigating our proof one can deduce that the \abbr{ESD} of $(A_n+D_n)/\sqrt{np_n}$ have a circular law limit for any sequence real diagonal matrices $\{D_n\}$ such that $\|D_n\|=O(\sqrt{np_n})$ and $\Tr(D_n^2) = o(n^2 p_n)$. It is possible to modify the proof of Theorem \ref{thm:sparse_general} to establish the circular law limit for general shifts. We do not pursue this direction here.
\end{rmk}

We next show that the circular law limit holds for the adjacency matrix of a directed Erd\H{o}s-R\'{e}nyi random graph which may be of interest in computer science and graph theory. Let us begin with the relevant definitions.

\begin{dfn}\label{dfn:ER}
Let $\mathsf{G}_n$ be a random directed graph on $n$ vertices, with vertex set $[n]$, such that for every $i \ne j$, a directed edge from $i$ to $j$ is present with probability $p$, independently of everything else. Assume that the graph $\mathsf{G}_n$ is simple, i.e.~no self-loops or multiple edges are present. We call this graph $\mathsf{G}_n$ a directed Erd\H{o}s-R\'{e}nyi graph with edge connectivity probability $p$. For any such graph $\mathsf{G}_n$ we denote $\Adj_n:=\Adj(\mathsf{G}_n)$ to be its adjacency matrix. That is, for any $i, j \in [n]$,
\[
\Adj_n(i,j)=\left\{\begin{array}{ll}
1 & \mbox{if a directed edge from $i$ to $j$ is present in $\mathsf{G}_n$}\\
0 & \mbox{otherwise}.\end{array} \right. \]
\end{dfn}

\begin{thm}\label{thm: sparse_er}
 Let $\Adj_n$ be the adjacency matrix of a directed Erd\H{o}s-R\'{e}nyi graph, with edge connectivity probability $p_n \in (0, 1)$. Denote $\bar{p}_n:=\min\{p_n,1-p_n\}$.

 \begin{enumerate}[(i)]
  \item If $p_n$ is such that $n\bar{p}_n = \omega(\log^2 n)$ the \abbr{ESD} of $\Adj_n/\sqrt{np_n(1-p_n)}$ converges weakly to the circular law, as $n \to \infty$, in probability.

 \item There exists an absolute constant $c_{\ref{thm: sparse_er}}$ such that if $p_n$ satisfies the inequality
 \beq\label{p:assumption-as-ber}
  n \bar{p}_n > \exp\left(c_{\ref{thm: sparse_er}} \sqrt{\log n}\right)
 \eeq
 then the conclusion of part (i) holds almost surely.
 \end{enumerate}

\end{thm}

The proof of Theorem \ref{thm: sparse_er} follows from a relatively standard modification of that of Theorem \ref{thm:sparse_general}. We refer the reader to the arXiv version of this paper \cite{br-cir}.

\begin{rmk}
Theorem \ref{thm:sparse_general} and Theorem \ref{thm: sparse_er} find the asymptotics of the eigenvalues of a large class of sparse non-Hermitian random matrices at a macroscopic level. An interesting question would be to prove the universality of the eigenvalue distribution at the microscopic level. This has been shown for a wide class of Hermitian random matrices (see \cite{knowles2016lecnotes} and references therein). For dense non-Hermitian random matrices,  it was shown in \cite{bourgade_yau_yin} that the  local circular law holds. In a forthcoming article \cite{BR-loc} we establish the same for sparse non-Hermitian random matrices.
\end{rmk}


\subsection*{Outline of the paper}
Section \ref{sec:proof_outline} provides a brief outline of the proof techniques of Theorem \ref{thm:sparse_general}. We begin Section \ref{sec:proof_outline} with a replacement principle (see Lemma \ref{lem:replacement}) which is a consequence of Girko's method. The replacement principle allows us to focus only on the integrability of $\log(\cdot)$ with respect to the \abbr{ESD} of $\wt{\bm B}_n^w:=[(B_n -wI_n)^*(B_n-wI_n)]^{1/2}$ for $w\in \C$, where $B_n$ is any random matrix. To implement this scheme one requires a good control on $s_{\min}(\wt{\bm B}_n^w)$ as well as on its {\em smallish} singular values. One also needs to establish weak convergence of the \abbr{ESD}s of $\wt{\bm B}_n^w$.

The required control on $s_{\min}$ and smallish singular values are derived in Theorem \ref{thm: smallest singular} and Theorem \ref{thm:intermed-sing}, and we outline of their proofs in Section \ref{sec:smallest_sing} and Section \ref{sec:inter-sing-val}, respectively. The limit of the \abbr{ESD}s of $\wt{\bm B}_n^w$ is derived in Theorem \ref{thm:weak-conv} with the outline of the proof appearing in Section \ref{subsec:weak-conv}.

Section \ref{sec:non-dom real part} - Section \ref{sec: singular value} are devoted to the proof of Theorem \ref{thm: smallest singular}. Since $s_{\min}(M_n)$ equals the infimum of $\|M_n u\|_2$ ($\|\cdot\|_2$ denotes the Euclidean norm) over all vectors $u$ of unit $\ell_2$ norm, we split the unit sphere into three parts: {\em compressible vectors}, {\em dominated vectors} and the complement of their union. The compressible vectors and dominated vectors are treated with results from \cite{BR}. The majority of the work is to control infimum over the vectors that are neither compressible nor dominated. Using a result of \cite{RV1} (see Lemma 3.5 there) this boils down to controlling the inner product of the first column of $(A_n-w\sqrt{np_n}I_n)$ and the vector normal to $H_n^w$, the subspace spanned by the last $(n-1)$ columns of $(A_n- w \sqrt{np_n}I_n)$. In Section \ref{sec: singular value}, it is shown that the last assertion can be proved using Berry-Ess\'{e}en Theorem. However, the probability bounds obtained from Berry-Ess\'{e}en Theorem is too weak to prove the almost sure convergence of Theorem \ref{thm:sparse_general}(ii).

In Section \ref{sec:non-dom real part} - Section \ref{sec: net real},  we derive a better probability bound that is suitable for the proof of Theorem \ref{thm:sparse_general}(ii). We split the set of vectors into two categories: {\em genuinely complex} and {\em essentially real}. Roughly speaking, the set of essentially real vectors are those for which the real and the imaginary parts are almost linearly dependent, and its complement is the set of genuinely complex vectors.

In Section \ref{sec:non-dom real part}, we show that the vector normal to $H_n^w$ has a non-dominated real part, with high probability. We construct a net of small cardinality for the genuinely complex vectors in Section \ref{sec: net complex}. We then use this net in Section \ref{sec: kernel complex} and results of Section \ref{sec:non-dom real part} to show that with high probability, the normal vector cannot be a genuinely complex vector with a sub-exponential (in $\sqrt{np_n}$) \abbr{LCD}. A similar result for essentially real vectors is obtained in Section \ref{sec: net real}. Then we finish the proof of Theorem \ref{thm: smallest singular} in Section \ref{sec: singular value}.

In Section \ref{sec: intermediate singular values} we prove Theorem \ref{thm:intermed-sing}. The key idea is to show that the distance of any row of $A_n$ from any given subspace of relatively small dimension cannot be too small with large probability. This observation together with \cite[Lemma A.4]{tao_vu} finishes the proof.

Section \ref{sec:weak-conv} is devoted to the proof of Theorem \ref{thm:weak-conv}, which establishes the weak convergence of the empirical measure of the singular values of $(A_n/\sqrt{np_n} -w I_n)$. The weak convergence is established by appealing to the Lindeberg replacement lemma, which was introduced by Chatterjee in \cite{cha05}, in conjunction with the standard concentration inequalities for the spectral measure of a Hermitian matrix and a truncation argument.

Finally in Section \ref{sec:proof-main-thm}, combining the results of Section \ref{sec: singular value} - Section \ref{sec:weak-conv}, we finish the proof of Theorem \ref{thm:sparse_general}. 

\section{Preliminaries and Proof Outline}\label{sec:proof_outline}
In this section we provide an outline of the proof of Theorem \ref{thm:sparse_general} and introduce necessary definitions and notation. As mentioned in Section \ref{sec:intro}, the standard technique to find the limiting spectral distribution of a non-normal matrix is the  Girko's method. We refer the reader to \cite{bcz} for a detailed description of it. The utility of Girko's method, in the context of our set-up, can be captured by the following {\em replacement principle}.  The replacement principle, which has its origin in \cite{tao_vu}, gives a sufficient criterion for the \abbr{ESD} of two random matrix ensembles to have the same limit. To state the relevant result, first we introduce a few definitions. A sequence of random variables $\{X_n\}$ is said to be bounded in probability if
\[
\lim_{K \ra \infty} \limsup_{n\ra \infty} \P (|X_n| \le K)=1
\]
and $\{X_n\}$ is said to be almost surely bounded if
\[
\P(\limsup_n |X_n| < \infty)=1.
\]
Next, for a matrix $B_n$, we denote $\norm{B_n}_2$  its Frobenius norm, i.e. $\norm{B_n}_2:= \sqrt{\Tr(B_n^* B_n)}$.

\begin{lem}[Replacement lemma]
\begin{enumerate}[(a)] \item Let $B_n^{(1)}$ and $B_n^{(2)}$ be two sequences of $n \times n $ random matrices, such that
\begin{enumerate}[(i)]
\item The expression
\beq\label{eq:frob-bd}
\f{1}{n}\norm{B_n^{(1)}}_2^2 + \f{1}{n}\norm{B_n^{(2)}}_2^2 \quad \text{ is bounded in probability},
\eeq
\listintertext*{and}

\item For Lebesgue almost all $w \in \D \subset B_\C(0,R) \subset \C$, for some domain $\D$ and some finite $R$,
\beq\label{eq:diff-conv}
\f{1}{n} \log |\det (B_n^{(1)}-w I_n)| - \f{1}{n} \log |\det (B_n^{(2)}-w I_n)|  \ra 0, \quad \text{ in probability}.
\eeq
\end{enumerate}
Then for every $f \in C_c^2(\C)$ supported on $\D$,
\beq\label{eq:int-diff-conv}
\int f(z)dL_{B_n^{(1)}}(w) - \int f(z)dL_{B_n^{(2)}}(w) \ra 0, \quad \text{ in probability}.
\eeq

\item If \eqref{eq:frob-bd} is almost surely bounded and \eqref{eq:diff-conv} holds almost surely then \eqref{eq:int-diff-conv} holds almost surely as well.
\end{enumerate}
\label{lem:replacement}
\end{lem}

The replacement principle of \cite{tao_vu} requires a uniform control on $s_{\min}(A_n - w \sqrt{np}I_n)$ for Lebesgue almost every $w\in \C$. Theorem \ref{thm: smallest singular} (see below) provides such a control only when $w$ is away from the real line. Therefore we need to use Lemma \ref{lem:replacement}, borrowed from \cite{bcz}, instead of \cite[Theorem 2.1]{tao_vu}.

Lemma \ref{lem:replacement}(a) was proved in \cite{bcz}. Repeating the proof of \cite[Lemma 10.1]{bcz} one can derive Lemma \ref{lem:replacement}(b). Details are left to the reader. 

We apply Lemma \ref{lem:replacement} with $B_n^{(1)}=\f{1}{\sqrt{np_n}}A_n$ and $B_n^{(2)}$ which is the matrix of i.i.d.~centered complex Gaussian entries with variance $1/n$. The assumption (i) is straightforward to verify: it follows from laws of large numbers.
It is well known that $\f{1}{n}\log|\det(B_n^{(2)}-wI_n)|$ admits a limit. Hence, establishing  assumption (ii) of Lemma \ref{lem:replacement} boils down to showing that $\log(\cdot)$ is integrable with respect to the empirical measure of the singular values of $B_n^{(1)}-w I_n$.  As $\log(\cdot)$ is unbounded near zero, one needs to establish the weak convergence of the empirical measure of the singular values, find bounds on $s_{\min}$, and show that there are not many singular values in an interval near zero (the unboundedness of $\log(\cdot)$ near infinity is not a problem since the maximal singular value of $B_n^{(1)}-wI_n$ is $O(1)$ with probability $1-o(1)$). These are the three ingredients of the proof of Theorem \ref{thm:sparse_general}.

\subsection{Smallest singular value} \label{sec:smallest_sing}The desired bound on $s_{\min}(A_n - \sqrt{n p_n} w I_n)$ is derived in the theorem below.

 \begin{thm}  \label{thm: smallest singular}
 Let ${A}_n$ be an $n \times n$ matrix with i.i.d.~entries $a_{i,j}= \delta_{i,j} \xi_{i,j}$, where $\{\delta_{i,j}\}$ are independent Bernoulli random variables taking value 1 with probability $p_n \in (0,1]$, and $\{\xi_{i,j}\}$ are i.i.d.~real-valued centered sub-Gaussian with unit variance. Fix $R \ge 1$, $r \in (0,1]$ and let $D_n$ be a diagonal matrix such that $ \norm{D_n} \le R \sqrt{np_n}$ and $ {\mathrm{Im}}(D_n)= r'\sqrt{np_n} I_n$ for some $r'$ with $|r'| \in [r,1]$. Then there exist constants $0< c_{\ref{thm: smallest singular}}, {\bar{c}_{\ref{thm: smallest singular}}}, c'_{\ref{thm: smallest singular}}, C_{\ref{thm: smallest singular}}, C'_{\ref{thm: smallest singular}}, \ol{C}_{\ref{thm: smallest singular}} < \infty$, depending only on $R, r$, and the sub-Gaussian norm of $\{\xi_{i,j}\}$, such that for any $\vep>0$ we have the following:

 \begin{enumerate}[(i)]

\item If
 \beq
p_n \ge  \frac{\ol{C}_{\ref{thm: smallest singular}} \log n}{n}  ,\notag
 \eeq
 then
 \[
  \P \left( s_{\min}({A}_n + D_n) \le c_{\ref{thm: smallest singular}} \vep \exp \left(-C_{\ref{thm: smallest singular}} \frac{\log (1/p_n)}{\log (np_n)} \right) \sqrt{\frac{p_n}{n}}  \right)
  \le \vep +  \f{C'_{\ref{thm: smallest singular}}}{\sqrt{np_n}}. 
 \]

 \item Additionally, if
 \beq\label{eq:p-n-condition}
 \log(1/p_n) < {\bar{c}_{\ref{thm: smallest singular}}} (\log np_n)^2,
 \eeq
 then
  \begin{multline}\label{eq:strong-s-min}
  \P \left( s_{\min}({A}_n + D_n) \le c_{\ref{thm: smallest singular}} \vep \exp \left(-C_{\ref{thm: smallest singular}} \frac{\log (1/p_n)}{\log (np_n)} \right) \sqrt{\frac{p_n}{n}}  \right)\\
  \le \vep +  \exp(-c'_{\ref{thm: smallest singular}}\sqrt{np_n}).
 \end{multline}
  \end{enumerate}
\end{thm}

\begin{rmk}
It is easy to check that if $np_n > \exp(c_\star \sqrt{\log n})$ for some constant $c_\star$ then $p_n$ satisfies \eqref{eq:p-n-condition}. Therefore, it is enough to prove Theorem \ref{thm: smallest singular}(ii) under the assumption \eqref{eq:p-n-condition} in order to apply it to the proof of Theorem \ref{thm:sparse_general}(ii).
\end{rmk}


\begin{rmk}[Optimality of Theorem \ref{thm: smallest singular}]
It is believed that for $A_n$ and $D_n$ as in Theorem \ref{thm: smallest singular} one has that $s_{\min}(A_n + D_n) \sim p_n^{1/2}n^{-1/2}$. Hence, for $p_n \sim n^{\alpha -1}$ for some $\alpha >0$, Theorem \ref{thm: smallest singular} gives an optimal lower bound on $s_{\min}( A_n+D_n)$. However, when $np_n$ grows at a rate sub-polynomial in $n$, we get an additional factor $n^{-o(1)}$. This is due to the fact that one needs $\rho=o(1)$ in \cite[Proposition 3.1]{BR}. To obtain the optimal lower bound on $s_{\min}(A_n+ D_n)$ one needs $\rho=\Omega(1)$ there.

We also add that the optimal probability bound for the event on the \abbr{LHS} of \eqref{eq:strong-s-min} is $\vep + \exp(-c'_{\ref{thm: smallest singular}}{np_n})$. The sub-optimality of the probability bound in \eqref{eq:strong-s-min} is again due to the fact that we can only allow $\rho=o(1)$ in \cite[Proposition 3.1]{BR}.
\end{rmk}

Similar to \cite{BR}, without loss of generality, we can and will assume that $p \le cR^{-2}$, for some small positive constant $c$. For  larger values of $p$ the entries $a_{i,j}$ have variance bounded below by an absolute constant.
In such case, we can ignore sparsity and regard entries $a_{i,j}$ as i.i.d. centered sub-Gaussian random variables whose variance is bounded below.

To prove Theorem \ref{thm: smallest singular} we follow the same scheme as in \cite{BR} and borrow some of its results.
Recalling the definition of the smallest singular value we have
\[
  s_{\min}({A}_n+D_n)=\inf_{z \in S_\C^{n-1}} \norm{({A}_n+D_n)z}_2,
 \]
 where $S_\C^{n-1}:= \{ z \in \C^n: \norm{z}_2=1\}$.
Thus, to bound $s_{\min}$ we need a lower bound on this infimum. To obtain such a bound we decompose the unit sphere into  {\em compressible}, {\em dominated}, and {\em incompressible} vectors, and obtain necessary bounds on the infimum on each of these parts separately. The definitions of compressible, dominated, and incompressible vectors are borrowed from \cite{BR}. However, we now need to treat complex shifts of the matrix ${A}_n$ which necessitates a straightforward modification of those definitions to accommodate vectors with complex valued entries.
We start with the definition of compressible and incompressible vectors.
\begin{dfn}\label{dfn:comp}
Fix $m<n$. The set of $m$-sparse vectors is given by
 \[
  \mathrm{Sparse}(m):=\{ z \in \C^{n} \mid |\mathrm{supp}(z)| \le m \},
 \]
 where $|S|$ denotes the cardinality of a set $S$ and $\mathrm{supp}(\cdot)$ denotes the support.
 Furthermore, for any $\delta>0$, the vectors which are $\delta$-close to $m$-sparse vectors in Euclidean norm, are called $(m, \delta)$-compressible vectors. The set of all such vectors, hereafter will be denoted by $\mathrm{Comp}(m, \delta)$. Thus,
 \[
   \mathrm{Comp}(m, \delta):= \{z \in S_\C^{n-1} \mid \exists y \in \mathrm{Sparse}(m) \text{ such that } \norm{z-y}_2 \le \delta \}.
  \]
  The vectors in $S_\C^{n-1}$ which are not compressible are defined to be incompressible, and the set of all incompressible vectors is denoted as $\mathrm{Incomp}(m,\delta)$.
  \end{dfn}

 Next we define the dominated vectors. These are close to sparse vectors but in a different sense.
 \begin{dfn}
 For any $z \in S_\C^{n-1}$, let $\pi_z: [n] \to [n]$ be a permutation which arranges the absolute values of the coordinates of $z$ in a non-increasing order. For $1 \le m \le m' \le n$, denote by $z_{[m:m']} \in \C^n$ the vector with coordinates
  \[
    z_{[m:m']}(j):=z(j) \cdot \mathbf{1}_{[m:m']}(\pi_z(j)).
  \]
  In other words, we include in $z_{[m:m']}$ the coordinates of $z$ which take places from $m$ to $m'$ in the non-increasing rearrangement of its absolute values. For $\alpha<1$ and $m \le n$ define the set of vectors with dominated tail as follows:
  \[
    \mathrm{Dom}(m, \alpha):= \{ z \in S_\C^{n-1} \mid\norm{z_{[m+1:n]}}_2 \le \alpha \sqrt{m} \norm{z_{[m+1:n]}}_{\infty} \}.
  \]
  \end{dfn}
 \noindent
 Note that by definition, $\text{Sparse}(m) \cap S_\C^{n-1} \subset \text{Dom}(m, \alpha)$, since for $m$-sparse vectors, $z_{[m+1:n]}=0$.

\vskip10pt

While studying the behavior of $s_{\min}$ of real shifts of ${A}_n$ in \cite{BR}, we noted that the control of the infimum over compressible and dominated vectors can be extended when they are viewed as subsets of $S_\C^{n-1}$ (cf.~\cite[Remark 3.10]{BR}). So we only need to control the infimum over vectors that are neither compressible nor dominated. The infimum over the incompressible vectors is tackled by associating it with the average  distance of a column of the matrix ${A}_n$ from the subspace spanned by the rest of the columns. We use the following result:
 \begin{lem}[Invertibility via distance {\cite[Lemma 3.5]{RV1}}]  \label{l: via distance}
Let $\tilde A_n$ be any $n \times n$ random matrix. For $j \in [n]$, let $\tilde{A}_{n,j} \in \C^n$ be the $j$-th column of $\tilde{A}_n$, and let $H_{n,j}$ be the subspace of $\C^n$ spanned by $\{\tilde{A}_{n,i}, i \in [n]\setminus\{j\}\}$. Then for any $\vep, \rho>0$, and $M<n$,
   \beq\label{eq:invertibility_distance}
    \P\left(  \inf_{z \in  \text{\rm Incomp}(M, \rho)} \norm{\tilde{A}_n z}_2 \le \vep \rho^2 \sqrt{\frac{p}{n}} \right)
    \le \frac{1}{M} \sum_{j=1}^n \P \left( \dist(\tilde{A}_{n,j},H_{n,j}) \le \rho\sqrt{p} \vep \right).
   \eeq
 \end{lem}

\vskip10pt
We should mention here that Lemma \ref{l: via distance} can be extended to the case when the event on the \abbr{LHS} of \eqref{eq:invertibility_distance} is intersected with an event $\Omega$, and in that case Lemma \ref{l: via distance} continues to hold if the \abbr{RHS} of \eqref{eq:invertibility_distance} is replaced by intersecting each of the event under the summation sign with the same event $\Omega$ (see also \cite[Remark 2.5]{BR}). We will actually use this generalized version of Lemma \ref{l: via distance}.


In order to apply Lemma \ref{l: via distance} in our set-up,  denote by $B^{D,n-1}$  the $(n-1) \times n$ matrix obtained by collecting the last $(n-1)$ rows of $({A}_n+D_n)^{\sf T}$. Hereafter, for brevity, we will often write $B^D$ instead of $B^{D,n-1}$. We note that any unit vector $z$ such that $B^{D} z=0$ is a vector normal to the subspace spanned by the last $(n-1)$ columns of $({A}_n+D_n)$. Thus, applying Lemma \ref{l: via distance} and the fact that the columns of $\bar{A}_n$ are i.i.d., we see that it is enough to find bounds on $\langle A_{n,1}^D, z \rangle$, such that $B^{D} z=0$, where $A_{n,1}^D$ is the first column of $({A}_n +D_n)$.

The small ball probability bounds on $\langle A_{n,1}^D, z \rangle$ depend on the additive structure of the vector $z$.
Following \cite{BR}, we  see that with high probability, we can assume that any $z \in \text{Ker}(B^D)$ is neither compressible nor dominated, where $\text{Ker}(B^D):=\{u \in \C^n: B^D u=0\}$. Therefore, it is enough to obtain estimates on the  small ball probability for incompressible and non-dominated vectors. To this end, we define the following notion of {\em L\'{e}vy concentration function}:
\begin{dfn}\label{dfn:levy}
Let $Z$ be a random variable in $\C^n$. For every $\vep >0$, the L\'{e}vy concentration function of $Z$ is defined as
\beq
\cL(Z,\vep):=\sup_{u \in \C^n}\P(\norm{ Z-u}_2 \le \vep).\notag
\eeq
\end{dfn}

\vskip10pt
The Berry-Ess\'{e}en bound of \cite[Theorem 2.2.17]{St} yields a weak control on L\'{e}vy concentration function which is enough to prove Theorem \ref{thm: smallest singular}(i). To prove Theorem \ref{thm: smallest singular}(ii) a significant amount of additional work is needed which is the key contribution of this paper.

To obtain a strong probability bound on the L\'{e}vy concentration function, the standard approach is to first quantify the additive structure present in an incompressible vector via the definition of {\em least common denominator} (\abbr{LCD}). When the \abbr{LCD} is large, one can derive a good bound on the L\'{e}vy concentration function using Ess\'{e}en's inequality \cite{esseen} (see also \cite[Theorem 6.3]{V}). However, Ess\'{e}en's inequality does not yield a strong small ball probability estimate for vectors with small values of \abbr{LCD}. Nevertheless, these vectors are shown to admit a special net of small cardinality and therefore one can still apply the union bound to complete the proof. For example, see \cite{BR, RV1, RV2}. One would hope to carry out the same program here. However, when we view the incompressible and non-dominated vectors of small \abbr{LCD} as a subset of $S_\C^{n-1}$, its real dimension is twice as large as in the proof \cite[Proposition 4.1]{BR}. On the other hand, for the real-valued  random variables in ${A}_n$,  one does not expect to obtain better control on the L\'{e}vy concentration function. Thus the proof of \cite[Proposition 4.1]{BR} breaks down as the bounds on the L\'{e}vy concentration function and the size of the net do not match (see also \cite[Remark 4.5]{BR}).

To tackle this obstacle we decompose the vectors according to the angle between their real and imaginary parts. More precisely, we define the {\em real-imaginary de-correlation} as follows:
\begin{dfn} \label{dfn: real-imaginary}
Let $z \in \C^m$ for some positive integer $m$. Denote $V:=V(z):= \left(\begin{smallmatrix} x^{\sf T} \\ y^{\sf T} \end{smallmatrix} \right)$, where $z= x+ \sqrtneg y$. Then we denote the real-imaginary de-correlation of $z$ by
 \[
  d(z):= \left( \det(V V^{\sf T}) \right)^{1/2}.
 \]
 \end{dfn}

 This notion of real-imaginary de-correlation was introduced in \cite{RV no-gaps} to study the no-gap deocalization property of the eigenvectors of a wide class of random matrices. In \cite{RV no-gaps} it is termed as ``real-complex correlation''. Here we deviate from that terminology upon noting that small values of $d(z)$ indicate that the real and the imaginary part of $z$ are close to being linearly dependent.


If a vector $z \in S_\C^{n-1}$ has a large value of $d(z)$, then we call this vector  {\em genuinely complex}, whereas vectors with small real-imaginary de-correlations are termed  {\em essentially real} vectors {(See \eqref{eq:compl-dfn} and \eqref{eq:Real-dfn} for a precise formulation)}. The real and imaginary parts of essentially real vectors being almost linearly dependent it allows us to construct a net whose cardinality is a polynomial of degree $n$ in terms of the mesh. Therefore, one can use the small ball probability estimates from \cite{BR} to show that with high probability, there does not exist any essentially real vector in the kernel of $B^D$ with a small \abbr{LCD}.

The analysis of genuinely complex vectors is more delicate.
Following the recent work of \cite{RV no-gaps} we define a notion of a {\em two-dimensional \abbr{LCD}}. Roughly speaking the two-dimensional \abbr{LCD} $D_2(\cdot)$ of $z= x+ \mathrm{i} y$ identifies a non-trivial $\theta^\star(z):=(\theta^\star_1(z),\theta^\star_2(z)) \in \R^2$ such that $\theta_1^\star(z) x + \theta_2^\star(z) y$ is close to an integer point and $\theta^\star(z)$ has the least possible Euclidean norm among all such choices.  See Definition \ref{dfn:wtz-V} for a precise formulation.
%
%
Using a result of \cite{RV no-gaps} (Theorem 7.5 there) we show that the small ball probability bound of genuinely complex vectors decays roughly as the inverse of the $(2n)$-th power of $D_2(\cdot)$ (see the bound in \eqref{eq: Levy complex}). This probability bound  balances  the cardinality of the net of genuinely complex vectors for which $\Delta(z_{\rm{small}}):=\|\theta_1^\star(z_{\rm{small}}) x_{\rm{small}} + \theta_2^\star(z_{\rm{small}}) y_{\rm{small}}\|_2$ (precise definition of $\Delta(\cdot)$ can be found in Definition \ref{dfn:auxiliary}) is large, where $z_{\rm{small}}= x_{\rm{small}} + {\mathrm i} y_{\rm{small}}$ is the part of $z$ containing the coordinates of small modulus. It allows us to take the union bound over the net of such vectors. To treat the remaining set of genuinely complex vectors, using results from \cite{BR}, we show that, with high probability, there cannot exist a vector $z \in \text{Ker}(B^D)$ with a dominated real part. This additional observation then shows that for any $z \in \text{Ker}(B^D)$  the quantity $\Delta(z_{\rm{small}})$ must also be large. This finishes the outline of the proof of Theorem \ref{thm: smallest singular}.

\subsection{Intermediate singular values}\label{sec:inter-sing-val}
We also need to show that there are not too many singular values of $(A_n- w\sqrt{np}I_n)$ in a small interval around zero. The following theorem does that job. Before stating the theorem, for $i \in [n]$, let us denote $s_i(\cdot)$ to be the $i$-th largest singular value.

\begin{thm}\label{thm:intermed-sing}
Let $A_n$ be an $n \times n$ matrix whose entries are $\{\xi_{i,j} \delta_{i,j}\}_{i,j=1}^n$ where $\{\xi_{i,j}\}_{i,j=1}^n$ are i.i.d.~real-valued random variables with zero mean and unit variance, and $\{\delta_{i,j}\}_{i,j=1}^n$ are i.i.d.~$\dBer(p_n)$ random variables. There exist constants $c_{\ref{thm:intermed-sing}}$ and $C_{\ref{thm:intermed-sing}}$\footnote{the constants $c_{\ref{thm:intermed-sing}}$ and $C_{\ref{thm:intermed-sing}}$ can potentially depend on the tail of the distribution of $\{\xi_i\}_{i=1}^n$.} such that the following holds:
Let $\psi: \N \mapsto \N$ be such that $\psi(n) < n$ and $\min\{p_n \psi(n), \psi^2(n)/n\} \ge C_{\ref{thm:intermed-sing}} \log n$. Then for any $w \in B_\C(0,1)$ we have
\[
\P\left(\bigcup_{i=3 \psi(n)}^{n-1} \left\{s_{n-i}\left(\f{A_n}{\sqrt{np_n}} - w I_n\right) \le c_{\ref{thm:intermed-sing}} \f{i}{n}\right\}\right) \le \f{2}{n^2}.
\]
\end{thm}

\vskip10pt
To prove Theorem \ref{thm:intermed-sing} we follow the approach of \cite{tao_vu}, which was adapted to the sparse case in \cite{bcc-generator, wood}. We first show that the distance of any row of $A_n$ from any given subspace of not very large dimension cannot be too small with large probability. This observation together with a result from \cite{tao_vu} finishes the proof.

\subsection{Weak convergence}\label{subsec:weak-conv}
Recall that to show the integrability of $\log(\cdot)$  we further need to establish the weak convergence of the empirical measure of the singular values of $\f{1}{\sqrt{np}}A_n - w I_n$. Define
\beq\label{eq:bmA_n}
{\bm A}_n^w:= \begin{bmatrix} 0 & \f{1}{\sqrt{np}}A_n - wI_n\\  \f{1}{\sqrt{np}}A_n^* - \bar{w}I_n & 0 \end{bmatrix}
\eeq
and denote by $\nu_n^w$   the \abbr{ESD} of ${\bm A}_n^w$. It can be easily checked that $\nu_n^w$ is the symmetrized version of the empirical measure of the singular values $\f{1}{\sqrt{np}}A_n - w I_n$. Thus, it is enough to prove the weak convergence of $\nu_n^w$.
\begin{thm}\label{thm:weak-conv}
(i) Let ${A}_n$ be an $n \times n$ matrix with entries $a_{i,j}=\delta_{i,j} \cdot \xi_{i,j}$, where $\delta_{i,j}$ are i.i.d. Bernoulli random variables with $\P(\delta_{i,j}=1)=p_n$, and $\xi_{i,j}$ are centered i.i.d.~real-valued random variables with unit variance. Assume $p_n=\omega(\f{\log n}{n})$. Fix any  $w \in B_\C(0,1)$. Then there exists a probability measure $\nu_\infty^w$ such that $\nu_n^w$ converges weakly to $\nu_\infty^w$, in probability.

(ii) If additionally $\{\xi_{i,j}\}_{i,j=1}^n$ have finite fourth moment and $\sum_{n=1}^\infty (n^2 p_n)^{-1} < \infty$ then the above convergence holds almost surely.
\end{thm}

To prove Theorem \ref{thm:weak-conv} we first apply a standard truncation technique which shows that it is enough to prove the weak convergence of $\nu_n^w$ to $\nu_\infty^w$ for bounded $\{\xi_{i,j}\}_{i,j=1}^n$ (see Lemma \ref{lem:truncate_lsd}). This truncation argument requires the additional assumptions of part (ii) of Theorem \ref{thm:weak-conv} to establish the almost sure convergence. 

It is well known that $\nu_{G_n}^w$, the symmetrized version of the empirical law of the singular values of $\f{1}{\sqrt{n}} G_n - w I_n$, where $G_n$ is a complex Ginibre matrix, converges weakly, almost surely to $\nu_\infty^w$. Therefore, to obtain Theorem \ref{thm:weak-conv} it is enough to show that for bounded $\{\xi_{i,j}\}_{i,j=1}^n$, the signed measure $\nu_n^w - \nu^w_{G_n}$ converges weakly to the point mass to zero, as $n \to \infty$, in probability or almost surely, depending on the sparsity parameter $p$.

This is done in Section \ref{sec:weak-conv} using the following two-fold argument. First we establish that both the random probability measures $\nu_n^w$ and $\nu_{G_n}^w$ are close to their expectations, denoted hereafter by $\E \nu_n^w$ and $\E \nu_{G_n}^w$, respectively. This step uses standard concentration inequalities for the spectral measure of Hermitian random matrices.

To complete the proof of Theorem \ref{thm:weak-conv} we then need to show that $\E \nu_n^w$ and $\E \nu_{G_n}^w$ are themselves close to each other. Here we appeal to the Lindeberg replacement principle which was introduced to the random matrix theory in \cite{cha05, cha06} to prove the semicircle law for random symmetric matrices with exchangeable entries on  and above the diagonal. Subsequently, this technique has been used on numerous occasions in the random matrix theory literature.

\section{The structure of the kernel: Vectors with non-dominated real part}\label{sec:non-dom real part}
Recall from Section \ref{sec:smallest_sing} that the main challenge in proving Theorem \ref{thm: smallest singular} is to show that there does not exist a genuinely complex vector $z \in \text{Ker}(B^D)$ with a small two-dimensional \abbr{LCD}. As a first step we show that for any $z \in \text{Ker}(B^D)$, its real part must have a {non-dominated} component with high probability. This is shown in the following result, which is the main result of this section. Before stating the result let us introduce some notation, which is borrowed from \cite{RV no-gaps}: For a number $M<n/2$, to be determined during the course of the proof, we denote by $\text{small}(z)$ the set of the $(n-M)$ coordinates of $z$ having the smallest absolute values. The ties are broken arbitrarily. We also write $z_{\text{small}}=x_{\text{small}}+ \sqrtneg y_{\text{small}}:=z_{\text{small}(z)}$. Hereafter, we drop the subscript in $p_n$ and for ease we write $p$ instead. 

\begin{prop}  \label{thm: large LCD for real part}
Let ${A}_n$ be a matrix with i.i.d.~entries $a_{i,j}=\xi_{ij} \delta_{ij}$, where $\{\xi_{ij}\}$ are i.i.d.~centered real-valued random variables with unit variance and finite fourth moment, and $\{\delta_{ij}\}$ are i.i.d.~$\dBer(p)$ random variables. Set
 \[
   \ell_0:= \left \lceil \frac{\log 1/(8 p)}{\log \sqrt{pn}} \right \rceil.
  \]
Fix $r \in (0,1]$ and $R \ge 1$ such that $\Im(D_n)=r'\sqrt{np}I_n$ with $|r'|  \in [r,1]$ and $\|D_n\| \le R \sqrt{np}$. Fix another positive real $K \ge 1$. Then there exist constants ${C}_{\ref{thm: large LCD for real part}}, \wt{C}_{\ref{thm: large LCD for real part}}, c_{\ref{thm: large LCD for real part}}$, and $\bar{c}_{\ref{thm: large LCD for real part}}$, depending only on $r, R, K$, and the fourth moment of $\{\xi_{ij}\}$, such that the following holds.
Denote $\rho:= (\wt{C}_{\ref{thm: large LCD for real part}}(K+R))^{-\ell_0-6}$ and assume that
  \begin{equation}  \label{eq: min p-1}
  c_{\ref{thm: large LCD for real part}}{\rho^5 pn} >1.
  \end{equation}
 Set
\(
 M= {{C}_{\ref{thm: large LCD for real part}}  \rho^{-4} p^{-1}}.
\)
Then
  \begin{multline*}
    \P \bigg(\exists z \in \mathrm{Ker}(B^D) \cap S_{\C}^{n-1} :
    \norm{\frac{x_{\mathrm{small}}}
    {\norm{x_{\mathrm{small}}}_2}}_{\infty} \ge  \rho p^{1/2} \text{ and } \\
    \norm{{A}_n} \le K  \sqrt{np} \bigg) \le \exp(-\bar{c}_{\ref{thm: large LCD for real part}}np).
  \end{multline*}
\end{prop}

\begin{rmk}
For clarity we only prove Proposition \ref{thm: large LCD for real part} for $r'\in[r,1]$. It will be evident that the proof of the case $r' \in [-1,-r]$ is exactly the same. We spare the details.
\end{rmk}


The key to the proof of Proposition \ref{thm: large LCD for real part} is in showing that if the real part of a vector $z$ is compressible then $\norm{B^D z}_2$ cannot be too small.  This is derived in the following lemma:


\begin{lem}  \label{l: real comp}
Let $B^D, {A}_n, \rho, K,R,r$, and $r'$ be as in Proposition \ref{thm: large LCD for real part}. Then there exist constants $0<c_{\ref{l: real comp}} , c'_{\ref{l: real comp}} ,c'' _{\ref{l: real comp}}, \ol{c}_{\ref{l: real comp}}<\infty$, depending only on $K,R,r$, and the fourth moment of $\{\xi_{i,j}\}$, such that for any $p^{-1} \le M \le c'_{\ref{l: real comp}} n/\log (1/\rho)$,
 \begin{multline*}
   \P ( \exists z=x+\sqrtneg y \in S_{\C}^{n-1}:     \norm{B^{D}z}_2 \le c_{\ref{l: real comp}} \rho \sqrt{np} , \, \\ \norm{x_{\mathrm{small}}}_2 \le c''_{\ref{l: real comp}} \rho, \ \text{ and } \ \norm{{A}_n} \le K \sqrt{np} )    \le \exp(-\ol{c}_{\ref{l: real comp}} np).
 \end{multline*}
\end{lem}


To prove Lemma \ref{l: real comp} we borrow results from \cite{BR}. In \cite[Proposition 3.1]{BR} we showed that, with high probability, there does not exist any real-valued compressible or dominated vector $z$ such that $\norm{\Re(B^D) z}_2$ is small, where $\Re(B^D)$ denotes the real part of the matrix $B^D$.  In \cite[Remark 3.10]{BR} it was also argued that the same conclusion holds for $\norm{B^D z}_2$ when $z$ is now allowed to be complex valued. We will need this result to prove Lemma \ref{l: real comp}. For completeness we state it below.

\begin{prop}[{\cite[Proposition 3.1, Remark 3.10]{BR}}]   \label{p: dominated and compressible}
Let ${A}_n$ be as in Proposition \ref{thm: large LCD for real part}. Fix $K, R \ge 1$, and assume that $D_n$ is a non-random diagonal matrices with complex entries such that $\|D_n\|\le R \sqrt{pn}$. Then there exist constants $0< c_{\ref{p: dominated and compressible}}, \ol{c}_{\ref{p: dominated and compressible}}, c'_{\ref{p: dominated and compressible}},{C}_{\ref{p: dominated and compressible}}, \wt{C}_{\ref{p: dominated and compressible}}, \ol{C}_{\ref{p: dominated and compressible}}< \infty$, depending only on $K, R$, and the fourth moment of $\{\xi_{ij}\}$, such that
for
 \beq
 \frac{\ol{C}_{\ref{p: dominated and compressible}} \log n}{n} \le p \le \frac{1}{10},\label{eq:p_assumption}
 \eeq
and any $p^{-1} \le M \le  c_{\ref{p: dominated and compressible}} n$, we have
   \begin{align*}
   &\P\Big(\exists z \in \text{\rm Dom}(M, (C_{\ref{p: dominated and compressible}}(K+R))^{-4})
    \cup \text{\rm Comp}(M, \rho)\\
    & \  \norm{({A}_n+D_n)z}_2 \le {c}'_{\ref{p: dominated and compressible}}(K+R) \rho \sqrt{np}
   \text{ and } \norm{{A}_n} \le K  \sqrt{pn}\Big)\le \exp(- \ol{c}_{\ref{p: dominated and compressible}} pn),
  \end{align*}
where $ \rho=(\wt{C}_{\ref{p: dominated and compressible}}(K+R))^{-\ell_0-6}$ and $\ell_0$ are as in Proposition \ref{thm: large LCD for real part}.
 \end{prop}

Observe that Proposition \ref{p: dominated and compressible} is stated for the square matrix ${A}_n$. To prove Lemma \ref{l: real comp} we need a version of Proposition \ref{p: dominated and compressible} for $(n-1) \times n$ matrices. As noted in \cite[Remark 3.9]{BR} this follows from an easy adaptation. So, without loss of generality we will use Proposition \ref{p: dominated and compressible} also for $(n-1) \times n$ matrices. The final ingredient for the proof of Lemma \ref{l: real comp} is an estimate on the L\'{e}vy concentration function for incompressible and non-dominated vectors. Such an estimate was derived in \cite[Corollary 3.7]{BR} for real valued vectors and matrices with zero diagonal and i.i.d.~off-diagonal entries. One can investigate its proof to convince oneself that the same proof works for complex valued vectors and matrices with i.i.d.~entries. We state this modified version below.


 \begin{lem}[{\cite[Corollary 3.7]{BR}}]     \label{c: spread vector}
 Let ${A}_n$ be as in Proposition \ref{thm: large LCD for real part}. 
 Then  for any $\alpha >1$, there exist $\be, \gamma >0$, depending on $\alpha$ and the fourth moment of $\{\xi_{ij}\}$, such that for $z \in \C^n$, satisfying
  $\norm{z}_\infty/\norm{z}_2 \le \alpha \sqrt{p}$, we have
 \[
   \cL \left({A}_nz, \beta \cdot \sqrt{pn}\norm{z}_2\right)
   \le \exp (-\gamma n ).
 \]
 \end{lem}


We now proceed to the proof of Lemma \ref{l: real comp}.


\begin{proof}[Proof of Lemma \ref{l: real comp}]
The proof is based on ideas from \cite{Ge}. For ease of writing let us write $c_0:= (C_{\ref{p: dominated and compressible}} (K+R))^{-4}$. We also denote
 \begin{align*}
  \Omega_{D,C}: &= \{\Re(B^D): \ \exists z \in \text{Dom}(M, c_0) \cup \text{Comp}(M, \rho)  \\
  & \qquad \qquad \qquad \norm{B^{D}z}_2 \le {c}'_{\ref{p: dominated and compressible}}(K+R)\rho \sqrt{np} \text{ and } \norm{{A}_n} \le K\sqrt{np} \},
  \end{align*}
 Using Proposition \ref{p: dominated and compressible} we see that $\P(\Omega_{D,C}) \le \exp(-\ol{c}_{\ref{p: dominated and compressible}}np)$.
 We now make the following claim.
 \begin{claim}
 Fix any $J \subset [n]$ of cardinality $M$ and let
 \[
 \cZ_J':= \{  z=x+\sqrtneg y: \ \|x_{\mathrm{small}}\|_2 \le c''\rho \text{ and }  \supp ( x_{[1:M]}) \subset J \},
 \]
 for some small constant $c''$ to be determined later.
 Then
 \[
   \P \left( \left\{\exists z \in \cZ_J'   \text{ such that }  \norm{B^{D}z}_2 \le c \rho \sqrt{np}\right\} \cap \Om_{D,C}^c\right) \le \exp(-\ol{c}n),
 \]
 for some small constants $c$ and $\ol{c}$.
 \end{claim}

 The conclusion of the lemma immediately follows from the claim by taking a union bound over $J \subset [n]$, such that $|J|=M$. Thus we now only need to prove the claim.

 To prove this claim we will first show that if $z \in \cZ_J'$ such that $\|B^D z\|_2$ is small, then $y$, the imaginary part of $z$, belongs to a small neighborhood of a linear image of the subspace spanned by the largest $M$ coordinates of $x$, the real part of $z$. This together with the fact that $\norm{x_{{\rm{small}}}}_2$ is small enables us to obtain a net of $\cZ_J'$ with small cardinality. Finally using the estimate on L\'{e}vy concentration function of Lemma \ref{c: spread vector} and the union bound we finish the proof of the claim. Below we carry out the details.

Fix any $J \subset [n]$ and let $\Re(B^D)|_J$ denote the sub-matrix induced by the columns of $\Re(B^D)$ indexed by $J$. We first condition on a realization of $\Re(B^D)|_J$ and show that for every such realization the conditional probability of the event in the claim is less than $e^{-\bar{c}n}$. Then taking an average over the realizations of $\Re(B^D)|_J$ the proof will be completed.

So let us assume that $z \in \cZ_J'$ be such that $ \norm{B^{D}z}_2 \le c \rho \sqrt{np} $. Then we see that
 \begin{equation} \label{eq: first coordinates}
  \norm{\Re(B^D) x - \Im(B^D)y}_2 \le \norm{B^D z}_2 \le c \rho\sqrt{np}.
 \end{equation}

Notice that $\norm{x_{[M+1:n]}}_2 \le \norm{x_{\text{small}}}_2$ as $x_{[M+1:n]}$  consists of the smallest in the absolute value coordinates of $x$.
Since
\[
 \norm{\Re(B^D)|_{J^c}} \le \norm{\Re(B^D)} \le \norm{B^D} \le \norm{{A}_n}+\norm{D_n} \le (K+R)\sqrt{np},
\]
 applying the triangle inequality we further deduce that
 \begin{align}
 \norm{\Im(B^D) y - \Re(B^D)|_J x_{[1:M]}}_2 & \le c\rho \sqrt{np} + \norm{\Re(B^D)|_{J^c} x_{[M+1:n]}}_2  \notag\\
 & \le c\rho \sqrt{np} + \norm{\Re(B^D)|_{J^c}} \cdot \norm{ x_{\text{small}}}_2 \notag\\
 & \le 2c \rho \sqrt{np}, \notag
 \end{align}
 where in the last step we choose $c''$ so that $ c''(K+R)\le c$. 

Hereafter, we write $\Im(B^D)$ to denote the imaginary part of the matrix $B^D$. Hence $\Im (B^D)$ is a $(n-1)\times n$ matrix whose first column is zero and the last $(n-1)$ columns form a diagonal matrix whose entries are all equal to $r'\sqrt{np}$. Therefore denoting $y|_{[2:n]}$ to be the $(n-1)$ dimensional vector consisting of the last $(n-1)$ coordinates of $y$ we further have that
 \begin{align}\label{eq:y close to span}
  \norm{y|_{[2:n]}- \f{1}{r'\sqrt{np}} \Re(B^D)|_J x_{[1:M]}}_2\le 2 r'^{-1} c \rho \le 2r^{-1}c \rho.
 \end{align}
 Thus \eqref{eq:y close to span} implies that the vector $y|_{[2:n]}$ belongs to a $(2 r^{-1}c \rho)$-neighborhood of the linear subspace $\sE_J':=\text{span}(\Re(B^D)|_J \R^J) \subset \R^{n-1}$. Since $\|x_{[M+1:n]}\|_2 \le c''\rho \le r^{-1} c \rho$ we have that for any $z \in \cZ'_J$, such that $\|B^D z\|_2 \le c \rho \sqrt{np}$, belongs to a $(3 r^{-1}c \rho)$-neighborhood of the set
 \beq
  \sE_J:=\{x+\mathrm{i}y: \ \text{supp}(x) \subset J, \ y|_{[2:n]} \in \sE_J' , y_1 \in [-1,1]\}, \notag
 \eeq
 with $\text{dim}(\sE_J) \le 2M+1$. Since $\cZ_J' \subset S_\C^{n-1}$, applying the triangle inequality and choosing $c \le r/3$ we further see that every vector in $z \in \cZ_J'$,  such that $\|B^D z\|_2\le c\rho\sqrt{np}$, belongs to a $(3 r^{-1}c \rho)$-neighborhood of $(2B_\C^n) \cap \sE_J$.  Therefore we can choose a $(r^{-1} c\rho)$-net $\cN \subset (2B_{\C}^{n}) \cap \sE_J$ of cardinality
 \beq\label{eq:net bound}
  |\cN| \le \left( \frac{12}{c\rho} \right)^{2M+1}  \le \exp (3M \log (12/(c\rho))).
 \eeq
Note that, using the triangle inequality we see that $\cN$ is $(4r^{-1}c\rho)$-net of the set of all vectors $z\in \cZ'_J$ such that $\|B^D z\|_2\le c\rho\sqrt{np}$. Thus, for a $z \in \cZ_J'$ with $\|B^D z\|_2 \le c \rho \sqrt{np}$, there must exist at least one $w \in \cN$ such that $\norm{B^D w}_2 \le 5r^{-1}(K+R)c \rho \sqrt{np}$. Now shrink $c$ such that $10r^{-1}c \le {c}'_{\ref{p: dominated and compressible}}$. With this choice of the constant $c$ we see that $w \notin \text{Dom}(M, c_0) \cup \text{Comp}(M, \rho)$ on the event $\Om_{D,C}^c$.

However, for any $w \notin \text{Dom}(M, c_0)$ we have
\[
\f{\norm{w_{[M+1:n]}}_\infty}{\norm{w_{[M+1:n]}}_2} \le   (c_0\sqrt{M})^{-1} \le  c_0^{-1} \sqrt{p},
\]
where in the last step we used the fact that $M \ge p^{-1}$.
%
Thus applying Lemma \ref{c: spread vector} there exists constants ${c}_\star$ and $\ol{c}$ such that
 \begin{align*}
&  \P\left(\norm{B^D w}_2 \le  {c}_\star \rho \sqrt{np}\Big| \Re(B^D)|_J\right)\\
 \le &\,   \cL\left(\Re(B^D)|_{J^c} w_{[M+1:n]}, {c}_\star \norm{w_{[M+1:n]}}_2\sqrt{np} \right) \le \exp(-2\ol{c}n),
 \end{align*}
Hence, by the union bound,
 \[
   \P\left(\exists w \in \cN: \ \norm{B^{D}w}_2 \le {c}_\star \rho \sqrt{np} \Big| \Re(B^D)|_J\right)
   \le |\cN| \cdot \exp(-2\ol{c}n) \le \exp (-\ol{c}n),
 \]
 where the last step follows from the bound \eqref{eq:net bound} and the fact that $M \log(1/\rho)<{c}'n$ for a sufficiently chosen small constant ${c}'$. Thus shrinking $c$ again such that $5r^{-1} (K+R)c \le c_\star$ we obtain that
\[
\P\left( \left\{\exists z \in \cZ_J'   \text{ such that }  \norm{B^{D}z}_2 \le c \rho \sqrt{np}\right\} \cap \Om_{D,C}^c \Big| \Re(B^D)|_J\right) \le \exp(-\ol{c}n).
\]
Finally taking an average over all the realizations of $\Re(B^D)|_J$ completes the proof.
\end{proof}




\vskip10pt

We are now ready to prove Proposition \ref{thm: large LCD for real part}.

\begin{proof}[Proof of Proposition \ref{thm: large LCD for real part}]
Let $z \in \text{Ker}(B^D) \cap S^{n-1}_{\C}$. Assume that the event $\Omega_{D,C}^c$ defined above occurs, so $z \notin \text{Dom}(M, c_0) \cup \text{Comp}(M, \rho)$.
We note that if $c_{\ref{thm: large LCD for real part}}$ is chosen sufficiently small then the assumption \eqref{eq: min p-1} implies that
\[
 \frac{M \log(1/\rho)}{c'_{\ref{l: real comp}}n}
  <1,
\]
 whenever $n$ is large enough. So Lemma \ref{l: real comp} can be applied, which implies that, with high probability, $\|x_{\mathrm{small}}\|_2 > c''_{\ref{l: real comp}}\rho$.
On the other hand,
\[
  \norm{x_{\mathrm{small}}}_{\infty} \le \norm{z_{\text{small}}}_{\infty} \le \frac{1}{c_0 \sqrt{M}}.
\]
Combining the last two inequalities, we show that on the event $\Omega_{D,C}^c$,
\beq\label{eq:enlarge-C}
 \frac{\norm{x_{\mathrm{small}}}_{\infty}}{\norm{x_{\mathrm{small}}}_2} \le \frac{1}{c''_{\ref{l: real comp}} \rho c_0 \sqrt{M}},
\eeq
and {the result follows upon choosing $C_{\ref{thm: large LCD for real part}}$ sufficiently large.}
\end{proof}

\vskip10pt

\begin{rmk}
Note that the inequality  \eqref{eq:enlarge-C} continues to hold even if the constant $C_{\ref{thm: large LCD for real part}}$ is increased without changing other constants $\wt{C}_{\ref{thm: large LCD for real part}}, c_{\ref{thm: large LCD for real part}}$, and $\bar{c}_{\ref{thm: large LCD for real part}}$, appearing in Proposition \ref{thm: large LCD for real part}. This implies that, if needed, we can arbitrarily increase the constant $C_{\ref{thm: large LCD for real part}}$. This observation will be used later in the paper.
\end{rmk}

 \section{Net construction:~Genuinely complex case}  \label{sec: net complex}
In this section we show that the set of genuinely complex vectors admits a net of small cardinality. We begin with the relevant definitions.
\begin{dfn}\label{dfn:two-lcd-dfn}
For $y>0$, denote $\log_1(y):=\log y \cdot \bI(y \ge  e)$. Fixing $L \ge 1$, for a non-zero vector $x \in \R^m$, we set
 \beq\label{eq:D_1 define}
    D_1(x):= \inf \left\{ \theta>0: \  \mathrm{dist}(\theta x, \Z^{m}) < 2^5L \sqrt{\log_1 \frac{\norm{\theta x}_2}{2^6 L}} \right\}.
 \eeq
 If $V$ is a $2 \times m$ matrix, define
 \beq\label{eq:D_2_define}
   D_2(V):= \inf \left\{ \norm{\theta}_2: \ \theta \in \R^2, \  \mathrm{dist}( V^{\sf T} \theta, \Z^m) < L \sqrt{\log_1 \frac{\norm{V^{\sf T} \theta}_2}{{2^8} L}} \right\}.
 \eeq
We will call the first version of the \abbr{LCD} one-dimensional, and the second one two-dimensional.
Note that $D_1(\cdot)$ matches with the definition of the \abbr{LCD} used in \cite{BR} up to constants.
\end{dfn}

\begin{rmk}\label{rmk:power-of-2}
The different powers of $2$ appearing in the definitions \eqref{dfn:two-lcd-dfn} and \eqref{eq:D_2_define} play only a technical role. They do not affect most of the proof, and they will be needed in Section \ref{sec: net real} to compare the one and the two-dimensional \abbr{LCD} for almost real vectors (see Lemma \ref{l: D_2 to D_1}).
\end{rmk}

Observe that $D_1(\cdot)$ and $D_2(\cdot)$ are defined for real-valued vectors and matrices, respectively. However, both these notions can be extended for complex valued vectors by the following simple adaptation.

\begin{dfn}\label{dfn:wtz-V}
Consider a complex vector $z=x+\sqrtneg y\in \C^m$. Denote $\tilde{z}:=\tilde{z}(z):= \left(\begin{smallmatrix} x \\ y \end{smallmatrix} \right) \in \R^{2m}$, and define a $2 \times m$ matrix $V :=V(z):= \left(\begin{smallmatrix} x^{\sf T} \\ y^{\sf T} \end{smallmatrix} \right) $.
 Using these two different representations of $z \in \C^m$, we now define:
 \[D_2(z):=D_2(V)   \quad \text{and} \quad D_1(z):= D_1(\tilde{z}).\]
 \end{dfn}

 Let us assume that the infimum in \eqref{eq:D_2_define} is attained at $\theta^\star$. Then from Definition \ref{dfn:two-lcd-dfn}  we have that $D_2(V)$ equals $\|\theta^\star\|_2$. We will see below that the cardinality of the desired net for the genuinely complex vectors also depends on $\|V^{\sf T} \theta^\star\|_2$. However, the infimum in \eqref{eq:D_2_define} is not always achieved. Hence, we have the following definition.


   \begin{dfn}\label{dfn:auxiliary}
   For a  real-valued $2 \times m$ matrix $V$, define
   \begin{multline*}
    \Delta(V):= \liminf_{\t \to 1+}
      \bigg\{ \norm{V^{\sf T} \theta}_2: \ \mathrm{dist} (V^{\sf T} \theta, \Z^m) < L \sqrt{{\log_{1}} \frac{\norm{V^{\sf T} \theta}_2}{{2^8}L}}, \\   \norm{\theta}_2 \le \t D_2(V)\bigg\}.
   \end{multline*}
   As before, for a $z \in \C^m$, we define $\Delta(z):=\Delta(V)$ where  $V = V(z)$. For later use let us note that for any $z \in S_\C^{m-1}$
\beq\label{eq:d-delta-D}
d(z) D_2(z) \le \Delta(z) \le D_2(z),
\eeq
\end{dfn}
where $d(z)$ denotes the real-imaginary de-correlation of $z$ appearing in Definition \ref{dfn: real-imaginary}.
Indeed, the inequalities \eqref{eq:d-delta-D} are immediate from the fact that the singular values of $V^{\sf T}$ are bounded by one and $d(z)$ is the product of the singular values of $V^{\sf T}$.

\begin{rmk}\label{dfn:delta-0}
We take $L=(\delta_0p)^{-1/2}$, where $\delta_0 \in (0,1)$ is a universal constant as in 
\cite[Remark 2.7]{BR}.
\end{rmk}

Equipped with the above definitions, we consider the following simple reduction.
Fix $M< n/2$, $z \in S^{n-1}_{\C}$, and let $J= \text{small}(z)$. 
It can be easily verified that for any $z \in \C^n$  there exists a $\t \in [0, 2\pi)$ such that $z_J=e^{\mathrm{i} \t}(w_1+\mathrm{i} w_2)$, where $w_1 \perp w_2$ and $\|w_2\|_2\le \|w_1\|_2$ . As $z_J \in \text{Ker}(B)$ if and only if $e^{-\mathrm{i} \t} z_J \in \text{Ker}(B)$, without loss of generality, we can only consider the following set
   \begin{multline}\label{eq:cZ-dfn}
     \cZ  :=\{ z \in S_{\C}^{n-1} \setminus( \text{\rm Dom}(M, (C_{\ref{p: dominated and compressible}}(K+R))^{-4}) \cup \text{\rm Comp}(M, \rho)): \\
     z_{\text{small}}=w_1+\mathrm{i} w_2,
      w_1 \perp w_2, \ \norm{w_1}_2 \ge \norm{w_2}_2 \}
   \end{multline}
instead of $S_{\C}^{n-1} \setminus (\text{\rm Dom}(M, (C_{\ref{p: dominated and compressible}}(K+R))^{-4})\cup \text{\rm Comp}(M, \rho))$. Therefore our revised goal is to show that the set of genuinely complex vectors, when specialized to $\cZ$, admits a net of small cardinality.

To this end, fixing a set $J \subset [n]$, we start with constructing a small net for the set of pairs $(\phi,\psi)$ with $\phi \perp \psi$ in the unit sphere of $\R^J \times \R^J$ for which the value of the two-dimensional \abbr{LCD}, the auxiliary parameter $\Delta(\cdot)$, and the de-correlation $d(\phi,\psi)=\norm{\phi}_2\norm{\psi}_2$ are approximately constant. The condition on the two-dimensional \abbr{LCD} means that there exists a linear combination of the vectors $\phi$ and $\psi$ which is close to an integer point. Our aim is to use this linear combination  to construct separate approximations of $\phi$ and $\psi$.




For any $\gamma>0$, let us denote  $\Z_{\gamma}^{J}:= \Z^{J} \cap \gamma B_2^{|J|}$.
   Using a simple volumetric comparison argument we have following estimate on $|\Z_\gamma^{J}|$ :
   \begin{equation} \label{eq: integer net}
    |\Z_\gamma^{J}| \le \left( C_0 \left( \frac{\gamma}{\sqrt{|J|}}+1 \right) \right)^{|J|},
   \end{equation}
   for some absolute constant $C_0$. The main technical result of this section is the following lemma.
   \begin{lem} \label{l: net in S_J}
   Let $d \in (0,1)$, and $0< \alpha \le d\gD \le \Delta \le \gD$.
         Define the set
   \begin{align*}
     S_J(\gD,\Delta,d)
     &:=\{ (\phi,\psi) \in \R^J \times \R^J: \  \phi \perp \psi, \ \norm{\phi}_2 \in [{1}/{2}, 1], \norm{\psi}_2 \in [d,3d] \\
      &\quad   \exists \zeta \in \R^2 \ \norm{\zeta}_2 \in [\gD, 2\gD],  \ \norm{\zeta_1 \phi+ \zeta_2 \psi}_2 \in [\Delta, 2\Delta], \\
       & \quad \text{and } \mathrm{dist}(\zeta_1 \phi + \zeta_2 \psi, \Z^{J})< \alpha \}
   \end{align*}
Then, there exists a $\left( \frac{C_{\ref{eq: integer net}} \alpha}{\gD} \right)$-net $\cM_J(\gD, \Delta,d) \subset S_J(\gD,\Delta,d)$ such that
     \[
      |\cM_J(\gD,\Delta,d)|
      \le  \left( \bar{C}_{\ref{eq: integer net}} \frac{d \gD^2}{\alpha} \cdot \left( \frac{1}{\sqrt{|J|}}+\frac{1}{\Delta} \right)  \right)^{|J|}
     \cdot \left(\frac{ \gD}{\alpha} \right)^2,
     \]
for some {absolute} constants $C_{\ref{eq: integer net}}$, and $\bar{C}_{\ref{eq: integer net}}$.
   \end{lem}


  This lemma provides a significant improvement over the standard volumetric estimate yielding $\left(c\gD^2 / \a^2 \right)^{|J|} $. This improved bound precisely  balances the term appearing in the small ball probability estimate. Note that the bounds on $\norm{\phi}_2$ and $\norm{\psi}_2$ imply that the de-correlation $d(\phi,\psi)$ is approximately constant in the set $S_J(\gD,\Delta,d)$, whereas the bounds on $\norm{\zeta}_2$ and $\dist(\zeta_1 \phi+ \zeta_2 \psi, \Z^J)$ ensure that the two-dimensional \abbr{LCD} and the auxiliary parameter $\Delta(\cdot)$ are approximately constant. {Lemma \ref{l: net in S_J} deals with the case when the de-correlation between $\phi$ and $\psi$ is relatively large, represented by the assumption $d \ge \alpha/\gD$, which in turn implies that the angle between the real and the imaginary part of the vectors is non-negligible. In Section \ref{sec: kernel complex} we use this criterion to formally define the notion of genuinely complex vectors.}
 
 We also point out to the reader that a net for the genuinely complex vectors was constructed in \cite{RV no-gaps} (see Lemma 11.2 there). Using that net and the bound on the L\'{e}vy concentration function (see \cite[Theorem 10.3]{RV no-gaps}) it was then showed that there are no vectors with \abbr{LCD} less than $O(n)$ in the kernel of the matrix in context, with high probability. Repeating the same argument here one can at best hope to show that there does not exist any vector in ${\rm Ker}(B^D)$ with \abbr{LCD} $O(np)$, with high probability. To treat the remaining vectors again one needs to apply bounds on L\'{e}vy concentration function (for example, see the bound derived in Proposition \ref{prop: Levy scalar}). However, for such vectors the bound is too weak to deduce almost sure convergence of the \abbr{ESD} of $\f{1}{\sqrt{np}}A_n$. 
 
 Hence, we need to proceed differently. In particular, we use all the parameters $D, \Delta,$ and $d$ to find a net of appropriate size such that its cardinality balances with the small ball probability derived in Proposition \ref{prop: Levy vector} so that we are able to show that there are no vectors in ${\rm Ker}(B^D)$ with \abbr{LCD} less than $\exp(O(np \rho^4))$ with high probability. 
 
 

   \begin{proof}[Proof of Lemma \ref{l: net in S_J}]

  Assume that there exists $\zeta:= (\zeta_1, \zeta_2) \in \R^2$ and $q \in \Z^J$ satisfying
   \begin{equation}\label{eq: small lin comb}
    \norm{\zeta_1 \phi + \zeta_2 \psi}_2 \in [\Delta, 2 \Delta] \quad \text{and} \quad \norm{\zeta_1 \phi + \zeta_2 \psi -q}_2 <\a.
   \end{equation}
   We consider two cases depending on the size of $\zeta_1$.
   Let us start with the case when this value is small.
   Consider the set
   \begin{multline*}
     S_J^0(\gD, \Delta, d):=\{ (\phi,\psi) \in S_J(\gD,\Delta, d): \ \exists (\zeta_1,\zeta_2) \in \R^2, \ \norm{(\zeta_1,\zeta_2)}_2 \in [\gD,2\gD], \\
       |\zeta_1| \le \frac{1}{2} d\gD, \norm{\zeta_1 \phi + \zeta_2 \psi}_2 \in [\Delta, 2 \Delta],  \text{and }\exists q \in \Z^J \text{ such that }\\
        \norm{\zeta_1 \phi + \zeta_2 \psi -q}_2 <\a \}.
    \end{multline*}
Since $d  <1 $, note that the condition on $\zeta_1$ implies that $\gD/2 \le |\zeta_2| \le 2\gD$. Hence
    \begin{equation}  \label{eq: Delta is dD}
     \Delta \le \norm{\zeta_1 \phi + \zeta_2 \psi}_2 \le \frac{1}{2} d\gD \norm{\phi}_2+ 2\gD \norm{\psi}_2 \le 7 d \gD.
    \end{equation}
    We will approximate $\phi$ using the standard volumetric net and use \eqref{eq: small lin comb} to construct a small net for $\psi$.
   To this end, consider $(\phi,\psi) \in S_J^0(\gD,\Delta, d)$ and let $(\zeta_1,\zeta_2) \in \R^2$ be the corresponding vector (i.e.~for which \eqref{eq: small lin comb} holds). Then, by the triangle inequality,
 \[
     \norm{q}_2 < \a + 2 \Delta
     \le 3 \Delta,
    \]
i.e.~$q \in \Z_{3 \Delta}^J$.
Denote by $\cN_\phi$ an $(\a/\gD)$-net in $B_2^J$ with
\[
     |\cN_\phi| \le \left( \frac{3\gD}{\a} \right)^{|J|}.
\]
Choose $\phi' \in \cN_\phi$ such that $\norm{\phi-\phi'}_2 <\a/\gD$.
Then
\[
     \norm{\zeta_1 \phi' +\zeta_2 \psi -q}_2 < \a+ |\zeta_1| \cdot \norm{\phi-\phi'}_2 < 2 \a,
\]
as $|\zeta_1| \le \f{1}{2}d\gD \le \gD$. Therefore
\[
\norm{ \psi + \frac{\zeta_1}{\zeta_2} \phi' - \frac{\gD/2}{\zeta_2}\cdot \frac{q}{\gD/2}}_2 < \frac{2 \a}{|\zeta_2|}< \frac{4\a}{\gD}.
\]
We observe that
\beq\label{eq:coeff_less_1_1}
\left| \frac{\zeta_1}{\zeta_2} \right| \bigvee  \left| \frac{\gD/2}{\zeta_2}  \right| \le 1,
\quad \norm{\phi'}_2 \le 1, \quad \text{and} \quad \norm{\frac{q}{\gD/2}}_2 \le 6 \frac{\Delta}{\gD} \le 6,
\eeq
where the last inequality follows from our assumption $\Delta \le \gD$.
Next let $\cN_\Box$ be an $(\a/\gD)$-net in the unit square in $\R^2$ with $|\cN_\Box| \le (6\gD/\a)^2$. Using \eqref{eq:coeff_less_1_1}, and applying the triangle inequality, we now see that
 there exists $(x_1,x_2) \in \cN_\Box$ such that
 \[
     \norm{ \psi - x_1 \phi' - x_2 \frac{q}{\gD/2}}_2  < \frac{11\a}{\gD}.
\]
Hence,
    \begin{align*}
     \cM_J^0(\gD,\Delta, d)
     := \Bigg\{ \left(\phi', x_1 \phi'+x_2 \frac{q}{\gD/2} \right):
      \phi' \in \cN_\phi,  q \in \Z_{3 \Delta}^J,  (x_1,x_2) \in \cN_\Box \Bigg \},
    \end{align*}
    is a $\frac{12\alpha}{\gD}$-net of  $S_J^0(\gD,\Delta, d)$, with
    \begin{multline*}
     | \cM_J^0(\gD,\Delta, d)|
     \le |\cN_\phi| \cdot |\Z_{3 \Delta}^J| \cdot |\cN_\Box|
      \le \left( \frac{3C_0 \gD}{\alpha} \cdot \left( \frac{3 \Delta}{\sqrt{|J|}}+1 \right) \right)^{|J|}
         \cdot \left( \frac{6 \gD}{\alpha} \right)^2 \\
     \le \left( 63 C_0 \frac{d \gD^2}{\alpha} \cdot \left( \frac{1}{\sqrt{|J|}}+\frac{1}{\Delta} \right) \right)^{|J|}
        \cdot \left( \frac{6 \gD}{\alpha} \right)^2,
    \end{multline*}
   where \eqref{eq: integer net} has been used to bound $|\Z_{3 \Delta}^J|$ and \eqref{eq: Delta is dD} has been used to replace $\Delta$ by $d\gD$ in the last inequality. 
       \vskip 0.1in

 Turning to prove the case of $|\zeta_1|>\f{1}{2}d\gD$ we denote \[S_J^1(\gD,\Delta,d) := S_J(\gD,\Delta,d) \setminus S_J^0(\gD,\Delta,d).\] That is,
    \begin{multline*}
     S_J^1(\gD,\Delta,d):=\{ (\phi,\psi) \in S_J(\gD,\Delta,d): \ \exists (\zeta_1,\zeta_2) \in \R^2, \ \norm{(\zeta_1,\zeta_2)}_2 \in [\gD,2\gD],\\ |\zeta_1| \in \left[ \frac{1}{2} d \gD, 2 \gD \right],
     \norm{\zeta_1 \phi + \zeta_2 \psi}_2 \in [\Delta, 2\Delta]  \text{and } \exists q \in \Z^J \text{ such that } \\ \norm{\zeta_1 \phi + \zeta_2 \psi -q}_2 <\a \}.
    \end{multline*}
Now let us construct a net in $S_J^1(\gD,\Delta,d)$. Our strategy here is opposite to what we used in the previous case. Namely, we use the volumetric approximation for $\psi$ and then use \eqref{eq: small lin comb} to approximate $\phi$.  To this end, consider any  $(\phi, \psi) \in  S_J^1(\gD,\Delta, d)$ and let $(\zeta_1,\zeta_2) \in \R^2$ be the corresponding vector. As in the previous case we see $\norm{q}_2 < 3 \Delta$, i.e.~$q \in \Z_{3 \Delta}^J$. Since $|\zeta_1| \ge \f{1}{2}d\gD$ and $|\zeta_2| \le 2\gD$ we also see that $24 \norm{\zeta_1 \phi}_2 \ge 6 d \gD \ge \norm{\zeta_2 \psi}_2$. Therefore
    \beq\label{eq:delta-zeta1}
     \Delta \le \norm{\zeta_1 \phi}_2+ \norm{\zeta_2 \psi}_2 \le 25 \norm{\zeta_1 \phi}_2 \le 25 |\zeta_1|.
    \eeq
Recall that by assumption, $\a/\gD \le d$. Hence, we see that
    \[
      |\cN_\psi| \le \left( \frac{9 d \gD}{\a} \right)^{|J|},
    \]
where $\cN_\psi$ is an $(\a/\gD)$-net in $3d B_2^J$.  Since $\norm{\psi}_2 \le 3d$, there exists $\psi' \in \cN_\psi$ such that $\norm{\psi-\psi'}_2 < \a/\gD$.
    As in the previous case, this yields
    \[
      \norm{\zeta_1 \phi + \zeta_2 \psi' -q}_2 <\a + |\zeta_2| \cdot \norm{\psi-\psi'}_2
      \le 3 \a,
    \]
    and so
    \[
     \norm{\phi+\frac{\Delta \zeta_2}{50 \gD \zeta_1} \cdot \frac{50 \gD \psi'}{\Delta} - \frac{\Delta }{25 \zeta_1} \cdot \frac{25 q}{\Delta}}_2
     < \frac{3 \a}{|\zeta_1|}
     \le \frac{75 \a}{\Delta},
    \]
    where we have used \eqref{eq:delta-zeta1} in the last step. Note that
    \[
     \left| \frac{\Delta \zeta_2}{50 \gD \zeta_1} \right| \bigvee \left|  \frac{\Delta }{25 \zeta_1} \right| \le 1,
     \quad \norm{ \frac{50 \gD \psi'}{\Delta}}_2 \le \frac{150 d\gD}{\Delta} \le 150, \quad \text{and} \quad
     \norm{ \frac{25 q}{\Delta}}_2 \le 75.
    \]
   Let $\cN_\Box$ be the same $(\a/\gD)$-net in the unit square as in the previous case. Since $\Delta \le \gD$, combining the previous estimates with the triangle inequality, we have that there exists a $(x_1,x_2) \in \cN_\Box$ such that
    \[
     \norm{\phi-x_1 \cdot \frac{50 \gD \psi'}{\Delta} - x_2 \cdot \frac{25q}{\Delta}}_2
     < \frac{300\a}{\Delta}.
    \]
Using the fact $\Delta \le \gD$ again, we now obtain an $(\a/\gD)$-net $\cM_\phi$ in $\left(\frac{300 \a}{\Delta} \right) \cdot B_2^J$ with
    \[
      |\cM_\phi| \le \left( \frac{900 \gD}{\Delta} \right)^{|J|}.
    \]
    Thus there exists $\nu \in \cM_\phi$ such that
    \[
     \norm{\phi-x_1 \cdot \frac{50 \gD \psi'}{\Delta} - x_2 \cdot \frac{25 q}{\Delta} -\nu}_2
     < \frac{\a}{\gD}.
    \]
    This implies that the set
    \begin{multline*}
     \cM_J^1(\gD,\Delta, d)
     := \Bigg \{ \left(x_1 \cdot \frac{50 \gD \psi'}{\Delta} +x_2 \cdot \frac{ 25 q}{\Delta}+\nu, \ \psi' \right):  \\
      \psi' \in \cN_\psi, \ q \in \Z_{3 \Delta}^\gamma, \ \nu \in \cM_\phi, \ (x_1,x_2) \in \cN_\Box \Bigg \}
    \end{multline*}
    is a $(2\a/\gD)$-net for $S_J^1(\gD,\Delta,d)$. We observe that 
    \begin{multline*}
     |\cM_J^1(\gD,\Delta,d)|
     \le |\cN_\psi| \cdot |\Z_{3 \Delta}| \cdot |\cM_\phi| \cdot |\cN_\Box|
    \\ \le \left( \bar{C}  \frac{d \gD}{\a} \cdot  \left( \frac{ \Delta}{\sqrt{|J|}}+1 \right) \cdot \frac{\gD}{\Delta} \right)^{|J|}
     \cdot \left(\frac{\gD}{\a} \right)^2\\
     \le \left( \bar{C} \frac{d \gD^2}{\a} \cdot \left( \frac{1}{\sqrt{|J|}}+\frac{1}{\Delta} \right)  \right)^{|J|}
     \cdot \left(\frac{\gD}{\a} \right)^2,
    \end{multline*}
where $\bar{C}$ is some absolute constant.

   \vskip 0.1in

Since $S_J(\gD,\Delta, d) = S_J^0(\gD,\Delta, d) \cup S_J^1(\gD,\Delta, d)$, it therefore means that
   \[
     \cM_J(\gD,\Delta, d) := \cM_J^0(\gD,\Delta, d) \cup \cM_J^1(\gD,\Delta, d)
   \]
   is a $(C\a/\gD)$-net for the set  $S_J(\gD,\Delta, d)$, where $C$ is an absolute constant.

 The net $\cM_{J}(\gD,\Delta, d)$  constructed above is not necessarily contained in $S_J(\gD,\Delta, d)$.
 However, we can construct a new net by replacing each point of this net by a point of the set $S_J(\gD,\Delta,d)$ which is within distance $C\alpha/\gD$ from this point.
 If a $(C\alpha/\gD)$-close  point does not exist, we skip the original point. Such process creates a $(2C\alpha/\gD)$-net contained in $S_J(\gD,\Delta, d)$ without increasing the cardinality.
  Thus the lemma is proved.
   \end{proof}


Building on Lemma \ref{l: net in S_J} we now obtain a net with small cardinality for the collection of vectors $z$ for which $D_2(z_{\rm{small}}/\|z_{\rm{small}}\|_2) \approx \gD, \Delta(z_{\rm{small}}/\|z_{\rm{small}}\|_2) \approx \Delta,$  and $d(z_{\rm{small}}/\|z_{\rm{small}}\|_2) \approx d$, where we recall that the vector $z_{\text{small}}$ contains  $n-M>n/2$  coordinates of $z$ having the smallest magnitude. To this end, let us define the following set: 
  \begin{multline}  \label{eq: def Z(D,d)}
  \cZ(\gD,\Delta, d):=\{  z \in \cZ :    \ D_2(z_{\text{small}}/ \norm{z_{\text{small}}}_2) \in [\gD, (3/2)\gD], \\
   \Delta(z_{\text{small}}/ \norm{z_{\text{small}}}_2) \in [\Delta, (3/2)\Delta],
      d(z_{\text{small}}/ \norm{z_{\text{small}}}_2) \in [d, (3/2)d] \}.
 \end{multline}

As will be seen in Section \ref{sec: kernel complex}, the small ball probability for the images of such vectors is controlled by the values of the two-dimensional \abbr{LCD} and the real-imaginary de-correlation. So we partition this set according to $D_2(\cdot)$, $\Delta(\cdot)$, and $d(\cdot)$. The net $\cM_J(\gD,\Delta,d)$ provides a net for the vectors which have $D_2(\cdot) \approx \gD, \Delta(\cdot) \approx \Delta,$  and $d(\cdot) \approx d$.  This is shown in the proposition below.
 \begin{prop} \label{prop: net in Z(D,d)}
  Let $d \in (0,1)$,  $\gD, \Delta>1$, and denote
  \beq\label{eq:dfn-a-0}
\alpha:= L \sqrt{{\log_1} \f{\Delta}{2^7 L}}.
  \eeq
Assume that
  \(
   \alpha \le d\gD \le \Delta \le \gD.
  \)
Then there exist absolute constants $C_{\ref{prop: net in Z(D,d)}}$, $\bar{C}_{\ref{prop: net in Z(D,d)}}$, and a set $\cN(\gD,\Delta, d) \subset \cZ(\gD,\Delta,d)$ with
  \[\displaybreak[3]
   |\cN(\gD,\Delta, d)|
   \le \bar{C}_{\ref{prop: net in Z(D,d)}}^n \left(\frac{n}{\rho M}  \cdot  \frac{\gD}{\alpha} \right)^{5 M} \cdot \left(  \frac{d \gD^2}{\alpha} \cdot \left( \frac{1}{\sqrt{n}}+\frac{1}{\Delta} \right)  \right)^{n-M}
  \]
  having the following approximation properties: Let $z \in \cZ(\gD,\Delta, d)$ be any vector and denote $J={\rm{small}}(z)$. Then there exists $w \in \cN(\gD,\Delta,d)$ such that
 \begin{multline*}
 \norm{\frac{z_J}{\norm{z_J}_2}-\f{w_J}{\norm{w_J}_2}}_2 < C_{\ref{prop: net in Z(D,d)}} \frac{\alpha}{\gD}, \quad \norm{z_{J^c} - w_{J^c}}_2 \le C_{\ref{prop: net in Z(D,d)}} \frac{\rho \alpha}{\gD}, \\
  |\|z_{J}\|_2 - \|w_J\|_2| \le C_{\ref{prop: net in Z(D,d)}} \frac{\rho \alpha}{\gD}.
 \end{multline*}
 \end{prop}

 \begin{rmk}\label{rmk:net-subset-cM}
Note that Lemma \ref{l: net in S_J} holds also for any subset of $S_J(\gD,\Delta,d)$. That is, given any $\cS \subset S_J(\gD,\Delta,d)$ there exists a net $\cM_J^\cS(\gD,\Delta,d) \subset \cS$ with the same properties as in Lemma \ref{l: net in S_J}. We use this version of Lemma \ref{l: net in S_J} to prove Proposition \ref{prop: net in Z(D,d)}. 
 Similarly,  we will see that given any $\cS \subset \cZ(\gD,\Delta,d)$ there exists a set $\cN^\cS(\gD,\Delta,d) \subset \cS$ with the same approximation properties and the cardinality bound as in Proposition \ref{prop: net in Z(D,d)}. This version of Proposition \ref{prop: net in Z(D,d)} will be used in Section \ref{sec: kernel complex}.
 \end{rmk}

 In proving Proposition \ref{prop: net in Z(D,d)} our strategy will be to use the net $\cM_J^\cS(\gD,\Delta,d)$, for some suitable choice of $\cS$, obtained from Lemma \ref{l: net in S_J}, to approximate the small coordinates. The cardinality of the net to approximate the large ones will be obtained by a simple volumetric estimate.


 \begin{proof}[Proof of Proposition \ref{prop: net in Z(D,d)}]
  Fix a set $J \subset [n], \ |J|=n-M$, and denote
  \[
   \cZ_J(\gD,\Delta, d):= \{ z \in \cZ(\gD,\Delta, d): \ \text{small}(z)=J \}.
  \]
  Let us now construct an approximating set for this subset.
Denote $\phi +\sqrtneg v=\phi(z)+\sqrtneg \psi(z):=z_J/\norm{z_J}_2 \in \C^J$. Recalling the definition of $\Delta (\phi+\sqrtneg \psi)$ we see that there exists $\zeta \in \R^2$ such that
  \begin{multline*}
   \norm{\zeta}_2 \le(4/3) D_2(\phi+\mathrm{i} \psi) \le 2\gD, \qquad \norm{\zeta_1 \phi+\zeta_2 \psi}_2 \le (4/3) \Delta(\phi+\mathrm{i} \psi)\le 2 \Delta, \\
   \text{and }   \text{dist} (\zeta_1 \phi+\zeta_2 \psi, \Z^n) < L \sqrt{{\log_{1}} \frac{\norm{\zeta_1 \phi+\zeta_2 \psi}_2}{{2^8}L}}
   \le  L \sqrt{{\log_{1}} \frac{ \Delta}{2^7 L}}= \alpha.
  \end{multline*}
Further recall that $d(z_{\text{small}}/\|z_{\text{small}}\|_2) = \norm{\phi}_2\norm{\psi}_2$ and note that by our convention we have $\norm{\phi}_2 \in [1/2,1]$. Thus  we deduce that $\norm{\psi}_2 \in [d,3d]$. Hence
  $(\phi,\psi) \in S_J(\gD,\Delta,d)$, and in particular $(\phi,\psi) \in \cS$ where $\cS:=\{(\phi(z),\psi(z)): z \in \cZ(\gD,\Delta,d)\}$. So it can be approximated by an element of $\cM_J^\cS(\gD,\Delta, d)$. Set
  \[
   \cM_J:= \{ \phi+\sqrtneg \psi: (\phi,\psi) \in \cM_J^\cS(\gD,\Delta,d) \}.
  \]
  Then for any  $z \in \cZ_J(\gD,\Delta,d)$, there exists $w' \in \cM_J$ such that
   \[
     \norm{\frac{z_{J}}{\norm{z_J}_2}-w'}_2 < C_{\ref{eq: integer net}} \frac{\alpha}{\gD}.
   \]
For the set $J^c$, we will use a net satisfying the volumetric estimate.
Since $z \in S_\C^{n-1}$, there exists a set $\cN_{J^c}$ with
  \[
   |\cN_{J^c}| \le \left( \frac{3}{ C_{\ref{eq: integer net}} \rho} \cdot \frac{\gD}{\alpha}\right)^{2M}.
  \]
such that for every $z \in \cZ_J(D,\Delta,d)$ there exists a $w_{J^c} \in \cN_{J^c}$ for which
\[
 \norm{{z_{J^c}}-w_{J^c}}_2 \le  C_{\ref{eq: integer net}}\f{ \rho \alpha}{\gD}.
 \]
 Finally consider a net $\cN_{[0,1]}$ with $|\cN_{[0,1]}| \le 3\gD/(C_{\ref{eq: integer net}} \rho \alpha)$ such that for every $z \in \cZ_J(\gD,\Delta,d)$ there exists a $\rho' \in [0,1]$ for which
\[
\left| \norm{z_J}_2 - \rho'\right| \le C_{\ref{eq: integer net}}\f{ \rho \alpha}{\gD}.
\]
Now let us define
    \[
   \cN(\gD,\Delta, d)
   := \bigcup_{|J|=n-M} \left \{ \rho' w' + w'': \ \rho' \in \cN_{[0,1]}, \ w' \in \cM_J, \ w'' \in \cN_{J^c} \right \}.
  \]
Then  for any $z \in \cZ(\gD,\Delta,d)$ there exists a $w=w_J+w_{J^c} \in \cN(\gD,\Delta, d)$ such that
 \begin{multline*}
  \norm{\frac{z_J}{\norm{z_J}_2}-\f{w_J}{\norm{w_J}_2}}_2 < C_{\ref{eq: integer net}} \frac{\alpha}{\gD}, \quad \norm{z_{J^c} - w_{J^c}}_2 \le C_{\ref{eq: integer net}} \frac{\rho \alpha}{\gD}, \\
  |\|z_{J}\|_2 - \|w_J\|_2| \le C_{\ref{eq: integer net}} \frac{\rho \alpha}{\gD}.
 \end{multline*}

The set $\cN(\gD,\Delta, d)$ thus constructed may not be contained in $\cZ(\gD,\Delta,d)$. However, as in the proof of Lemma \ref{eq: integer net} this can be rectified easily. It thus remains to bound the cardinality of $\cN(\gD,\Delta, d)$.
  By Lemma \ref{l: net in S_J}, we have
  \begin{align*}
   &|\cN(\gD,\Delta, d)|
   \le \sum_{|J|=n-M} |\cN_{J^c}| \cdot |\cN_{[0,1]}|\cdot |\cM_J| \\
   &\le \binom{n}{M} \cdot \left( \frac{3}{{C}_{\ref{l: net in S_J}} \rho} \cdot \frac{\gD}{\alpha} \right)^{2M+1} \cdot \left( \bar{C}_{\ref{l: net in S_J}} \frac{d \gD^2}{\a} \cdot \left( \frac{1}{\sqrt{n-M}}+\frac{1}{\Delta} \right)  \right)^{n-M}
     \cdot \left(\frac{ \gD}{\a} \right)^2.
  \end{align*}
Since $1 \le M<n/2$ the required estimate follows from a straightforward calculation.
This completes the proof.
 \end{proof}


\section{The structure of the kernel in the genuinely complex case} \label{sec: kernel complex}
In this section, our goal is to show that with high probability, any genuinely complex vector in $\text{Ker}(B^D)$ has a large two-dimensional \abbr{LCD}.
 Before proceeding any further let us formally define the notion of genuinely complex vectors: 
  \begin{multline}\label{eq:compl-dfn}
  \text{Compl}(\cZ)
    := \Bigg \{  z \in \cZ: \\
    d(z_{\text{small}}/\norm{z_{\text{small}}}_2)
   \ge  \frac{4 L}{D_2(z_{\text{small}}/\norm{z_{\text{small}}}_2)} \sqrt{\log_{1} \frac{\Delta(z_{\text{small}}/\norm{z_{\text{small}}}_2)}{{2^7}L}} \Bigg \},
  \end{multline}
  where we recall the definition of $\cZ$ from \eqref{eq:cZ-dfn}.
Equipped with the notion of genuinely complex vectors we state the main result of this section.
 \begin{thm} \label{thm: kernel complex}
Let $B^D, {A}_n, \rho, K,R,r$, and $r'$ be as in Proposition \ref{thm: large LCD for real part}. Then there exist constants ${c}_{\ref{thm: kernel complex}}, c'_{\ref{thm: kernel complex}}$, 
depending only on $K,R, r,$ and the fourth moment of $\{\xi_{ij}\}$, such that if $p$ satisfies the inequality
\begin{equation}
  {c}_{\ref{thm: kernel complex}}{\rho^5 pn} >1, \label{eq:p-rho-gc}
\end{equation}
then we have
   \begin{multline*}
     \P \Big(\exists z \in \mathrm{Compl}(\cZ) \cap \mathrm{Ker}(B^D):
     D_2(z_{\rm{small}}/\norm{z_{\rm{small}}}_2) \le \exp(c'_{\ref{thm: kernel complex}} \frac{n}{M}),   \\
       \norm{{A}_n} \le K \sqrt{pn} \Big) \le \exp(-\bar{c}_{\ref{thm: kernel complex}} np),
   \end{multline*}
where \(
 M= {{C}_{\ref{thm: large LCD for real part}}  \rho^{-4} p^{-1}}
\) and $z_{\rm{small}}$ is the smallest $(n-M)$ coordinates of $z$ in modulus.

 \end{thm}

The proof of Theorem \ref{thm: kernel complex} is carried out by the following two-fold argument. Using Proposition \ref{prop: net in Z(D,d)} we show that the subset of vectors  in $\text{Compl}(\cZ)$ that have a large value of $\Delta(z_{\text{small}}/\norm{z_{\text{small}}}_2)$ admits a net of small cardinality. This observation together with an estimate on the small ball probability, obtained  from \cite[Theorem 7.5]{RV no-gaps}, yields the desired conclusion for vectors $z \in \text{Compl}(\cZ)$ which possess a large value of $\Delta(z_{\text{small}}/\norm{z_{\text{small}}}_2)$ (see Proposition \ref{prop: kernel complex}). For the other case, we first show that such vectors, upon rotation, must have a dominated real part. Applying Proposition \ref{thm: large LCD for real part} we show that this is impossible with high probability, which finishes the proof of Theorem \ref{thm: kernel complex}. The rest of this section is devoted to implementing this idea.

First let us consider the case of large $\Delta(z_{\text{small}}/\norm{z_{\text{small}}}_2)$. For such vectors we prove that the following holds:

 \begin{prop} \label{prop: kernel complex}
   Let $s \ge 1$.
        Define the set $\cZ(s)$  by
  \[
    \cZ(s)
    := \left \{  z \in \mathrm{Compl}(\cZ): \ \Delta(z_{\rm{small}}/\norm{z_{\rm{small}}}_2) \ge s L\right\}.
  \]
Let ${A}_n, B^D, K,R,$ and $\rho$ be as in Theorem \ref{thm: kernel complex}.   Then there exist $s_{\ref{prop: kernel complex}}, r_\star > 1$, and $c'_{\ref{prop: kernel complex}}>0$, depending only on $K, R$, and the fourth moment of $\{\xi_{ij}\}$ such that for any $ r^2_\star p^{-1} \le M \le \rho n$ we have,
   \begin{multline*}
     \P \Big(\exists z \in \cZ(s_{\ref{prop: kernel complex}}) \cap {\rm{Ker}}(B^D): \,
     D_2(z_{\rm{small}}/\norm{z_{\rm{small}}}_2) \le \exp(c'_{\ref{prop: kernel complex}} n/M) \\
      \text{ and } \norm{{A}_n} \le K \sqrt{pn} \Big) \le e^{- n}.
   \end{multline*}
 \end{prop}

To prove Proposition \ref{prop: kernel complex} we use bounds on L\'{e}vy concentration function.
%
%
%
Using such bounds we show that for any vector $z \in \cZ(s)$ the $\ell_2$ norm of $B^Dz$ cannot be too small with large probability. From the net constructed in Section \ref{sec: net complex} it follows that $\cZ(s)$ admits a net of small cardinality which enables us to take the union bound and complete the proof of Proposition \ref{prop: kernel complex}.

As mentioned above to prove Proposition \ref{prop: kernel complex} we need to derive bounds on the small ball probability, in particular on L\'{e}vy concentration function. 
Before deriving the bounds such bounds 
we need to fix some notation. Let $z \in \C^m$ and $J \subset [m]$. Denote $z_J:=(z_i)_{i \in J}\in \C^{J}$ and $V_J:=V(z_J)$, where we recall that for any $z'= x+ \mathrm{i}y \in \C^{m'}$ we define $V(z'):=  \left(\begin{smallmatrix} x^{\sf T} \\ y^{\sf T} \end{smallmatrix} \right)$. Further denote the real-imaginary de-correlation of $V_J$ by
 \[
  d(V_J) := d(z_J)=\left( \text{det}(V_J V_J^{\sf T}) \right)^{1/2}.
 \]
This parameter, along with the \abbr{LCD} of $z_J/\|z_J\|_2$ controls the L\'{e}vy concentration function of $\sum_{j=1}^m \Xi_j z_j$, for a sequence of independent random variables $\{\Xi_j\}_{j \in[m]}$. 
%
%
Below is the desired estimate on the L\'{e}vy concentration function, which is a direct corollary of \cite[Theorem 7.5]{RV no-gaps}. 
 \begin{prop}  \label{prop: Levy scalar}
  Fix any positive integer $m$ and let $\Xi:=(\Xi_1 \etc \Xi_m) \in \R^m, \ \Xi_j:=\d_j \xi_j, \ j=1 \etc m$, where $\d_1 \etc \d_m$ are i.i.d.~$\dBer(p)$, and $\xi_j$ are i.i.d.~random variables satisfying
  \beq\label{eq:assumption_xi}
    \cL(\xi_j,1) \le 1-c_1 \quad \text{and} \quad \P(|\xi_j| >C_1) \le c_1/2
  \eeq
  for some absolute constants $C_1$ and $c_1\in (0,1)$. Then for any $z \in \C^m$,  $J \in [m]$ such that $z_J \neq 0$,
, and $\vep >0$, we have
  \begin{align}
    \cL(V \Xi,  \vep \sqrt{p} \norm{V_J} ) & \le \frac{C_{\ref{prop: Levy scalar}}}{{d}(V_J)/  \norm{V_J}} \left(\vep + \frac{1}{\sqrt{p} {D}_2(V_J / \norm{V_J})} \right)^2 \label{eq:ineq_D2}
    \intertext{and if $\norm{\Re(z_J)}_2 \ge \norm{\Im(z_J)}_2$,  then}
    \cL(V \Xi,  \vep \sqrt{p} \norm{V_J} ) & \le \ol{C}_{\ref{prop: Levy scalar}} \left(\vep + \frac{1}{\sqrt{p} {D}_{1}(\Re(z_J)) /  \norm{V_J}} \right), \label{eq:ineq_D1}
  \end{align}
  where $V:=V(z)$, $C_{\ref{prop: Levy scalar}}$ and $\ol{C}_{\ref{prop: Levy scalar}}$ are some constants, depending only on $c_1$ and $C_1$.
 \end{prop}

 \begin{rmk}
 We point out to the reader that the definition of \abbr{LCD} in \cite{RV no-gaps} and that of ours are slightly different from each other. For example, to define \abbr{LCD} in \cite{RV no-gaps} the function $\log_+(x):=\max\{\log x,0\}$ was used instead of $\log_1(\cdot)$. Moreover the constants appearing in front of $L$ are different (compare Definition \ref{dfn:two-lcd-dfn} with \cite[Definition 7.1]{RV no-gaps}). However, upon investigating the proof of \cite[Theorem 7.5]{RV no-gaps} it becomes evident that the same proof can be carried out for the \abbr{LCD}s ${D}_2(\cdot)$ and ${D}_1(\cdot)$ to obtain the same estimates on the L\'evy concentration function. It only changes the constant that appears in \cite[Eqn.~(7.3)]{RV no-gaps}. Below we apply this version of \cite[Theorem 7.5]{RV no-gaps}.
\end{rmk}


 \begin{proof}[Proof of Proposition \ref{prop: Levy scalar}]
 As mentioned above the proof is a straightforward application of \cite[Theorem 7.5]{RV no-gaps}. Indeed, we note that $\cL(V\Xi, t) \le \cL(V_J \Xi_J, t)$, for any $t >0$, where $\Xi_J:=(\Xi_j)_{j \in J}$. The assertion \eqref{eq:assumption_xi} implies that
 \beq
  \cL(\Xi_j,1) \le 1-pc_1 \quad \text{and} \quad \P(|\Xi_j| >C_1) \le pc_1/2.
 \eeq
Since $L=(\delta_0p)^{-1/2}$ (see Remark \ref{dfn:delta-0}), shrinking $\delta_0$ if necessary, the inequality \eqref{eq:ineq_D2} follows directly from \cite[Theorem 7.5]{RV no-gaps}, applied with $m=2$. To prove \eqref{eq:ineq_D1}, using the triangle inequality we further note that
$\cL(V_J \Xi_J, t) \le \cL(x_J^{\sf T} \Xi_J, t)$. Thus applying \cite[Theorem 7.5]{RV no-gaps}, with $m=1$ we obtain \eqref{eq:ineq_D1}.
\end{proof}

Applying Proposition \ref{prop: Levy scalar} and standard tensorization techniques we obtain the following result, the proof of which is omitted.

 \begin{prop}  \label{prop: Levy vector}
  Let $B^D$ be as in Proposition \ref{thm: large LCD for real part}.  Fix any $z \in \C^{n}$ and   $J:={\rm{small}}(z)$ such that $z_J \neq 0$. Then for any $\vep>0$, we have
  \begin{equation}
    \cL(B^D z,  \vep \sqrt{p (n-1)}\norm{z_{J}}_2 )  \le \left[ \frac{C_{\ref{prop: Levy vector}}}{d(z_{J}/\norm{z_{J}}_2)} \left(\vep + \frac{1}{\sqrt{p} D_2(z_J/\norm{z_J}_2)} \right)^2 \right]^{n-1},
    \label{eq: Levy complex}
\end{equation}
and  if $\norm{\Re(z_J)}_2 \ge \norm{\Im(z_J)}_2$,  then
\begin{align}
   {\cL(B^D z,  \vep \sqrt{p (n-1)} \norm{z_{J}}_2 )} & {\le \left[ \ol{C}_{\ref{prop: Levy vector}}\left(\vep + \frac{1}{\sqrt{p} D_{1}(\Re(z_J) / \norm{z_{J}}_2)} \right) \right]^{n-1},}
    \label{eq: Levy real}
  \end{align}
  for some constants $C_{\ref{prop: Levy vector}}$, and $\ol{C}_{\ref{prop: Levy vector}}$, depending only on $\E|\xi_{ij}|$ and $\E(\xi_{ij}^4)$.
 \end{prop}

\begin{rmk}
The inequality \eqref{eq: Levy real} provides bounds on L\'{e}vy concentration function based on one-dimensional \abbr{LCD}. It will be used later in Section \ref{sec: net real} to treat essentially real vectors.
\end{rmk}

To prove Proposition \ref{prop: kernel complex} we also need the following elementary lower bound on the \abbr{LCD} of non-dominated vectors. Its proof follows from \cite[Proposition 7.4]{RV no-gaps} and the definition of dominated vectors.
\begin{lem}\label{lem:D_2-lbd}
If $z \notin \rm{Dom}(M,\alpha_\star)$ for some $M <n$ and $\alpha_\star>0$, then we have
\[
D_2(z_{\rm{small}}/\|z_{\rm{small}}\|_2) \ge \f{\alpha_\star \sqrt{M}}{2}.
\]
\end{lem}



 \vskip10pt
 We are now ready to prove Proposition \ref{prop: kernel complex}.

 \begin{proof}[Proof of Proposition \ref{prop: kernel complex}]
  The set in question can be partitioned into the subsets of $\cZ(\gD,\Delta,d)$ appearing in Section \ref{sec: net complex}.
  Indeed, using Lemma \ref{lem:D_2-lbd} we note that for any $z \in \cZ(s)$ we have $D_2(z_{\rm{small}}/\|z_{\rm{small}}\|_2) \ge \f{1}{2}\alpha_\star \sqrt{M} \ge \f{\alpha_\star r_\star}{2}p^{-1/2}$ for some $\alpha_\star >0$. Since $L=(\delta_0 p)^{-1/2}$, choosing $r_\star$ sufficiently large, we therefore obtain that
  \begin{align*}
   \{z \in \cZ(s) : \ D_2(z_{\text{small}}/\norm{z_{\text{small}}}_2) \le  \exp(c' n/M) \}
   \subset \bigcup_{\gD,\Delta, d} \cZ(\gD,\Delta,d) \cap \cZ(s),
  \end{align*}
  where the union is taken over all $\gD=2^k, \ C_\star L \le \gD \le \exp(c' n/M)$, for some large constant $C_\star$, and over all $\Delta=2^m$, $d=2^{-\ell}$ satisfying $d\gD \le \Delta \le \gD$.
Also note that for any $z \in \rm{Compl}(\cZ)$ we have
  \beq\label{eq:complex-dfn}
  d(z_{\text{small}}/\norm{z_{\text{small}}}_2)
   \ge  \frac{4 L}{D_2(z_{\text{small}}/\norm{z_{\text{small}}}_2)} \sqrt{\log_{1} \frac{\Delta(z_{\text{small}}/\norm{z_{\text{small}}}_2)}{{2^7}L}}.
  \eeq
If $z \in \cZ(\gD,\Delta,d)$ we further have that $D_2(z_{\rm{small}}/\|z_{\rm{small}}\|_2) \le 2\gD$,    $d(z_{\rm{small}}/\|z_{\rm{small}}\|_2) \le 2d$, and $\Delta(z_{\rm{small}}/\|z_{\rm{small}}\|_2) \ge \Delta$. Therefore, it follows from \eqref{eq:complex-dfn} that
    \[
   \alpha:=L \sqrt{\log_{1} \frac{\Delta}{2^7 L}} \le {d\gD}.
  \]
  So it allows us to use Proposition \ref{prop: net in Z(D,d)}.
  Recalling that $M \ge r^2_\star p^{-1} \ge p^{-1}$, we see that the number of different values of $\gD$ appearing in the partitions is bounded by $c' pn$. Using the fact that $\a \ge L$, we see that the number of different values of $d$ is bounded by the same number, and so is the number of different values of $\Delta$. Therefore, using the union bound, we deduce that it is enough to show that
   \begin{align}\label{eq:prob-bound-D-Delta-d}
     \P \left(\exists z \in \cZ(D,\Delta,d) \cap \cZ(s):
     \ B^Dz=0, \ \text{and } \norm{B^D} \le (K+R) \sqrt{pn} \right)      \le e^{-2 n},
   \end{align}
  for each such tripple $(\gD,\Delta, d)$. 

 To this end, we note that $\cZ(\gD,\Delta,d) \cap \cZ(s)$ admits a net $\cN(\gD,\Delta,d) \subset \cZ(\gD,\Delta,d) \cap \cZ(s)$. Therefore, 
from Proposition \ref{prop: Levy vector} it follows that
  \[
   \cL(B^D w,  \vep \sqrt{p (n-1)} \norm{w_{\rm{small}}}_2 )  \le \left[ \frac{C_{\ref{prop: Levy vector}}}{d} \left(\vep + \frac{1}{\sqrt{p} \gD} \right)^2 \right]^{n-1},
  \]
 for any $w \in \cN(\gD,\Delta,d)$ and $\vep >0$. Set
  \beq\label{eq:vep_0-dfn}
   \vep_0:=40 C_{\ref{prop: net in Z(D,d)}}(K+R) \frac{\alpha}{\gD}.
  \eeq
  Since $\alpha \ge L=(\delta_0 p)^{-1/2} \ge p^{-1/2}$ and $K, R \ge 1$ we note that $\vep_0 \ge \frac{1}{\sqrt{p}\gD}$. 
  Therefore
  \[
   \P \left( \norm{B^Dw}_2 \le \f{\vep_0}{2}\|w_{\rm{small}}\|_2 \cdot \sqrt{p n} \right)
   \le \left[ \frac{C}{d} \left((K+R) \frac{\alpha}{\gD} \right)^2 \right]^{n-1},
  \]
  for some positive constant $C$. Hence, by the union bound and applying Proposition \ref{prop: net in Z(D,d)}
  \begin{align}\label{eq:union-cN}
   &\P \left(\exists w \in \cN(\gD,\Delta, d): \ \norm{B^Dw}_2 \le \f{\vep_0}{2}\|w_{\rm{small}}\|_2 \cdot \sqrt{p n} \right)\notag\\
    \le\, & |\cN(\gD,\Delta,d)| \cdot  \left[ \frac{C}{d} \left((K+R) \frac{\alpha}{\gD} \right)^2 \right]^{n-1} \notag\\
  \le \, & \left[ \frac{C}{d} \left((K+R) \frac{\alpha}{\gD} \right)^2 \right]^{n-1} \cdot \bar{C}_{\ref{prop: net in Z(D,d)}}^n
    \left(\frac{n}{\rho M}  \cdot  \frac{\gD}{\alpha} \right)^{5 M} \cdot
     \left(  \frac{d \gD^2}{\alpha} \cdot \left( \frac{1}{\sqrt{n}}+\frac{1}{\Delta} \right)  \right)^{n-M}.
  \end{align}
  Recalling the definition of $\alpha$ and using the inequalities $L \le \alpha \le d \gD$ and $\Delta \le \gD$ we note that
  \[
  \f{1}{d} \cdot \left(\f{\alpha}{\gD}\right)^2 \le \f{\alpha}{\gD} = \f{L}{\gD}\sqrt{\log_1 \f{\Delta}{2^7 L}} \le \f{L}{\gD}\sqrt{\log_1 \f{\gD}{2^7 L}} \le \f{1}{2^7}\sqrt{\f{2^7 L}{\gD}},
  \]
  where the last step follows from the fact that $\log x \le x$ for any $x \ge e$. As we have already noted that $\gD \ge C_\star L$, for some large $C_\star$, we can enlarge $C_\star$ further (i.e.~we increase $r_\star$) so that $\f{C}{d} ((K+R) \f{\alpha}{\gD})^2 \bar{C}_{\ref{prop: net in Z(D,d)}} <1$.
  This means that we can drop the term
 \[
  \left[ \frac{C}{d} \left((K+R) \frac{\alpha}{\gD} \right)^2 \right]^{M-1}
  \cdot \bar{C}_{\ref{prop: net in Z(D,d)}}^M <1
 \]
  in  \eqref{eq:union-cN}.
  Therefore, from \eqref{eq:union-cN} we obtain
\begin{align}
   &\P \left(\exists w \in \cN(\gD,\Delta, d): \ \norm{B^Dw}_2 \le \f{\vep_0}{2}\|w_{\rm{small}}\|_2 \cdot \sqrt{p n} \right)  \le \exp(-\Gamma n),
 \end{align}
  where
  \begin{align*}
    \Gamma
    :=
     - \left(1-\frac{M}{n} \right) \cdot \log  \left( \frac{C' (K+R)^2 \alpha}{\sqrt{n}}+\frac{C' (K+R)^2 \alpha}{\Delta} \right) - \frac{5 M}{n} \log   \left(\frac{n}{\rho M}  \cdot \frac{\gD}{\alpha} \right),
  \end{align*}
  and $C':=C \cdot \bar{C}_{\ref{prop: net in Z(D,d)}}$. To finish the proof we need to show that $\Gamma \ge 2$. 

Turning to proof of this, we recall that $\frac{n}{M} \ge \frac{1}{\rho}$, $\a \ge L \ge 1$ and $\gD \le \exp(c' \frac{n}{M})$. So, choosing $c'$ sufficiently small we obtain
  \beq
   \frac{5M}{n} \log   \left(\frac{n}{\rho M}  \cdot \frac{\gD}{\a} \right)
   \le \frac{10 M}{n} \log  \left(\frac{n}{M}  \right) + \frac{5 M}{n} \log   \gD
   \le 20. \notag
  \eeq
Next recall that  $L=(\d_0 p)^{-1/2}$ and $\Delta \le \gD \le \exp(c' n/M)$. So
 \[
  \a = L \sqrt{\log_{1} \frac{\Delta}{2^7 L}} \le \sqrt{\frac{c' n}{\d_0 Mp}}.
 \]
  Therefore, using the fact that $M<n/2$ we obtain
  \begin{align*}
   \Gamma
  \ge -\frac{1}{2} \cdot \log  \left( \frac{C'(K+R)^2 \sqrt{c'}}{\sqrt{\d_0 M p}}+\frac{C' (K+R)^2 \a}{\Delta} \right) -20.
  \end{align*}
  We claim that by choosing $s$ to be a sufficiently large constant, we can guarantee that
  \beq\label{eq:e-50}
    \frac{C' (K+R)^2 \a}{\Delta} <e^{-50}.
  \eeq
  Thus choosing $c'$ small enough and recalling that $Mp \ge r^2_\star$, from claim \eqref{eq:e-50} we see that $\Gamma \ge 2$, providing the required bound for the probability.

Now let us check our claim \eqref{eq:e-50}. Using the definition of $\a$, choosing $s_{\ref{prop: kernel complex}}$ sufficiently large, and using the fact that the function $f(x):=x^{-1} \sqrt{\log_{1}x}$ tends to $0$ as $x \to \infty$, we note that
  \[
e^{50} \frac{C' (K+R)^2 \a}{\Delta}    = e^{50} C' (K+R)^2 \frac{L}{\Delta} \sqrt{\log_{1} \frac{\Delta}{2^7 L}} \le 1,
  \]
for any $\Delta$ such that $\Delta \ge s_{\ref{prop: kernel complex}} L$. This proves the claim \eqref{eq:e-50}.

Thus we have shown that for a sufficiently large value of $s_{\ref{prop: kernel complex}}$,
\begin{align}\label{eq:prob-bound-D-Delta-d-net}
   &\P \left(\exists w \in \cN(\gD,\Delta, d): \ \norm{B^Dw}_2 \le \f{\vep_0}{2}\|w_{\rm{small}}\|_2 \cdot \sqrt{p n} \right)  \le \exp(-2 n).
 \end{align}
 To deduce \eqref{eq:prob-bound-D-Delta-d} from \eqref{eq:prob-bound-D-Delta-d-net} we simply use the property of the net $\cN(\gD,\Delta,d)$. Indeed, let us assume that there exists a $z \in \cZ(s_{\ref{prop: kernel complex}}) \cap \cZ(\gD,\Delta,d)$ so that $B^D z=0$. Denoting $J={\rm{small}}(z)$, using Proposition \ref{prop: net in Z(D,d)}, and the triangle inequality we see that there exists a $w \in \cN(\gD,\Delta,d)$ such that
 \begin{multline*}
 \|B^D w\|_2   = \|B^D(w-z)\|_2 \\
   \le \|B^D\|_2 \cdot\left(\|z_{J^c}-w_{J^c}\|_2 +\norm{w_J}_2\norm{\f{z_J}{\|z_J\|_2}-\f{w_J}{\|w_J\|_2}}_2+ |\norm{w_J}-\norm{z_J}_2| \right)\\
  \le (K+R) \sqrt{np} \cdot \left( 2 C_{\ref{prop: net in Z(D,d)}}\f{\rho \alpha}{\gD} +C_{\ref{prop: net in Z(D,d)}}\f{ \alpha}{\gD} \|w_J\|_2\right)
 \le 3C_{\ref{prop: net in Z(D,d)}}(K+R)\f{\alpha}{\gD}\|w_J\|_2 \sqrt{np},
 \end{multline*}
 where the last inequality follows from the fact that  $w \in \cN(\gD,\Delta,d) \subset \cZ(s_{\ref{prop: kernel complex}}) \subset {\rm{Comp}}(M,\rho)^c$.
 To complete the proof, let us show that $\norm{w_J}_2 \le 4 \norm{w_{\rm{small}}}_2$.
 Assume for a moment that the opposite inequality holds.
  Since $\gD \ge\f{\alpha_\star r_\star}{2}p^{-1/2}$, and $L=(\delta_0 p)^{-1/2}$, choosing a sufficiently large $r_\star$, we may assume that
  \[
   C_{\ref{prop: net in Z(D,d)}} \frac{\alpha}{\gD}
   \le  C_{\ref{prop: net in Z(D,d)}}  \frac{L}{\gD} \sqrt{\log_{1} \frac{\gD}{2^7 L}}
   \le \frac{1}{8}.
  \]
 Denote $\text{small}(w)=I$.
 Combining the estimate above and Proposition \ref{prop: net in Z(D,d)}, we see that
 \begin{align*}
  \norm{z_I}_2
&  \le \norm{w_I}_2+ \norm{z_{I \cap J} -w_{I \cap J} }_2 + \norm{z_{I \setminus J} -w_{I \setminus J} }_2 \\
& \le \frac{1}{4} \norm{w_J}_2 + \norm{z_{J} -w_{J} }_2 + \norm{z_{J^c} -w_{J^c} }_2 \\
  &\le \frac{1}{2} \norm{z_J}_2 + \frac{1}{4} \norm{z_J}_2   + \frac{1}{8} \rho
  < \norm{z_J}_2,
 \end{align*}
 where we used $\norm{z_J}_2 \ge \rho$ in the last inequality. This contradicts the definition of $J$ as $J=\text{small}(z)$ and proves the desired inequality $\norm{w_J}_2 \le 4 \norm{w_{\rm{small}}}_2$.

 Recalling the definition of $\vep_0$ in \eqref{eq:vep_0-dfn}, we now deduce \eqref{eq:prob-bound-D-Delta-d} from \eqref{eq:prob-bound-D-Delta-d-net}. This finishes the proof.
 \end{proof}


Using Proposition \ref{prop: kernel complex} we now complete the proof of Theorem \ref{thm: kernel complex}. The final ingredient for the proof of Theorem \ref{thm: kernel complex} is a lower bound on the one-dimensional \abbr{LCD}, the proof of which follows from  \cite[Lemma 6.2]{V}.

\begin{lem}\label{lem:D_1-lbd}
For any $x \in \R^m$,
$
D_1(x) \ge \f{1}{2\|x\|_\infty}.
$
\end{lem}



 \begin{proof}[Proof of Theorem \ref{thm: kernel complex}]
We first claim that
\beq\label{eq:small-Delta}
 \P \left(\exists z\in {\rm{Ker}}(B^D) \cap S_\C^{n-1} :
     \Delta \left({z_{\rm{small}}}/{\norm{z_{\rm{small}}}_2} \right) \le \frac{\rho^{-1}}{ 4\sqrt{p} }  \right)
    \le \exp(-\bar{c}_{\ref{thm: large LCD for real part}}np).
 \eeq
 The probability bound above would follow from Proposition \ref{thm: large LCD for real part} if for any such $z$, we find a number $\nu \in \C, \  |\nu|=1$ such that the vector $\nu z_{\text{small}}$ has a dominated real part.
 To implement this idea and show \eqref{eq:small-Delta}, we fix $z \in {\rm{Ker}}(B^D) \cap S_\C^{n-1}$ and denote $z_{\text{small}}/\norm{z_{\text{small}}}_2=\phi+{\rm{i}} \psi$, where $\phi, \psi \in \R^J, \ J =\text{small}(z)$.
 Let $\theta=(\theta_1,\theta_2) \in \R^2$ be such that
 \beq\label{eq:w-z-relation}
   \text{dist}( \theta_1 \phi+ \theta_2 \psi, \Z^J) < L \sqrt{\log_1 \frac{\norm{ \theta_1 \phi+ \theta_2 \psi}_2}{2^8 L}}
 \eeq
  and  $\norm{\theta_1 \phi+ \theta_2 \psi}_2 \le 2 \Delta(z_{\text{small}}/\norm{z_{\text{small}}}_2)$.
  Denote
  \[
   w:=\frac{\theta_1 - {\rm{i}} \theta_2}{|\theta_1 - {\rm{i}} \theta_2|} z.
  \]
Then $w \in {\rm{Ker}}(B^D) \cap S_\C^{n-1}$ and $w_{\text{small}}=z_{\text{small}}$.
  Therefore \eqref{eq:w-z-relation} implies that
  \[
   D_1 \left( {\text{Re}(w_{\text{small}})}/{\norm{\text{Re}(w_{\text{small}})}_2} \right)
   \le 2 \Delta \left({z_{\text{small}}}/{\norm{z_{\text{small}}}_2} \right).
  \]
  Upon applying Lemma \ref{lem:D_1-lbd} we find that
  \[
    \norm{ \frac{\text{Re}(w_{\text{small}})}{\norm{\text{Re}(w_{\text{small}})}_2}}_{\infty}
    \ge \left[4 \Delta \left(\frac{z_{\text{small}}}{\norm{z_{\text{small}}}_2} \right) \right]^{-1}
    \ge \rho p^{1/2}.
  \]
  Since $w \in {\rm{Ker}}(B^D) \cap S_\C^{n-1}$ the claim \eqref{eq:small-Delta} follows from  Proposition \ref{thm: large LCD for real part}.

Next, recalling the fact that $L=(\delta_0 p)^{-1/2}$, and shrinking $\rho$, if necessary, we obtain that 
\[
s_{\ref{prop: kernel complex}} L < \f{\rho^{-1}}{4 \sqrt{p}}.
\]
The desired result then follows upon combining Proposition \ref{prop: kernel complex} and \eqref{eq:small-Delta}.
 \end{proof}

 \section{Construction of the net  and the structure of the kernel \\ in the essentially real case}  \label{sec: net real}

 In this section, we consider the class of vectors whose real and imaginary parts are almost linearly dependent.
 Namely, we introduce the set of \emph{essentially real} vectors $\text{Real}(\cZ)$ defined by
  \beq\label{eq:Real-dfn}
   \text{Real}(\cZ)
   :=\cZ \setminus  \text{Compl}(\cZ),
  \eeq
  where we recall the definitions of $\cZ$ and $\text{Compl}(\cZ)$ from \eqref{eq:cZ-dfn} and \eqref{eq:compl-dfn} respectively. Having shown that there does not exist any vector in $\text{Compl}(\cZ) \cap \text{Ker}(B^D)$ such that its two-dimensional \abbr{LCD} is small, it remains to show the same for $\text{Real}(\cZ) \cap \text{Ker}(B^D)$. For essentially real vectors, the real-imaginary de-correlation $d(\cdot)$ is very small which precludes using \eqref{eq: Levy complex}. Instead we have to rely on the probability bound obtained in \eqref{eq: Levy real}, which depends on the one-dimensional \abbr{LCD}.
 As the bound on $D_1(u)$ implies a much more rigid arithmetic structure than a bound on  $D_2(u)$, construction of a net of $\text{real}(\cZ)$ would be easier. To construct such a net we will follow the method of \cite{RV no-gaps}. Before finding a net let us remind the reader that the definition of $\text{Compl}(\cZ)$ and hence that of $\text{Real}(\cZ)$, depends on the two-dimensional \abbr{LCD} (see \eqref{eq:compl-dfn}). Since the bound on L\'{e}vy concentration function, for vectors in $\text{Real}(\cZ)$, depends on the one-dimensional \abbr{LCD}, we need a result that connects $D_1(\cdot)$ with $D_2(\cdot)$. The lemma below does that job and this is the sole reason of introducing different powers of $2$ in the definitions \eqref{dfn:two-lcd-dfn} and \eqref{eq:D_2_define} (recall Remark \ref{rmk:power-of-2}).

   \begin{lem} \label{l: D_2 to D_1}
   Fix $z \in \rm{Real}(\cZ)$ and let $z_{\rm{small}}/\|z_{\rm{small}}\|_2=:\phi+\rm{i} \psi$. Then $D_1(\phi) \le  2 D_2(z_{\rm{small}}/\|z_{\rm{small}}\|_2)$. In particular, if $D_2(z_{\rm{small}}/\|z_{\rm{small}}\|_2)\le \gD$ then  \(
    D_1(\phi) \le 2\gD.
   \)
  \end{lem}


  \begin{proof}
  Let us denote $J={\rm{small}}(z)$. Denoting $\gD=D_2(\phi+{\rm{i}}\psi)$, we see that there exists $\theta=(\theta_1,\theta_2) \in \R^2$ with $\norm{\theta}_2 \le 2\gD$ and $\|\phi \theta_1+ \psi \theta_2\|_2 \ge \Delta(z_J/\|z_J\|_2)/\sqrt{2}$, and $q \in \Z^J$ such that
   \begin{align}\label{eq:D-2-to-D-1}
     \norm{\theta_1 \phi + \theta_2 \psi -q}_2
     < L \sqrt{ \log_{1} \frac{\norm{\theta_1 \phi+\theta_2  \psi}_2}{2^8 L}}.
   \end{align}
Using the triangle inequality, and the facts that $|\theta_2 | \le \norm{\theta}_2$, $\norm{\phi}_2 \cdot \norm{\psi}_2 = d(z_J/\|z_J\|_2)$, and $\norm{\phi}_2 \ge 1/2$, we also obtain
\beq\label{eq:D-2-to-D-1-two}
{\norm{\theta_1 \phi+\theta_2  \psi}_2}
     \le {\norm{\theta_1 \phi}_2+ 4d(z_J/\|z_J\|_2)D_2(z_J/\|z_J\|_2)}.
\eeq
Since $\phi+\mathrm{i} \psi \in \text{Real}(\cZ)$ we further note that
\[
d(z_{J}/\norm{z_{J}}_2) D_2(z_{J}/\norm{z_{J}}_2)
   \le  {4 L} \sqrt{\log_{1} \frac{\Delta(z_{J}/\norm{z_{J}}_2)}{{2^7}L}}
\]
(see \eqref{eq:compl-dfn} and \eqref{eq:Real-dfn}). Therefore denoting
\[
\alpha_0:= L \sqrt{ \log_{1} \frac{\norm{\theta_1 \phi+\theta_2  \psi}_2}{2^{6}\sqrt{2} L}},
\]
from \eqref{eq:D-2-to-D-1}-\eqref{eq:D-2-to-D-1-two} we note that
\begin{equation}\label{eq:D-2-to-D-1-three}
\norm{\theta_1 \phi + \theta_2 \psi -q}_2 < \alpha_0 \le L \sqrt{ \log_{1} \frac{\norm{\theta_1 \phi}_2+16 \a_0}{2^6 \sqrt{2} L}}.
\end{equation}
It is easy to check that
$$s \le \sqrt{\log_{1}(t+s/4\sqrt{2})}, \ s>0 \text{ and } t\ge 0 \Rightarrow s \le \sqrt{\log_{1}(\sqrt{2}t)}.$$
Hence we deduce that
   \[
    \norm{\theta_1 \phi + \theta_2 \psi -q}_2 <   L \sqrt{ \log_{1} \frac{ \norm{\theta_1 \phi}_2}{2^6 L}}.
   \]
As we have already noted $\|\theta_2 \psi\|_2 \le 4 d(z_{J}/\|z_{J}\|_2) D_2(z_J/\|z_J\|_2)$, using the fact $z \in \rm{Real}(\cZ)$, the triangle inequality, and \eqref{eq:D-2-to-D-1-three}, we conclude
   \[
     \norm{\theta_1 \phi  -q}_2
    \le \norm{\theta_1 \phi + \theta_2 \psi -q}_2 +\norm{\theta_2 \psi}_2  \le 17 \alpha_0 < 2^5  L \sqrt{ \log_{1} \frac{\norm{\theta_1 \phi}_2}{2^6 L}}.
   \]
Since $|\theta_1| \le \|\theta\|_2 \le 2\gD$, the proof of the lemma is now complete.
  \end{proof}

\vskip10pt

Next  we find a small net for $\text{Real}(\cZ)$. As in the genuinely complex case, we start with constructing a small net for the set of the small coordinates.
 \begin{lem} \label{l: net in S_J real}
  Fix $J \subset[n]$ and $0< \wt{\a} \le \gD$.
  Define
   \begin{align*}
     S_J(\gD)
     &:=\{ (u,v) \in \R^J \times \R^J: \  \norm{u}_2^2 + \norm{v}_2^2 =1, \ \norm{u}_2 \ge \norm{v}_2,  \\
      &\ d(u,v) \le  \wt{\a}/\gD,   \text{ and } \exists \theta \in [\gD, 3\gD],  \ \text{ such that } \
        \text{dist} \left(\theta u, \Z^{J} \right)< \wt{\a} \}.
   \end{align*}
     Then, there exists a $\left( \frac{C_{\ref{l: net in S_J real}} \wt{\a}}{\gD} \right)$-net $\cM_J(\gD) \subset S_J(\gD)$ with
     \[
      |\cM_J(\gD)|
      \le   \frac{\gD}{\wt{\a}}  \cdot  \left( \bar{C}_{\ref{l: net in S_J real}} \left( \frac{\gD}{\sqrt{|J|}}+1 \right) \right)^{|J|},
     \]
where $C_{\ref{l: net in S_J real}}$ and  $\bar{C}_{\ref{l: net in S_J real}}$ are some absolute constants.
  \end{lem}

 \begin{proof}
 Let $(u,v) \in S_J(\gD)$, and let $\theta \in [\gD, 3\gD]$, $q \in \Z^J$ be such that
 \[
  \norm{\theta u-q}_2< \wt{\a}.
 \]
 Then, using the triangle inequality,
 $
   \norm{q}_2 < \wt{\a}+ |\theta| \le 4\gD,
 $
 and so $q \in \Z_{4\gD}$.
 This implies that
 \beq\label{eq:D-1-net-eq-1}
   \norm{u- \frac{ \gD}{\theta} \cdot \frac{ q}{\gD}}_2 < \frac{\a}{\gD}
   \quad \text{where } \left| \frac{ \gD}{\theta} \right| \le 1, \ \norm{ \frac{ q}{\gD}}_2 \le 4.
 \eeq
 From the definition of real-imaginary de-correlation it also follows that
 \beq\label{eq:D-1-net-eq-2}
   \norm{v}_2 \le  2d(u,v)  \le \frac{2\wt{\a}}{\gD}.
 \eeq
 Let $\cN_1$ be an $(\wt\a/\gD)$-net in $[-1,1]$ with
 $
  |\cN_1| \le 2 \gD / \wt\a.
 $
 Define $\cM_J^1(\gD)$ by
 \[
  \cM_J^1(\gD):= \left \{ \left( x \frac{q}{\gD}, 0 \right): \ q \in \Z_{4\gD}, \ x \in \cN_1 \right \}.
 \]
 Then from \eqref{eq:D-1-net-eq-1}-\eqref{eq:D-1-net-eq-2} we deduce that $\cM_J^1(\gD)$ is a $(7 \wt{\a}/\gD)$-net for $S_J(\gD)$ and $|\cM_J^1(\gD)| = |\Z_{4\gD}| \cdot |\cN_1|$.
This in combination with the bound in \eqref{eq: integer net} yields the required estimate for the cardinality of the net.
To complete the proof, we have to replace the constructed set of vectors by a subset of $S_J(\gD)$.
 This is done in the same way as in Lemma \ref{l: net in S_J}.
  We skip the details.
  \end{proof}

  \vskip10pt

Now we use Lemma \ref{l: net in S_J real} to construct a small net in the set of essentially real vectors with an approximately constant value of the one-dimensional \abbr{LCD}.
 Define  the set $\wt{\cZ}(\gD)$ by
 \begin{multline*}  
  \wt{\cZ}(\gD)
  :=\bigg\{  z \in \text{Real}(\cZ) : \frac{z_{\text{small}}}{\norm{ z_{\text{small}} }_2} = \phi+\mathrm{i} \psi,  \ \norm{\phi}_2 \ge \norm{\psi}_2, D_1(\phi) \in [\gD, 2\gD],\\
   \ d(\phi,\psi) \le \wt{\a}/\gD \bigg\},
 \end{multline*}
 where
  \beq\label{eq:wta}
\wt{\alpha}:= 2^5 L \sqrt{{\log_1} \f{\gD}{2^5 L}},
  \eeq
The set $\wt{\cZ}(\gD)$ is the collection of vectors in $\text{Real}(\cZ)$ for which $D_1(z_{\rm{small}}/\|z_{\rm{small}}\|_2) \approx \gD$. The condition $d(\phi,\psi) \le \wt\a/\gD$ ensures that the real-imaginary de-correlation is small.

 \begin{prop} \label{prop: net in Z(D)}
Fix $\gD>1$. Let $\wt\a$ be as in \eqref{eq:wta} and assume $0< \wt{\a}\le \gD$. Then there exist absolute constants $C_{\ref{prop: net in Z(D)}}$, $\bar{C}_{\ref{prop: net in Z(D)}}$, and a set $\wt{\cN}(\gD) \subset \wt{\cZ}(\gD)$ with
  \[
   |\wt{\cN}(\gD)|
   \le \bar{C}_{\ref{prop: net in Z(D)}}^n \left(\frac{n}{\rho M}  \cdot  \frac{\gD}{\wt{\alpha}} \right)^{4 M} \cdot\left( \frac{\gD}{\sqrt{n}}+1 \right)^{n-M}
  \]
  having the following approximation property: Let $z \in \wt{\cZ}(D)$ be any vector and denote $J=\mathrm{small}(z)$. Then there exists $w \in \wt{\cN}(D)$ such that
     \begin{multline*}
      \norm{\frac{z_J}{\norm{z_J}_2}-\f{w_J}{\norm{w_J}_2}}_2 < C_{\ref{prop: net in Z(D)}}\frac{\wt{\alpha}}{\gD}, \quad \norm{z_{J^c} - w_{J^c}}_2 \le C_{\ref{prop: net in Z(D)}} \frac{\rho \wt{\alpha}}{\gD}, \\
       |\|z_{J}\|_2 - \|w_J\|_2| \le C_{\ref{prop: net in Z(D)}} \frac{\rho \wt{\alpha}}{\gD}.
      \end{multline*}
 \end{prop}


 Proposition \ref{prop: net in Z(D)} is derived from Lemma \ref{l: net in S_J real} in the same way as Proposition \ref{prop: net in Z(D,d)} was derived from Lemma \ref{l: net in S_J}.
 We omit the details.


 Now, we are ready to prove the main result of this section which shows that with high probability, there are no essentially real vectors with a subexponential \abbr{LCD} in the kernel of $B^D$.

 \begin{prop} \label{prop: kernel real}
Let $B^D, {A}_n, \rho, K,R,r$, and $r'$ be as in Proposition \ref{thm: large LCD for real part}. Then there exists a positive constant $c'_{\ref{prop: kernel real}}$, depending only on $K,R$, and the fourth moment of $\{\xi_{ij}\}$, such that
     \begin{multline*}
     \P \Big(\exists z \in \mathrm{Real}(\cZ) \cap {\rm{Ker}}(B^D):
     D_2(z_{\rm{small}}/\norm{z_{\rm{small}}}_2) \le \exp(c'_{\ref{prop: kernel real}} n/M)
   \\
    \text{and } \norm{{A}_n} \le K \sqrt{pn} \Big)
     \le e^{- n},
   \end{multline*}
where   \(
 M= {{C}_{\ref{thm: large LCD for real part}} 
 \rho^{-4} p^{-1}}.
\)
 \end{prop}

 \vskip10pt

 \begin{proof}
 The proof of this proposition is very similar to that of Proposition \ref{prop: kernel complex}. First we note that using Lemma \ref{l: D_2 to D_1} it follows that it is enough to show that, with high probability, there does not exist $z \in \text{Ker}(B^D) \cap \text{Real}(\cZ)$ such that $D_1(\phi(z)) \le \exp(c'n/M)$ for some small constant $c'$, where $z_{\text{small}}/\|z_{\text{small}}\|_2=:\phi(z) +{\rm{i}} \psi(z)$ with $\|\phi(z)\|_2 \ge \|\psi(z)\|_2$. We then claim that the subset of $\text{Real}(\cZ)$ in context can be partitioned into the sets $\wt{\cZ}(\gD)$ as follows:
  \beq\label{eq:real-partition}
   \{ z \in \text{Real}(\cZ): \
     D_1(\phi(z)) \le \exp(c' n/M) \}
   \subset \bigcup_{\gD} \wt{\cZ}(\gD),
  \eeq
  where the union is taken over all $\gD=2^k, \ \gD \le \exp(c' n/M)$. Note that the claim in \eqref{eq:real-partition} is obvious if we drop the requirement $d(z_{\rm{small}}/\|z_{\rm{small}}\|_2)=d(\phi(z),\psi(z)) \le \wt{\a}/\gD$ from the definition of $\wt{\cZ}(\gD)$. We show that the required condition on the real-imaginary de-correlation is automatically satisfied for all $z \in {\rm{Real}}(\cZ)$. Indeed, recalling the definition of $\text{Real}(\cZ)$, and the fact that
$$ \Delta(z_{\rm{small}}/\|z_{\rm{small}}\|_2) \le D_2(z_{\rm{small}}/\|z_{\rm{small}}\|_2)$$
we see that for any $z \in \text{Real}(\cZ)$,
 \begin{multline}\label{eq:d-D-1-bd}
 d(z_{\rm{small}}/\|z_{\rm{small}}\|_2) 
   \le \f{4L}{D_2(z_{\rm{small}}/\|z_{\rm{small}}\|_2)} \sqrt{\log_1 \f{ D_2(z_{\rm{small}}/\|z_{\rm{small}}\|_2)}{2^7 L}}\\ \le \f{8L}{D_1(\phi(z))} \sqrt{\log_1 \f{ D_1(\phi(z))}{2^8 L}},
 \end{multline}
 where the last inequality is obtained upon noting that $x \sqrt{\log_1(1/x)}$ is an increasing function for $x \in (0,e^{-1}]$ together with an application of Lemma \ref{l: D_2 to D_1}. If $z \in \text{Real}(\cZ)$ such that $D_1(\phi(z)) \in [\gD,2\gD]$ then recalling the definition of $\wt{\a}$, from \eqref{eq:d-D-1-bd} we see that
\[
d(\phi(z),\psi(z))=d(z_{\rm{small}}/\|z_{\rm{small}}\|_2) \le \wt{\a}/\gD,
\]
which in turn proves the claim \eqref{eq:real-partition}. We further claim that the lower bound on $\gD$ in \eqref{eq:real-partition} can be improved to
\[
\gD_0:= C_0 \rho^{-1} p^{-1/2},
\]
where $C_0:= \sqrt{{C}_{\ref{thm: large LCD for real part}}}/2$.
To see this we note that $\text{Real}(\cZ) \subset \text{Incomp}(M, \rho)$. Therefore for any $z \in \text{Real}(\cZ)$ we have
\[
\norm{\phi(z)}_\infty \le 2\f{\norm{z_{\rm{small}}}_\infty}{\norm{z_{\rm{small}}}_2} \le \frac{2}{\rho \sqrt{M}} = \f{2 \sqrt{p}}{\rho^{-1}\sqrt{{C}_{\ref{thm: large LCD for real part}}}},
\]
where the last step follows from our choice of $M$. Hence, using Lemma \ref{lem:D_1-lbd} we see that for any $z \in \text{Real}(\cZ)$ we must have $D_1(\phi(z)) \ge \gD_0$. This establishes that  the union in the \abbr{RHS} of \eqref{eq:real-partition} can be taken over all $\gD=2^k, \ \gD_0 \le \gD \le \exp(c' n/M)$. So using the union bound, we deduce that it is enough to show that
   \begin{align*}
     &\P \left(\exists z \in \wt{\cZ}(\gD):
     \ B^Dz=0, \ \text{and } \norm{B^D} \le (K+R) \sqrt{pn} \right) \le e^{-2 n}.
   \end{align*}
  for each such $\gD$.

To this end, using Proposition \ref{prop: Levy vector} we see that for any $w \in \wt{\cN}(\gD)$ we have
  \[
   \cL(B^D w,  \vep \sqrt{p (n-1)}\|w_{{\rm{small}}}\|_2 )  \le \left[ \ol{C}_{\ref{prop: Levy vector}} \left(\vep + \frac{1}{\sqrt{p} \gD} \right) \right]^{n-1}.
  \]
Now set
  \[
   \wt{\vep}_0:= 40 C_{\ref{prop: net in Z(D)}}(K+R) \frac{\wt{\alpha}}{\gD}.
  \]
Since the fact $\wt{\alpha} \ge L = (\delta_0p)^{-1/2}$ implies that $\wt{\vep}_0 \ge \frac{1}{\sqrt{p}\gD}$,   we obtain that for any $w \in \wt{\cN}(\gD)$,
  \begin{multline*}
  \P \left( \norm{B^Dw}_2 \le \f{\vep_0}{2}\|w_{\rm{small}}\|_2 \cdot \sqrt{p n} \right)
   \le \left[ \ol{C}_{\ref{prop: Levy vector}} \left(\wt{\vep}_0 + \frac{1}{\sqrt{p} \gD} \right) \right]^{n-1} \\
   \le \left( \wt{C} (K+R) \frac{\wt{\alpha}}{\gD} \right)^{n-1},
  \end{multline*}
for some constant $\wt{C}$. Hence, by the union bound and applying Proposition \ref{prop: net in Z(D)} we obtain
  \begin{multline*}
   \P \left(\exists w \in \wt{\cN}(\gD): \   \norm{B^Dw}_2 \le \f{\vep_0}{2}\|w_{\rm{small}}\|_2 \cdot \sqrt{p n} \right)\\
   \le |\wt{\cN}(\gD)| \cdot \left( \wt{C} (K+R) \frac{\wt{\alpha}}{\gD} \right)^{n-1}\\
   \le \left(C'(K+R) \frac{\wt{\alpha}}{\gD} \right)^{n-1}
      \left(\frac{n}{\rho M} \cdot  \frac{\gD}{\wt{\alpha}} \right)^{4 M} \left( \frac{\gD}{\sqrt{n}}+1 \right)^{n-M},  \end{multline*}
where $C'$ is some large constant. Next recalling the definitions of $\wt{\alpha}$ and $\gD_0$, using the facts that $\gD \ge \gD_0$, $L=(\delta_0 p)^{-1/2}$ and the function $f(x):= x \sqrt{\log_1(1/x)}$ is increasing for $x \in (0,e^{-1})$ we find that
  \beq\label{eq:wta/D-ubd}
  \f{\wt{\alpha}}{\gD} = \f{2^5 L}{\gD} \sqrt{{\log_1} \f{\gD}{2^5 L}} \le \f{2^5 L}{\gD_0} \sqrt{{\log_1} \f{\gD_0}{2^5 L}}=f\left(\f{2^5} {C_0\rho^{-1} \delta_0^{1/2}}\right).
  \eeq
Recalling the definition of $C_0$ and enlarging ${C}_{\ref{thm: large LCD for real part}}$ we therefore note from above that we can assume $\wt{C}(K+R)\wt{\alpha}/\gD <1$. This yields
\[
  \P \left(\exists w \in \wt{\cN}(\gD): \   \norm{B^Dw}_2 \le \f{\vep_0}{2}\|w_{\rm{small}}\|_2 \cdot \sqrt{p n} \right) \le \exp(-\wt{\Gamma}n),
\]
where
  \[
    \wt{\Gamma}
     :=
     - \left(1-\frac{M}{n} \right) \cdot \log  \left( \frac{C'(K+R) \wt{\alpha}}{\sqrt{n}}+\frac{C'(K+R) \wt{\alpha}}{\gD} \right)
     - \frac{4M}{n} \log   \left(\frac{n}{\rho M}  \cdot \frac{\gD}{\wt{\alpha}} \right).
  \]
We next show that $\wt{\Gamma} \ge 2$ which allows us to deduce that
\beq\label{eq:prob-bd-real-net}
  \P \left(\exists w \in \wt{\cN}(\gD): \   \norm{B^Dw}_2 \le \f{\vep_0}{2}\|w_{\rm{small}}\|_2 \cdot \sqrt{p n} \right) \le \exp(-2n).
\eeq
To prove that $\wt{\Gamma} \ge 2$, we recall that $ \frac{n}{M} \ge \frac{1}{\rho}$ and $L \le \wt{\a} \le \gD \le \exp(c' \frac{n}{M})$. Therefore
   \[
     \frac{4M}{n} \log   \left(\frac{n}{\rho M}  \cdot \frac{\gD}{\wt\a} \right) \le 10,
   \]
upon choosing $c'$ sufficiently small. Using the fact $M \le n/2$, this yields
   \[
    \wt{\Gamma}
     \ge
     - \frac{1}{2} \cdot \log  \left( \frac{C'(K+R) \wt{\a}}{\sqrt{n}}+\frac{C'(K+R) \wt\a}{\gD} \right)-10.
   \]
Recalling \eqref{eq:wta/D-ubd} we see that we may enlarge ${C}_{\ref{thm: large LCD for real part}}$ (and thus, the minimal value of $\gD$) further so that $\wt{C}(K+R)\wt{\alpha}/\gD <e^{-30}.$
Using the upper bound for $\gD$, we also note that
   \[
    \frac{\wt{\a}}{\sqrt{n}}
    \le \frac{2^5L \sqrt{\log_1 \frac{\gD}{2^5 L}}}{\sqrt{n}}
    \le \frac{2^5 L \sqrt{c' \frac{n}{M}}}{\sqrt{n}}
    =\frac{2^5 \sqrt{c'}}{\sqrt{\d_0 Mp}}
    \le \frac{2^5 \rho^2\sqrt{c'}}{\sqrt{{C}_{\ref{thm: large LCD for real part}} \d_0 }} \le e^{-30},
   \]
   where the second last inequality follows from our choice of $M$, and the last inequality results from enlarging ${C}_{\ref{thm: large LCD for real part}}$ once more. This completes the proof of the claim that $\wt{\Gamma} \ge 2$. Thus we have shown that \eqref{eq:prob-bd-real-net} holds.
The rest of the proof relies on the approximation of a general point of $\wt{\cZ}(\gD)$ by a point of the set $\wt{\cN}(\gD)$, and
is exactly the same as that of Proposition \ref{prop: kernel complex}. We leave the details to the reader. This completes the proof.
 \end{proof}

 \section{Proof of Theorem \ref{thm: smallest singular}} \label{sec: singular value}
 In this section our goal is to combine the results of previous sections and finish the proof of Theorem \ref{thm: smallest singular}. First let us state the following general result from which Theorem \ref{thm: smallest singular} follows.

 \begin{thm}  \label{thm: smallest singular + norm}
 Let ${A}_n$ be an $n \times n$ matrix with i.i.d.~entries $a_{i,j}= \delta_{i,j} \xi_{i,j}$, where $\{\delta_{i,j}\}$ are independent Bernoulli random variables taking value 1 with probability $p_n \in (0,1]$, and $\{\xi_{i,j}\}$ are i.i.d.~centered real-valued random variables with unit variance and finite fourth moment. Fix $K,R \ge 1$, and $r \in (0,1]$ and let $\Omega_K:= \{\|{{A}_n}\| \le K \sqrt{n p_n}\}$. Assume that $D_n$ is a diagonal matrix such that $ \norm{D_n} \le R \sqrt{np_n}$ and $ {\mathrm{Im}}(D_n)= r'\sqrt{np_n} I_n$ with $|r'| \in [r,1]$. Then there exists constants $0< c_{\ref{thm: smallest singular + norm}}, {\bar{c}_{\ref{thm: smallest singular + norm}}}, c'_{\ref{thm: smallest singular + norm}}, C_{\ref{thm: smallest singular + norm}}, C'_{\ref{thm: smallest singular + norm}}, \ol{C}_{\ref{thm: smallest singular + norm}} < \infty$, depending only on $K, R, r$, and the fourth moment of $\{\xi_{i,j}\}$, such that for any $\vep>0$ we have the following:

 \begin{enumerate}[(i)]

\item If
 \beq
p_n \ge  \frac{\ol{C}_{\ref{thm: smallest singular + norm}} \log n}{n}  ,\notag
 \eeq
 then
 \begin{multline*}
  \P \bigg( \Big\{s_{\min}({A}_n + D_n) \le c_{\ref{thm: smallest singular + norm}} \vep \exp \left(-C_{\ref{thm: smallest singular + norm}} \frac{\log (1/p_n)}{\log (np_n)} \right) \sqrt{\frac{p_n}{n}} \Big\} \bigcap   \Omega_K \bigg)\\
  \le \vep +  \f{C'_{\ref{thm: smallest singular + norm}}}{\sqrt{np_n}}. 
 \end{multline*}

 \item Additionally, if
 \beq\label{eq:p_n-assumption-ss+norm}
 \log(1/p_n) < {\bar{c}_{\ref{thm: smallest singular + norm}}} (\log np_n)^2,
 \eeq
 then
  \begin{multline*}
  \P \bigg( \Big\{s_{\min}({A}_n + D_n) \le c_{\ref{thm: smallest singular + norm}} \vep \exp \left(-C_{\ref{thm: smallest singular + norm}} \frac{\log (1/p_n)}{\log (np_n)} \right) \sqrt{\frac{p_n}{n}} \Big\} \bigcap   \Omega_K \bigg)\\
  \le \vep +  \exp(-c'_{\ref{thm: smallest singular + norm}}\sqrt{np_n}).
 \end{multline*}
  \end{enumerate}
\end{thm}

\vskip10pt
The proof of part (i) of Theorem \ref{thm: smallest singular + norm} follows from Berry-Ess\'{e}en theorem and Proposition \ref{p: dominated and compressible}. The proof of part (ii) uses results from Section \ref{sec: kernel complex} and Section \ref{sec: net real}. Recall that in Section \ref{sec: kernel complex} and Section \ref{sec: net real} we have shown that there does not exist any vector in $\text{Ker}(B^D)$ with a sub-exponential two-dimensional \abbr{LCD}, with high probability.
To prove the second part of Theorem \ref{thm: smallest singular + norm}, we  use \abbr{LCD} based bounds on L\'{e}vy concentration function. At this moment, we know that with high probability, any vector in  in $\text{Ker}(B^D)$ has an exponential two-dimensional \abbr{LCD}. However,  we do not have any control the real-imaginary de-correlation of this vector. This means that we cannot use the bound \eqref{eq: Levy complex}, and  have to rely on \eqref{eq: Levy real}.
 
 To apply \eqref{eq: Levy real}, we therefore need to show that any vector with a large two-dimensional \abbr{LCD} must also admit a large value of one-dimensional \abbr{LCD}. This calls for another modification to the definition of the one-dimensional \abbr{LCD}.

\begin{dfn}\label{dfn:one-lcd-modify-2}
For a non-zero vector $x \in \R^m$, we set
 \beq
    \wh{D}_1(x):= \inf \left\{ \theta>0: \  \text{dist}(\theta x, \Z^{m}) < L \sqrt{\log_1 \frac{\norm{\theta x}_2}{2^8 L}} \right\}. \notag
 \eeq
 \end{dfn}

 \vskip10pt
 The advantage of working with this one-dimensional \abbr{LCD} $\wh{D}_1(\cdot)$ can be seen from the following result.

 \begin{lem}\label{lem:D_1-wh-D_2}
 For $z :=x + {\rm{i}} y \in \C^m$ we have
 $
 \wh{D}_1(x) \ge D_2(z).
 $
 \end{lem}

 \begin{proof}
 The proof follows by simply noting that if there exists a $\theta' >0$ such that
 \[
 \text{dist}(\theta' x, \Z^{m}) < L \sqrt{\log_1 \frac{\norm{\theta' x}_2}{2^8 L}},
 \]
 then for $\theta=(\theta', 0)$ we also have that
 \[
 \text{dist}( V^{\sf T} \theta, \Z^m) < L \sqrt{\log_1 \frac{\norm{V^{\sf T} \theta}_2}{{2^8} L}}.
 \]
 \end{proof}




Now we are ready to prove Theorem \ref{thm: smallest singular + norm}.

\vskip10pt

\begin{proof}[Proof of Theorem \ref{thm: smallest singular + norm}]
 The proof is similar to that of \cite[Theorem 1.1]{BR}. We include it for completeness. Note that for any $\vartheta>0$,
 \begin{align}\label{eq:inf_split}
 & \P\Big( \{s_{\min}({A}_n+D_n) \le \vartheta\}  \cap  \Omega_K \Big) \notag\\
  \le& \,  \P\Big( \Big\{\inf_{x \in V^c} \norm{({A}_n+D_n)x}_2 \le \vartheta \Big\}  \cap\Omega_K \Big)
  + \P\Big( \Big\{\inf_{x \in V} \norm{({A}_n+D_n)x}_2 \le \vartheta  \Big\} \cap  \Omega_K \Big),
 \end{align}
 where
  \[
  V:=S_\C^{n-1} \setminus \Big( \text{Comp}(c_{\ref{p: dominated and compressible}}n, \rho) \cup \text{Dom}(c_{\ref{p: dominated and compressible}}n, (C_{\ref{p: dominated and compressible}}(K+R))^{-4}) \Big),
 \]
 and $\rho$ as in Proposition \ref{p: dominated and compressible}. Using Proposition \ref{p: dominated and compressible} with $M=c_{\ref{p: dominated and compressible}}n$,  we obtain that
 \[
   \P\Big( \inf_{x \in V^c} \norm{({A}_n+D_n)x}_2 \le {c}'_{\ref{p: dominated and compressible}}(K+R) \rho \sqrt{np},  \,  \norm{{A}_n} \le K \sqrt{pn} \Big)
    \le \exp(-\ol{c}_{\ref{p: dominated and compressible}}np).
 \]
Therefore it only remains to find an upper bound on the second term in the \abbr{RHS} of \eqref{eq:inf_split}. Applying Lemma \ref{l: via distance} we see that to find an upper bound of
\[
 \P\Big( \Big\{\inf_{x \in V} \norm{({A}_n+D_n)x}_2 \le \vep \rho^2 \sqrt{\frac{p}{n}}\Big\} \cap \Omega_K\Big)
 \]
 is enough to find the same for
 \[\P \Big( \Big\{\dist({A}_{n,j},H_{n,j}) \le \rho \sqrt{p} \vep\Big\} \cap \Omega_K \Big) \text{ for a fixed } j,\]
 where ${A}_{n,j}$ are columns of $({A}_n+D_n)$. As these estimates are the same for different $j$'s we only need to consider the case  $j=1$. Recall that $B^D$ is the matrix whose rows are the columns ${A}_{n,2} \etc {A}_{n,n}$.
Therefore
 \[
    \dist({A}_{n,1},H_{n,1}) \ge |\pr{v}{{A}_{n,1}}|,
 \]
for any $v \in S_\C^{n-1} \cap \text{Ker}(B^D)$. Thus it is enough to find an upper bound on
 \beq\label{eq:normal-inner-product}
   \P\Big(\Big\{\exists v \in \cZ\cap {\rm{Ker}}(B^D):  |\pr{{A}_{n,1}}{v}| \le \rho\vep \sqrt{p}\Big\}\cap  \Omega_K \Big).
 \eeq
First we obtain a bound on \eqref{eq:normal-inner-product} under the assumption of part (i). This follows from a simple Berry-Ess\'{e}en bound.

Since $v \in S_\C^{n-1}\cap {\rm{Ker}}(B^D)$ using Proposition \ref{p: dominated and compressible} again, we may assume that $v \notin \text{Comp}(c_{\ref{p: dominated and compressible}}n, \rho) \cup \text{Dom}(c_{\ref{p: dominated and compressible}}n, (C_{\ref{p: dominated and compressible}}(K+R))^{-4})$. 
Let $J=\supp(  v_{[c_{\ref{p: dominated and compressible}}n+1,n]})$. Then
\[
 \P \left(|\pr{{A}_{n,1}}{v}| \le \rho\vep \sqrt{p} \right)
\le \cL \left( \sum_{i\in J}v_i \d_i \xi_i, \rho \sqrt{p}\vep \right).
\]
Since $v \notin \text{Comp}(c_{\ref{p: dominated and compressible}}n, \rho) \cup \text{Dom}(c_{\ref{p: dominated and compressible}}n, (C_{\ref{p: dominated and compressible}}(K+R))^{-4})$ we have
\[
\|v_{J}\|_\infty \le \f{C_{\ref{p: dominated and compressible}}(K+R))^{4}}{\sqrt{c_{\ref{p: dominated and compressible}}n}}\|v_{J}\|_2 \quad \text{ and } \quad \|v_{J}\|_2 \ge \rho.
\]
The Berry--Ess\'{e}en Theorem (see \cite[Theorem 2.2.17]{St}) then yields that
\begin{equation}\label{eq:B-E-bd}
 \cL \left( \sum_{i\in J} v_i \d_i \xi_i, \rho \sqrt{p}\vep \right)
 \le C \vep+ C' \frac{p \norm{v_J}_3^3}{p^{3/2} \norm{v_J}_2^3} \le  C \vep+ C' \frac{ \norm{v_J}_{\infty}}{p^{1/2} \norm{v_J}_2} \le  C \vep+ \frac{C''}{\sqrt{pn}},
\end{equation}
where $C$ is an absolute constant, the constant $C'$ depends only on the fourth moment of $\{\xi_{i,j}\}$, and $$C''= \f{C_{\ref{p: dominated and compressible}}(K+R))^{4}}{\sqrt{c_{\ref{p: dominated and compressible}}}}\cdot C'.$$
Replacing $\vep$ by $\vep/C$ finishes the proof of part (i) of the theorem.

It  remains to prove part (ii). As seen above, we only need to obtain a bound on \eqref{eq:normal-inner-product} under the stronger assumption of $p_n$ of part (ii). To this end, we apply Proposition \ref{p: dominated and compressible} again. Setting $M_0={{C}_{\ref{thm: large LCD for real part}}  \rho^{-4} p^{-1}}$ from Proposition \ref{p: dominated and compressible} we find that it is enough to bound
\beq\label{eq:normal-inner-product-1}
   \P\Big(\Big\{\exists v \in V_0\cap {\rm{Ker}}(B^D):  |\pr{{A}_{n,1}}{v}| \le \rho\vep \sqrt{p}\Big\}\cap  \Omega_K \Big),
 \eeq
 where
 \[
 V_0:=S_\C^{n-1} \setminus \Big( \text{Comp}(M_0, \rho) \cup \text{Dom}(M_0, (C_{\ref{p: dominated and compressible}}(K+R))^{-4}) \Big).
 \]
 Further denote
 \[
  V_1:=\Big\{w \in V_0: D_{2}\left({w_{\text{small}}}/{\|w_{\text{small}}\|_2}\right) \le \exp (c' n/M_0) \Big\} \quad \text{and} \quad V_2:=V_0\setminus V_1,
 \]
 where $c':=\min\{c'_{\ref{thm: kernel complex}}, c'_{\ref{prop: kernel real}}\}$. We will show that
 \begin{align}\label{eq:normal-vector-not-small-lcd}
   \P\Big(\Big\{\exists v \in V_1\cap {\rm{Ker}}(B^D)\Big\}\cap  \Omega_K \Big) \le  \exp(-\bar{c}np),
 \end{align}
 for some $\bar{c} >0$. Since ${\rm{Ker}}(B^D)$ is invariant under rotation, recalling the definition of the set $\cZ$ {(see \eqref{eq:cZ-dfn})}, we see that it is enough to show that
\begin{multline*}
\P\Big(\Big\{\exists v \in \cZ \cap {\rm{Ker}}(B^D): D_{2}\left({v_{\text{small}}}/{\norm{v_{\text{small}}}_2}\right) \le \exp (c' n/M_0)\Big\}\cap  \Omega_K \Big) \\
\le \exp(-\bar{c}np).
\end{multline*}
Note that, if $p$ satisfies \eqref{eq:p_n-assumption-ss+norm} with a sufficiently small $\bar{c}_{\ref{thm: smallest singular + norm}}$, then it also satisfies the assumption \eqref{eq:p-rho-gc}. So we can apply Theorem \ref{thm: kernel complex}. Applying Theorem \ref{thm: kernel complex} and Proposition \ref{prop: kernel real} we then immediately obtain our claim \eqref{eq:normal-vector-not-small-lcd}. Therefore it only remains to find an upper bound on
\begin{multline}\label{eq:normal-inner-product-2}
  \P\Big(\Big\{\exists v \in \cZ \cap {\rm{Ker}}(B^D): D_{2}\left({v_{\text{small}}}/{\norm{v_{\text{small}}}_2}\right) > \exp (c' n/M_0) \text{ and } \\
   |\pr{{A}_{n,1}}{v}| \le \rho\vep \sqrt{p}\Big\}\cap  \Omega_K \Big).
 \end{multline}
To obtain the desired bound we condition on $B^D$ which fixes the vector $v$ for which
\[
D_{2}\left({v_{\text{small}}}/{\norm{v_{\text{small}}}_2}\right) > \exp (c' n/M_0).
\]
Lemma \ref{lem:D_1-wh-D_2} implies that
\[
\wh{D}_{1}\left(\phi(v)\right) > \exp (c' n/M_0),
\]
where we recall that $v_{\rm{small}}/\|v_{\rm{small}}\|_2=\phi(v) + {\rm{i}} \psi(v)$.
Inequality \eqref{eq: Levy real} holds with $\wh{D}_1(\cdot)$ instead of $D_1(\cdot)$ if the constant $\ol{C}_{\ref{prop: Levy vector}}$ is appropriately adjusted.
Recalling the definition of $M_0$ we deduce that
 \begin{align*}
 \P ( |\pr{{A}_{n,1}}{v}| \le \vep \rho \sqrt{p})
  &\le\bar{C} \left( \vep+\frac{1}{\sqrt{p} \wh{D}_1(\phi(v))} \right)   \le  \bar{C} \left( \vep+\frac{1}{\sqrt{p} } \exp (-c''np \rho^4) \right),
 \end{align*}
 for some constants $\bar{C}$ and $c''$. Choosing ${\bar{c}_{\ref{thm: smallest singular + norm}}}$ sufficiently small and recalling the definition of $\rho$ we further deduce that
 \[
 \frac{1}{\sqrt{p} } \exp (-c''np \rho^4) \le \exp(-c''\sqrt{np}).
 \]
 Therefore replacing $\vep$ by $\vep/\bar{C}$ we conclude that \eqref{eq:normal-inner-product-2} is bounded by
 \[
 \vep+ \bar{C}\exp(-c''\sqrt{np}).
 \]
 This completes the proof of the theorem.
\end{proof}


\begin{proof}[Proof of Theorem \ref{thm: smallest singular}]
The proof immediately follows from Theorem \ref{thm: smallest singular + norm}, \cite[Theorem 1.7]{BR}, and the triangle inequality.
\end{proof}

\begin{rmk}\label{rmk:p-ass-explain}
{From the proof of Theorem \ref{thm: smallest singular + norm} we note that the assumption \eqref{p:assumption-as} (equivalently \eqref{eq:p_n-assumption-ss+norm}) was needed to show that the assumption \eqref{eq:p-rho-gc} holds. From \cite[Proposition 3.1]{BR} we have $\rho=\exp(-C\log (1/p)/\log(np))$, for some large $C$. If one can improve the conclusion of \cite[Proposition 3.1]{BR} to accommodate $\rho=\Omega(1)$ then it is obvious that \eqref{eq:p-rho-gc} holds without the assumption \eqref{p:assumption-as}, and therefore Theorem \ref{thm:sparse_general}(ii) can be extended without any extra assumption.}
\end{rmk}

\section{Intermediate Singular Values}\label{sec: intermediate singular values}
The goal of this short section is to prove Theorem \ref{thm:intermed-sing} which shows that there are not too many singular values of the matrix $\f{1}{\sqrt{np}} A_n - w I_n$ near zero. In proving Theorem \ref{thm:intermed-sing}, we employ the same strategy as in \cite{bcc-generator, tao_vu, wood}. Namely, we first show that the distance of any row of $A_n$ from any given subspace of not very large dimension cannot be too small with  large probability.


\begin{lem}\label{lem:dist-conc-bd}
Let ${\bm a}:= (\xi_{i}\delta_{i})_{i=1}^n$ be an $n$-dimensional vector where $\{\xi_i\}_{i=1}^n$ are i.i.d.~with zero mean and unit variance and $\{\delta\}_{i=1}^n$ are i.i.d.~$\dBer(p)$. Let $\psi: \N \mapsto \N$ be such that $\psi(n) \ra \infty$ and $\psi(n) < n$. Then there exists a positive finite constant $c_{\ref{lem:dist-conc-bd}}$\footnote{the constant $c_{\ref{lem:dist-conc-bd}}$ may depend on the tail of the distribution of $\{\xi_i\}_{i=1}^n$} such that  for every sub-space $H$ of $\C^n$ with $1 \le \dim(H) \le n- \psi(n)$, we have
\[
\P\left( \dist({\bm a}, H) \le c_{\ref{lem:dist-conc-bd}} \sqrt{p(n- \dim(H))}\right) \le \exp(-c_{\ref{lem:dist-conc-bd}}p \psi(n)) + \exp(-c_{\ref{lem:dist-conc-bd}} \psi^2(n)/n).
\]
\end{lem}

A result similar to Lemma \ref{lem:dist-conc-bd} was obtained in \cite{tao_vu} (see Proposition 5.1 there) for the dense case. Later in \cite{bcc-generator} (and \cite{wood}) it was improved for the sparse case. Our Lemma \ref{lem:dist-conc-bd} follows from \cite[Lemma 3.5]{bcc-generator} when applied to the set-up of this paper. So we omit the proof and refer the reader to the proof of \cite[Lemma 3.5]{bcc-generator}.

We now complete the proof of Theorem \ref{thm:intermed-sing} using Lemma \ref{lem:dist-conc-bd}. We use same approach as in \cite[pp.~2055-2056]{tao_vu} (see also the proof of  \cite[Lemma 3.14]{bcc-generator}).

\begin{proof}[Proof of Theorem \ref{thm:intermed-sing}]
 To lighten the notation, let us denote by $s_1 \ge s_2\ge \cdots \ge s_n$ the singular values of $(A_n - \sqrt{np} w I_n)$. Fix $i$ such that $3\psi(n) \le i \le n-1$ and denote by $A_n^{m,w}$  the sub-matrix formed by first $m$ rows of the matrix $(A_n - \sqrt{np} w I_n)$, where $m= n- \lceil i/2\rceil$. Further denote by $s_1' \ge s_2'\ge \cdots \ge s_m'$  the singular values of $A_n^{m,w}$. Using Cauchy's interlacing inequality we see that
 \beq\label{eq:sing-val-deleted}
 s'_{n-i} \le s_{n-i}.
 \eeq
Next from \cite[Lemma A.4]{tao_vu} it follows that
\beq\label{eq:sing-val-and-dist}
s_1'^{-2} + s_2'^{-2} + \cdots + s_m'^{-2} = \dist_1'^{-2}+  \dist_2'^{-2} + \cdots +  \dist_m'^{-2},
\eeq
where $\dist_j':=\dist({\bm a}_j - w\sqrt{np}e_j, H_{j,n}^{m,w})$, ${\bm a}_{j}^{\sf T}$ is the $j$-th row of the matrix $A_n$, $H_{j,n}^{m,w}$ is the subspace spanned by all the rows of $A_n^{m,w}$ except the $j$-th row, and $e_j$ is the $j$-th canonical basis. We also note that $\dist_j \le \dist_j'$, where $\dist_j:= \dist({\bm a}_j, \text{span}(H_{j,n}^{m,w}, e_j))$. Thus from \eqref{eq:sing-val-deleted}-\eqref{eq:sing-val-and-dist} we deduce
\beq\label{eq:sing-val-and-dist-1}
\f{i}{2n} s_{n-i}^{-2} \le \f{1}{n}\sum_{j=n-i}^{m} s_j'^{-2}  \le\f{1}{n} \sum_{j=1}^{m} \dist_j^{-2}.
\eeq
It is easy to note that $\dim(\text{span}(H_{j,n}^{m,w}, e_j)) \le m+1 \le n- \psi(n)$ for all $j=1,2,\ldots, m$. Therefore from Lemma \ref{lem:dist-conc-bd} we further obtain
\[
\P\left( \dist_j \le c_{\ref{lem:dist-conc-bd}} \sqrt{p \cdot i/3}\right) \le 2n^{-4}, \qquad j=1,2,\ldots,m,
\]
where we used the fact that $n- \dim(\text{span}(H_{j,n}^{m,w}, e_j)) \ge n-(m+1) \ge i/3$ and chose $C_{\ref{thm:intermed-sing}} \ge 4 c_{\ref{lem:dist-conc-bd}}^{-1}$. Hence, from \eqref{eq:sing-val-and-dist-1} we see that
\[
\P\left(s_{n-i} \le \f{c_{\ref{lem:dist-conc-bd}}}{\sqrt{3}} \cdot \sqrt{np} \cdot \f{i}{n}\right) \le 2n^{-3},
\]
for all $i$ such that $3\psi(n) \le i \le n-1$. After taking the union over $i$, the proof of the theorem completes.
\end{proof}

\section{Weak Convergence}\label{sec:weak-conv}
Here our goal is to prove Theorem \ref{thm:weak-conv}. As mentioned in Section \ref{subsec:weak-conv}, using a truncation argument, we first show that it is enough to restrict to the case of bounded $\{\xi_{i,j}\}_{i,j=1}^n$. To this end, we have the following lemma.

\begin{lem}\label{lem:truncate_lsd}
If the conclusion of Theorem \ref{thm:weak-conv} holds for $\{\xi_{i,j}\}_{i,j=1}^n$ bounded then it continues to hold without the boundedness assumption.
\end{lem}

The proof of the truncation argument has now become standard in the random matrix literature which follows from an application of Hoffman-Wielandt inequality (see \cite[Lemma 2.1.19]{agz}) upon using the fact that the bounded Lipschitz metric on the space of all probability measures on $\R$ metrizes the weak convergence of probability measures (see \cite[Theorem C.8]{agz}). We refer the reader to \cite[Appendix C.2]{agz} for a definition of the bounded Lipschitz metric. Using the above two ingredients and proceeding similarly as in the the proof of \cite[Proposition 4.1]{bdj} one can complete the proof of Lemma \ref{lem:truncate_lsd}. Further details are omitted.


Equipped with Lemma \ref{lem:truncate_lsd}, hereafter we may and will assume that $\{\xi_{i,j}\}_{i,j=1}^n$ are bounded by some constant $\gK_0$. It is well known that the conclusion of  Theorem \ref{thm:weak-conv} holds for $\nu_{{G}_n}^w$, the symmetrized version of the empirical law of the singular values of $\left(n^{-1/2}{G_n}- wI_n\right)$, where $G_n$ is a real Ginibre matrix. Thus to prove Theorem \ref{thm:weak-conv} it is enough to show that
\beq\label{eq:weak-compare}
\int f(x) d\nu_{G_n}^w(x) - \int f(x) d\nu_n^w(x) \to 0, \quad \text{ as } n \to \infty, \quad \text{ almost surely},
\eeq
for every $f \in C_c(\R)$, the set of all continuous and compactly supported functions on $\R$, and $w \in\C$. To prove \eqref{eq:weak-compare} we first show that both the random probability measures $\nu_n^w$ and $\nu_{G_n}^w$ are close to their expectations, $\E \nu_n^w$ and $\E \nu_{G_n}^w$, respectively, and then we establish that $\E \nu_n^w$ and $\E \nu_{G_n}^w$ are close to each other as well. To carry out the first step we use the following concentration inequality.

\begin{lem}\label{lem:concentration-bd}
Let $H_n:=(h_{i,j})_{i,j=1}^n$ be a Hermitian random matrix with entries on and above the diagonal are jointly independent. Let $f : \R \mapsto \R$ be an $L$-Lipschitz function supported on a compact interval $I \subset \R$, with $\| f\|_\infty:=  \sup_{x \in I} |f(x)| \le 1$. Further let $H_0$ be an arbitrary $n \times n$  deterministic matrix and denote ${\sf H}_n:= H_n+H_0$. Fix an $\vep >0$.

\begin{enumeratei}
\item If the entries of $H_n$ are uniformly bounded by $\gK/\sqrt{n}$ for some $\gK < \infty$ then
\begin{align*}
\hspace{-\leftmargin}
\P\left(\left| \int f(x) dL_{{\sf H}_n}(x)- \E \int f(x) dL_{{\sf H}_n}(x)  \right| \ge \vep \right) \le \f{C_{\ref{lem:concentration-bd}}L |I|}{\vep} \exp \left( - \f{c_{\ref{lem:concentration-bd}} n^2 \vep^4}{\gK^2 L^4|I|^2}\right),
\end{align*}
for some absolute constants $c_{\ref{lem:concentration-bd}}, C_{\ref{lem:concentration-bd}} >0$.

\item If the entries of $H_n$ satisfy the logarithmic Sobolev inequality with uniform constant $\gL$ then
\begin{align*}
\hspace{-\leftmargin}
\P\left(\left| \int f(x) dL_{{\sf H}_n}(x)- \E \int f(x) dL_{{\sf H}_n}(x)  \right| \ge \vep \right) \le 2 \exp \left( - \f{c_{\ref{lem:concentration-bd}} n^2 \vep^2}{8 \gL L^2}\right).
\end{align*}
\end{enumeratei}

\end{lem}

\begin{proof}
Part (i) is a consequence of \cite[Lemma 3.2]{cook-a}. For $H_0=0$, the proof part (ii) is immediate from \cite[Theorem 1.1(b)]{GZ}. A key step in the proof \cite[Theorem 1.1(b)]{GZ} is to show that for any Lipschitz function $f : \R \mapsto \R$ the map $H \mapsto \int f(x) d_{L_H}(x)$ is also a Lipschitz function (see \cite[Lemma 1.2(b)]{GZ}). The same proof shows that for any deterministic $H_0$ the map $H \mapsto \int f(x) d_{L_{H+H_0}}(x)$ is again Lipschitz. Therefore the general case follows.
\end{proof}

Next we need to show that $\E \nu_n^w$ and $\E \nu_{G_n}^w$ are close to each other. This will be obtained upon showing that the corresponding Stieltjes transforms asymptotically are the same. Before stating the relevant result let us define the Stieltjes transform of a probability measure on $\R$.
\begin{dfn}
Let $\mu$ be a probability measure on $\R$. Its Stieltjes transform is given by
\beq
G_\mu(\zeta):=\int_\R \f{1}{x-\zeta} d\mu(x), \, \qquad \zeta\in \C\setminus \R. \notag
\eeq
\end{dfn}
We write $m_n(\zeta):=m_n(\zeta,w)$ and $m_{G_n}(\zeta):=m_{G_n}(\zeta, w)$ to denote the Stieltjes transform of $\nu_n^w$ and $\nu_{G_n}^w$, respectively. We now have the following lemma.
\begin{lem}\label{lem:stielt-compare}
Let ${A}_n$ be an $n \times n$ matrix with entries $a_{i,j}=\delta_{i,j} \cdot \xi_{i,j}$, where $\delta_{i,j}$ are i.i.d. Bernoulli random variables with $\P(\delta_{i,j}=1)=p$, and $\xi_{i,j}$ are centered i.i.d.~random variables with unit variance bounded by $\gK_0$ for some $\gK_0$. Fix any $w \in \C$ and $\zeta \in \C^+$. Then
\[
\left| \E m_n(\zeta) - \E m_{G_n}(\zeta)\right| \le \frac{C_{\ref{lem:stielt-compare}}}{\sqrt{np}(\Im \zeta)^4} \left(1+ \frac{1}{(n \Im \zeta)^2} \right),
\]
for some constant $C_{\ref{lem:stielt-compare}}$ depending only on $\gK_0$.
\end{lem}

A version of Lemma \ref{lem:stielt-compare} was proved in \cite{cook-rrd} when $A_n$ is replaced by $B_n$, a matrix of i.i.d.~centered $\dBer(p)$ random variables. A crucial step in the proof of \cite[Lemma 8.2]{cook-rrd} is the Lindeberg replacement strategy of \cite{cha05}. The latter depends only the bounds on the first three moments of the entries of $B_n$. Since $\{\xi_{i,j}\}_{i,j=1}^n$ are uniformly bounded by $\gK_0$ one can obtain the bounds necessary to apply \cite[Theorem 1.1, Corollary 1.2]{cha05}. Hence, following the same lines of arguments as in \cite[Lemma 8.2]{cook-rrd} one completes the proof of Lemma \ref{lem:stielt-compare}.

Equipped with all the required ingredients we now finish the proof of Theorem \ref{thm:weak-conv}.

\begin{proof}[Proof of Theorem \ref{thm:weak-conv}]
It is known that for any $w \in \C$ the random probability measure $\nu_{G_n}^w$ converges weakly to $\nu_\infty^w$, for some probability measure, almost surely. For example, see \cite{bai}. Therefore, it is enough show that
\beq\label{eq:as-diff}
\int f(x) d\nu_n^w(x) -\int f(x) d\nu_{G_n}^w(x) \to 0, \quad \text{almost surely,} \quad \text{ as } n \to \infty,
\eeq
for every bounded continuous function $f: \R \mapsto \R$. Since $\{\xi_{i,j}\}_{i,j=1}^n$ are uniformly bounded, by Chernoff's bound and Borel-Cantelli lemma it follows that the sequence of probability measures $\{\nu_n^w\}_{n \in \N}$ are almost surely compactly supported, whenever $np \ge C_0 \log n$ for some large constant $C_0$. The strong law of large number further shows that $\{\nu_{G_n}^w\}_{n \in \N}$ are almost surely compactly supported. Hence, we only need to show that \eqref{eq:as-diff} holds for continuous and compactly supported functions $f$.

Now, given any $\delta >0$, and a continuous and compactly supported function $f$, one can construct a Lipschitz function $f_\delta$, such that $\| f - f_\delta\|_\infty \le \delta$ with the Lipschitz constant of $f_\delta$ depending only on $f$ and $\delta$. Hence, it suffices to establish \eqref{eq:as-diff} only for compactly supported Lipschitz functions.

To this end,  recalling that Gaussian random variables satisfy the logarithmic Sobolev inequality (see \cite{gross}), from Lemma \ref{lem:concentration-bd} we deduce that
\beq\label{eq:as-diff-1}
\int f(x) d\gu_n(x) - \E \int f(x) d\gu_n(x) \to 0 \quad \text{ as } n \to \infty,
\eeq
almost surely, for $\gu_n=\nu_n^w$ and $\nu_{G_n}^w$. Since $np \to \infty$ as $n \to \infty$, from Lemma \ref{lem:stielt-compare} it also follows that (one can argue similarly as in the proof of \cite[Theorem 2.4.3]{agz})
\beq\label{eq:as-diff-2}
\E \int f(x) d\nu_n^w(x) - \E \int f(x) d\nu_{G_n}^w(x) \to 0, \quad \text{ as } n \to \infty.
\eeq
Combining \eqref{eq:as-diff-1}-\eqref{eq:as-diff-2}  we establish that \eqref{eq:as-diff} holds for any Lipschitz and compactly supported function $f$. This completes the proof.
\end{proof}


\section{Proof of Theorem \ref{thm:sparse_general}}\label{sec:proof-main-thm}
In this section we combine Theorem \ref{thm: smallest singular}, Theorem \ref{thm:intermed-sing}, and Theorem \ref{thm:weak-conv} to prove Theorem \ref{thm:sparse_general}. As already mentioned in Section \ref{sec:proof_outline}, to prove Theorem \ref{thm:sparse_general} we need to invoke the replacement principle. We fix $r \in (0,1)$ and define $\D_{r}:= \{ w \in B_\C(0,1-r): |\Im w| \ge r\}$. Then applying Lemma \ref{lem:replacement}, we show that for every $f\in C_c^2(\C)$ supported on $\D_{r}$, we have $\int f(w) dL_n(w) \to \int f(w) d \gm(w)$ in probability or almost surely, depending on the choice of the sparsity parameter $p$, where $L_n$ is the \abbr{ESD} of $\f{1}{\sqrt{np}}A_n$. Afterwards letting $r \to 0$ we establish the circular law limit. Below we make this idea precise.

Before we prove Theorem \ref{thm:sparse_general}, we need some properties of the probability measure $\nu_\infty^w$. Recall that $\nu_\infty^w$ is the limit of the \abbr{ESD} of ${\bm A}_n^w$ where ${\bm A}_n^w$ was defined in \eqref{eq:bmA_n}.

\begin{lem}\label{lem:prop-nu-inf}
\begin{enumerate}[(i)]
\item For any $w \in B_\C(0,1)$ the probability measure $\nu_\infty^w$ is supported on $[-\sqrt{\lambda_+},\sqrt{\lambda_+}]$, where
\[
\lambda_+:=\lambda_+(w):=\f{\left(\sqrt{1+8|w|^2}+3\right)^3}{8\left(\sqrt{1+8|w|^2}+1\right)}.
\]
\item {There exists some absolute constant $r_0 \in (0,1)$ such that for all $r \in (0,r_0)$, $\tau \in (0,1)$, and $w \in B_\C(0,1-r)$ we have}
\[
\int_{-\tau}^\tau |\log |x|| d\nu_\infty^w(x) \le C_{\ref{lem:prop-nu-inf}} \tau |\log \tau|,
\]
for some positive constant $C_{\ref{lem:prop-nu-inf}}$ which depends only on $r$.
\end{enumerate}
\end{lem}

\begin{proof}
In \cite[Lemma 4.2]{bai} it was shown that for any $w \in B_\C(0,1)$ the probability measure $\wt\nu^w_\infty$ is supported on $[0,\lambda_+(w)]$ where for any $t >0$, $\wt\nu_\infty^w\left((0,t^2)\right) = \nu_\infty^w\left((-t,t)\right)$. From this part (i) of the lemma follows.

Turning to prove (ii), using integration by parts we note that for any probability measure $\mu$ on $\R$ and $0<  \tau <1$,
\beq
\int_{0}^{\tau} |\log (x)| d\mu(x) \le |\log(\tau)| \mu((0,\tau))+ \int_{0}^{\tau} \f{\mu((0,t))}{t} dt.\label{eq:log_byparts}
\eeq
Using \cite[Lemma 7.8]{bcz} and \cite[Lemma 8.3]{bcz} we see that for any $t \in (0,1)$,
\[
\nu_\infty^w\left((0,t)\right) \le \nu_\infty^w\left((-t,t)\right) \le 2t \cdot (\Im m_\infty(\mathrm{i}t)) \le 2Ct,
\]
for some large constant $C$ depending on $r$. The rest follows from \eqref{eq:log_byparts}.
\end{proof}

We are now ready to prove Theorem \ref{thm:sparse_general}.

\begin{proof}[Proof of Theorem \ref{thm:sparse_general}]
We first prove part (i). That is we show that the \abbr{ESD} of $A_n/\sqrt{np}$ converges weakly to the circular law, in probability.
Fix $r \in (0,1/2)$ and denote $\D_{r}:= \{ w \in B_\C(0,1-r): |\Im w| \ge r\}$. Let us also fix $w \in \D_{r}$. Define
\[
\Omega_n':=\left\{s_{\min}\left(\f{A_n}{\sqrt{np}} - w I_n\right) \ge c_n\right\},
\]
 where
 \[
c_n:= c_{\ref{thm: smallest singular}} \exp \left(-C_{\ref{thm: smallest singular}} \frac{\log (1/p)}{\log (np)} \right) {\frac{1}{n\sqrt{np}}}.
\]
Setting $D_n:=- w \sqrt{np}I_n$ and applying Theorem \ref{thm: smallest singular} we deduce that
\beq\label{eq:smallest-sing-prob}
\P(\Omega_n') \ge 1 - \f{1+C'_{\ref{thm: smallest singular}}}{\sqrt{np}}.
\eeq
Fix any $\delta \in (0,1)$ and let $\tau:=\tau(\delta):= c_{\ref{thm:intermed-sing}} \delta$. Further denote $\psi(n):= \max\{\sqrt{{n}/{p}}, n/(\log n)^2\}$. Since $np =\omega(\log^2 n)$ we note that $\psi(n) =o(n/\log n)$. Equipped with these notations we recall the definition of $\nu^w_n$ to see that
\begin{align}\label{eq:log_int_split}
& \int_{-\tau}^\tau |\log (|x|)|d\nu_n^w(x) \notag\\
 = \, & \f{1}{n}\sum_{i=1}^{n-3\psi(n)} |\log (s_i)| \bI(s_i \le \tau) +\f{1}{n}\sum_{i=n-3\psi(n)+1}^n |\log (s_i)| \bI(s_i \le \tau).
\end{align}
We evaluate each term of the \abbr{RHS} of \eqref{eq:log_int_split} separately. Focusing on the second term we see that on the event $\Omega_n'$
\beq\label{eq:split-T2}
\f{1}{n}\sum_{i=n-3\psi(n)+1}^n |\log (s_i)| \bI(s_i \le \tau) \le |\log (s_n)| \cdot \f{3\psi(n)}{n} \le \log c_n^{-1} \cdot \f{3\psi(n)}{n} = o(1).
\eeq
We next consider the first term of \eqref{eq:log_int_split}. 
Denote the event described in Theorem \ref{thm:intermed-sing} by $\Omega_n''$.
Since $\min\{p \psi(n), \psi^2(n)/n\} \ge C_{\ref{thm:intermed-sing}} \log n$,  on the event $\Omega_n''$ we have
\begin{multline}\label{eq:split-T1-1}
 \f{1}{n}\sum_{i=1}^{n-3\psi(n)} |\log (s_i)| \bI(s_i \le \tau) =  \f{1}{n}\sum_{i=3\psi(n)}^{n-1} |\log (s_{n-i})| \bI(s_{n-i} \le \tau) \\
 \le \f{\log \left({1}/{c_{\ref{thm:intermed-sing}}}\right)}{n}\sum_{i=3\psi(n)}^{n-1} \bI(s_{n-i} \le \tau)  + \f{1}{n}\sum_{i=3\psi(n)}^{n-1} \log \left(\f{n}{i}\right) \bI(s_{n-i} \le \tau),
\end{multline}
and by Theorem \ref{thm:intermed-sing}, $\P((\Omega_n'')^c) \le 2/n^2$. Recalling the definition of $\tau$, from Theorem \ref{thm:intermed-sing} it also follows that
\[
s_{n-i} \le \tau \Rightarrow i \le  \delta n
\]
on the event $\Omega_n''$. So from \eqref{eq:split-T1-1} we deduce that
\begin{multline}\label{eq:split-T1-2}
\f{1}{n}\sum_{i=1}^{n-3\psi(n)} |\log (s_i)| \bI(s_i \le \tau)  \, \le \delta \cdot \log \left({1}/{c_{\ref{thm:intermed-sing}}}\right) + \f{1}{n}\sum_{i=3 \psi(n)}^{\delta n} \log \left( \f{n}{i}\right) \\
\le \delta \cdot \log \left({1}/{c_{\ref{thm:intermed-sing}}}\right)  - 2 \int_0^\delta \log x \, dx,
\end{multline}
for all large $n$. Hence, denoting $\Omega_n:= \Omega_n' \cup \Omega_n''$, from \eqref{eq:log_int_split}-\eqref{eq:split-T2} and \eqref{eq:split-T1-2} we obtain that
\beq
\int_{-\tau(\delta)}^{\tau(\delta)} |\log (|x|)|d\nu_n^w(x) \le \kappa(\delta), \notag
\eeq
for all large $n$, on the event $\Omega_n$, where $\kappa(\delta):=2 \delta \cdot \log \left({1}/{c_{\ref{thm:intermed-sing}}}\right)  - 2 \int_0^\delta \log x \, dx$. Note that  $\lim_{\delta \to 0} \kappa(\delta)=0$. Therefore given any $\kappa >0$ there exists $\tau_\kappa:=\tau(\kappa)$, with the property $\lim_{\kappa \to 0} \tau_k =0$, such that
\begin{multline}\label{eq:log_near0_prb_bd}
\limsup_{n \ra \infty} \P\left(\int_{-\tau_\kappa}^{\tau_\kappa} |\log |x|| d\nu_n^w(x) \ge \kappa\right) \\
\le \limsup_{n \ra \infty} \P\left(\left\{\int_{-\tau_\kappa}^{\tau_\kappa} |\log |x|| d\nu_n^w(x) \ge \kappa\right\} \cap \Omega_n \right)=0.
\end{multline}
Next noting that $\log(\cdot)$ is a bounded function on a compact interval that is bounded away from zero, we apply Theorem \ref{thm:weak-conv} to deduce that
\beq\label{eq:weak_conv}
\int_{(-R, - \tau_\kappa) \cup (\tau_\kappa, R)} |\log |x|| d\nu_n^w(x) \ra \int_{(-R, - \tau_\kappa) \cup (\tau_\kappa, R)} |\log |x| | d\nu_\infty^w(x)
\eeq
in probability, for any $R \ge 1$. Recall that for $w \in \D_{r}$ the support of $\nu_\infty^w$ is contained in $[-6,6]$ (see Lemma \ref{lem:prop-nu-inf}(i)).
  On the other hand, from \cite[Theorem 1.7]{BR} it follows that there exists a $K_0 < \infty$ such that
  \[
  \P(\|A_n\| \ge K_0 \sqrt{n p}) \le \exp(-c np),
  \]
  for some constant $c>0$. Therefore, using Borel-Cantelli lemma we deduce that 
  %
  %
%
  \beq\label{eq:log_near_infty}
\int_{(-R_0,R_0)^c} |\log |x|| d\nu_n^w(x) \to \int_{(-R_0,R_0)^c} |\log |x||d\nu_\infty^w(x) , \quad \text{ almost surely},
\eeq
for some $R_0 >0$. From Lemma \ref{lem:prop-nu-inf}(ii) we also have that
\beq\label{eq:log_near_0_limit}
\int_{-\tau_\kappa}^{\tau_\kappa} |\log |x|| d\nu_\infty^w(x) \le C\tau_\kappa |\log \tau_\kappa|,
\eeq
for some constant $C$.
As $\kappa >0$ is arbitrary and $\tau_\kappa \to 0$ as $\kappa \to 0$, combining \eqref{eq:log_near0_prb_bd}-\eqref{eq:log_near_0_limit} we deduce that
\beq\label{eq:log_integrates}
\f{1}{n}\log |\det(A_n/\sqrt{np} -w I_n)| = \int_{-\infty}^\infty \log|x| d\nu_n^w(x) \ra  \int_{-\infty}^\infty \log|x| d\nu_\infty^w(x),
\eeq
in probability. Now the rest of the proof is completed using Lemma \ref{lem:replacement}. Indeed, consider ${\mathfrak{G}}_n$ the $n \times n $ matrix with i.i.d.~centered Gaussian entries with variance one. It is well-known that, for Lebesgue almost all $w$,
\beq\label{eq:log_integrates_sub_gaussian}
\f{1}{n} \log |\det({\mathfrak{G}}_n/\sqrt{n} - w I_n)| \ra  \int_{-\infty}^\infty \log|x| d\nu_\infty^w(x), \text{ almost surely}.
\eeq
For example, one can obtain a proof of \eqref{eq:log_integrates_sub_gaussian} using \cite[Lemma 4.11, Lemma 4.12]{bordenave_chafai}, \cite[Theorem 3.4]{bourgade_yau_yin}, and \cite[Lemma 3.3]{R}.

Thus setting $\D=\D_{r}$, $B_n^{(1)}=A_n/\sqrt{np}$, and $B_n^{(2)}={\mathfrak{G}}_n/\sqrt{n}$ in Lemma \ref{lem:replacement}(a) we see that assumption (ii) there is satisfied. The assumption (i) of Lemma \ref{lem:replacement}(a) follows from weak laws of large numbers for triangular arrays.
Hence, using Lemma \ref{lem:replacement}(i) and the Circular law for i.i.d.~Gaussian matrix of unit variance (e.g.~\cite[Theorem 1.13]{tao_vu}), we obtain that for every $r >0$ and every $f_{r} \in C_c^2(\C)$, supported on $\D_{r}$,
\beq\label{eq:f_tau}
\int f_{r}(w) dL_n(w) \ra \f{1}{\pi}\int f_{r}(w) d\gm(w), \text{ in probability}.
\eeq
To finish the proof it now remains to show that one can extend the convergence of \eqref{eq:f_tau} to all $f \in C_c^2(\C)$. It follows from a standard argument.

Indeed, for any $r >0$ define $i_r \in C_c^2(\C)$ such that $i_r$ is supported on $\D_r$ and $i_r \equiv 1$ on $\D_{2r}$, and $i_r(\D_{r}\setminus \D_{2r}) \subset [0,1]$. Fixing any $\delta >0$ choose an $r>0$ such that $r \le \delta/64 \gK$ where $\gK:= \|f\|_\infty:=\sup_{w \in \C}|f(w)|$.
Denote $f_r:= f i_r$ and $\bar f_r:= f - f_r$. Applying \eqref{eq:f_tau} for the function $i_r$ and the triangle inequality we find that
\begin{multline}\label{eq:f_tau-1}
\P\left(\left|\int  \bar f_r(w) dL_n(w) \right| \ge \delta/4\right)    \le \P\left(\left|\int (1- i_r(w)) dL_n(w) \right| \ge \f{\delta}{4\gK}\right) \\ \le  \P\left(\left|\int  i_r(w) dL_n(w)  -\f{1}{\pi} \int i_r(w) d\gm(w)\right| \ge \f{\delta}{8\gK}\right) \to 0,
\end{multline}
as $n \to \infty$, where the last step follows from the fact that
\beq\label{eq:f_tau-2}
\left|\f1\pi \int_{B_\C(0,1)}(1-i_r(w)) d\gm(w)  \right| \le  \f1\pi \int_{B_\C(0,1)\setminus \D_{2r}} d\gm(w) \le \f{\delta}{8\gK},
\eeq
by our choice of $r$. Thus combining \eqref{eq:f_tau}-\eqref{eq:f_tau-2} and the triangle inequality we find that for any $f \in C_c^2(\C)$
\begin{multline*}
 \P\left( \left| f(w) dL_n(w) - \frac{1}{\pi}\int f(w) d\gm(w) \right|  \ge \delta \right) \\
\le  \P\left( \left| f_r(w) dL_n(w) - \frac{1}{\pi}\int f_r(w) d\gm(w) \right|  \ge \delta/2 \right) \\
 + \P\left(\left|\int  \bar f_r(w) dL_n(w) \right| \ge \delta/4\right) \to 0,
\end{multline*}
as $n \to \infty$.
This completes the proof of the first part of the theorem. To prove the second part of the theorem we note that under the assumption \eqref{p:assumption-as}, using Theorem \ref{thm: smallest singular} it follows
\[
\P(\tilde{\Omega}_n') \ge 1- O\left(\f{1}{n^2}\right),
\]
where
\[
\tilde{\Omega}_n':=\left\{s_{\min}\left(\f{A_n}{\sqrt{np}} - w I_n\right) \ge \tilde{c}_n\right\},
 \text{ and }
\tilde{c}_n:= c_{\ref{thm: smallest singular}} \exp \left(-C_{\ref{thm: smallest singular}} \frac{\log (1/p_n)}{\log (np_n)} \right)\cdot n^{-3}.
\]
Therefore, proceeding similarly as above, applying Borel-Cantelli lemma, and using Theorem \ref{thm:weak-conv}(ii) we see that the conclusions of \eqref{eq:log_near0_prb_bd}-\eqref{eq:weak_conv} hold almost surely. 

Thus under the assumption \eqref{p:assumption-as} we have shown that \eqref{eq:log_integrates} holds almost surely. Therefore proceeding same as above and using Lemma \ref{lem:replacement}(ii) we obtain that for every $r >0$ and every $f \in C_c^2(\C)$,
\beq\label{eq:f_tau-as}
\int f_{r}(w) dL_n(w) \ra \f{1}{\pi}\int f_{r}(w) d\gm(w), \text{ almost surely},
\eeq
where we recall that $f_r= f \cdot i_r$. The same proof shows that
\[
\limsup_{n \to \infty} \int (1-i_r(w)) dL_n(w) \le \frac{1}{\pi}\int_{B_\C(0,1)} (1-i_r(w)) d\gm(w) \le 8 r,
\]
almost surely. Therefore,
\[
\limsup_{n \ra \infty} \Big| \int f(w) dL_{n}(w) - \f{1}{\pi}\int_{B_\C(0,1)}f(w)d\gm(w) \Big| \le {16 \gK r}, \text{ almost surely},
\]
where we recall $\gK= \|f\|_\infty$. Since $r$ is arbitrary the result follows. 
\end{proof}


\subsection*{Acknowledgements}We thank the anonymous referee for her/his comments that lead to an improvement of the presentation and shortening of the paper. A.B.~thanks Amir Dembo and Ofer Zeitouni for suggesting the problem and helpful conversations.

\end{document}